\documentclass[11pt,leqno]{amsart}
\usepackage{amsmath,amssymb,mathrsfs}
\usepackage[pdftex]{graphicx}

\newtheorem{theorem}{Theorem}[section]
\newtheorem{lemma}[theorem]{Lemma}
\newtheorem{prop}[theorem]{Proposition}

\theoremstyle{definition}
\newtheorem{definition}[theorem]{Definition}
\newtheorem{example}[theorem]{Example}

\theoremstyle{remark}
\newtheorem{remark}[theorem]{Remark}

\numberwithin{equation}{section}
\DeclareMathOperator{\Hom}{\mathrm{Hom}}

\DeclareMathOperator{\rank}{\mathrm{rank}}
\DeclareMathOperator{\Vrt}{\mathrm{Vert}}
\DeclareMathOperator{\Conv}{\mathrm{Conv}}
\DeclareMathOperator{\sgn}{\mathrm{sgn}}

\let\Re\relax
\let\Im\relax
\DeclareMathOperator{\Re}{\mathrm{Re}}
\DeclareMathOperator{\Im}{\mathrm{Im}}
\usepackage{accents}
\makeatletter
\newcommand{\sbullet}{%
  \hbox{\fontfamily{lmr}\fontsize{.4\dimexpr(\f@size pt)}{0}\selectfont\textbullet}}
\DeclareRobustCommand{\const}{\accentset{\sbullet}}
\makeatother

\begin{document}
%\hfill\texttt{\jobname.tex}\qquad\today

\title{Lattice sums of hyperplane arrangements}
\author{Yasushi Komori, Kohji Matsumoto and Hirofumi Tsumura}


\maketitle

\begin{abstract}
We introduce certain lattice sums associated with hyperplane arrangements, which are (multiple)
sums running over integers, and can be regarded as 
generalizations of certain linear combinations of zeta-functions of root systems.
We also introduce generating functions of special values of those lattice sums, and study
their properties by virtue of the theory of convex polytopes.
Consequently we evaluate special values of those lattice sums, especially certain special values of
zeta-functions of root systems and their affine analogues.
In some special cases it is possible to treat sums running over positive integers, which may be
regarded as zeta-functions associated with hyperplane arrangements.
\end{abstract}

%%%%%%%%%%%%%%%%%%%%%%%%%%%%%%%%%%%%%%%%%%%%%%%%%%%%%%%%%%%%%%%%%%%%%%%%%%%%%%%%%%%%%%%%%%%%%%%%%%%%
\section{Introduction}\label{sec-1}
%%%%%%%%%%%%%%%%%%%%%%%%%%%%%%%%%%%%%%%%%%%%%%%%%%%%%%%%%%%%%%%%%%%%%%%%%%%%%%%%%%%%%%%%%%%%%%%%%%%%%

The notion of Witten zeta-functions associated with semisimple Lie algebras was introduced by
Zagier \cite{Zag94}, inspired by the work of Witten \cite{Wit91} in quantum gauge theory.
Recently the authors have developed the theory of zeta-functions of root systems
(e.g.\ \cite{KMKyushu,KM2,KM5,KM3}),  which are 
multi-variable generalizations of Witten
zeta-functions.    In particular, the ``Weyl group symmetric'' linear combinations of
zeta-functions of root systems $S(\mathbf{s}, \mathbf{y};\Delta)$ (where $\mathbf{s}$ is a
complex multi-variable, $\mathbf{y}$ is a certain vector and $\Delta$ is a finite reduced root 
system) and the generating functions of special values of those linear combinations
were introduced and studied in \cite{KMKyushu,KM5,KM3}.

In the present paper, we will introduce certain lattice sums
of hyperplane arrangements, which are generalizations of the above linear combinations of
zeta-functions of root systems.    We will also introduce the generating functions of special values
of those lattice sums.
It is to be stressed that those generating functions can describe not only values but also functional relations among zeta-functions of root systems.  Furthermore if they are combined with Poincar\'e polynomials of Weyl groups, we obtain explicit formulas for special odd values of zeta-functions of
root systems.  These results will be treated in the forthcoming paper \cite{KMTforth}.

Another application is to calculate special values of affine analogue of zeta-functions of
root systems.   Although in the cases of affine root systems it is natural to work with the character formulas instead of the dimension formulas, a straightforward generalization is also interesting. We will present some examples in Section \ref{sec-3}.
%Examples \ref{Exam-A1} and \ref{Exam-A2}.

In the present paper, our consideration is not restricted to the case in the domain of
absolute convergence; we will study the values of lattice sums outside the domain of absolute
convergence.    Here we explain this point by simple examples.

Let $\mathbb{N}$ be the set of positive integers, $\mathbb{N}_0=\mathbb{N}\cup\{0\}$, 
$\mathbb{Z}$ the ring of rational integers, $\mathbb{R}$ the field of real numbers, and
$\mathbb{C}$ the field of complex numbers.
For any set $S$, the symbol $\sharp S$ denotes the cardinality of $S$.

Let $k\in\mathbb{N}$, and let $y\in\mathbb{R}$
with $y\notin\mathbb{Z}$ if $k=1$.
It is well-known (cf.\ \cite[Theorem 12.19]{Apos})
that %
\begin{equation}
\label{eq:1d}
-\frac{(2\pi\sqrt{-1})^k}{k!}B_k(\{y\})=\lim_{N\to\infty}\sum_{\substack{|m|\leq N\\m\neq 0}}\frac{e^{2\pi\sqrt{-1}my}}{m^k},
\end{equation}
where $\{y\}=y-[y]$ is the fractional part of $y$ and $B_k(\cdot)$ is the $k$-th Bernoulli polynomial
defined by
\begin{equation}
  \frac{t e^{t\{y\}}}{e^{t}-1}=\sum_{k=0}^\infty B_k(\{y\})\frac{t^k}{k!}.
\end{equation}
In the case $k=0$, then \eqref{eq:1d} does not hold straightforward.
However this formula still holds in some sense via the following regularization.
We see  that the right-hand side of \eqref{eq:1d} is analytically continued to the whole space $\mathbb{C}$ in the variable $k$ and then
it is evaluated as $-1$ at $k=0$. This effect is (formally) realized in the series by replacing the condition $m\neq 0$ by $m=0$ and 
$\sum_{m\neq 0}$ by $-\sum_{m=0}$ with $0^0=1$.
% However by some modification we see that this formula still holds in some sense,
% that is, 
%we replace the condition $m\neq 0$ by $m=0$ and $m^{-k}$ by $-1$.
Hence the sum consists of only one term. As a result,
we may understand the case $k=0$ as
\begin{equation}
%\label{eq:1d}
-\frac{(2\pi\sqrt{-1})^k}{k!}B_k(\{y\})=
-\lim_{N\to\infty}\sum_{\substack{|m|\leq N\\m=0}}e^{2\pi\sqrt{-1}my}
=-e^{2\pi\sqrt{-1}my}|_{m=0}=-1,
\end{equation}
where $B_0(\{y\})=1$.

This interpretation works well in the multi-dimensional cases.
%We consider special values of a family of series
%and show that all these values are explicitly given as coefficients of a certain generating function.
For example, let $\alpha,\beta,\gamma\in\mathbb{C}$ and $k_1,k_2,k_3\in\mathbb{N}_0$,
and consider the sum
\begin{equation}
\label{eq:introS1}
%   S(k_1,k_2,k_3)=\lim_{N\to\infty}
%   \sum_{\substack{m,n\in\mathbb{Z}\\m+\alpha,n+\beta,m+n+\gamma\neq 0\\|m|,|n|\leq N}}
% \frac{1}{(m+\alpha)^{k_1}(n+\beta)^{k_2}(m+n+\gamma)^{k_3}}.
  S((k_1,k_2,k_3),(y_1,y_2))=\lim_{N\to\infty}
  \sum_{\substack{m,n\in\mathbb{Z}\\m+\alpha,n+\beta,m+n+\gamma\neq 0\\|m|,|n|\leq N}}
%  \sideset{}{'}\sum_{\substack{m,n\in\mathbb{Z}\\|m|,|n|\leq N}}
\frac{e^{2\pi\sqrt{-1}(my_1+ny_2)}}{(m+\alpha)^{k_1}(n+\beta)^{k_2}(m+n+\gamma)^{k_3}}.
\end{equation}
This is convergent if $k_1,k_2,k_3\geq 1$.
If some of $k_i$'s are $0$, then we modify the series. In the case when only $k_1=0$, we replace the condition $m+\alpha\neq 0$ by $m+\alpha=0$ in the sum with the minus sign and $0^0=1$, that is,
%the factor $(m+\alpha)^{k_1}$ by $-1$, that is,
\begin{equation}
\label{eq:introS2}
  S((0,k_2,k_3),(y_1,y_2))
  =-\lim_{N\to\infty}
  \sum_{\substack{m,n\in\mathbb{Z}\\|m|,|n|\leq N\\n+\beta,m+n+\gamma\neq 0\\m+\alpha=0}}
  \frac{e^{2\pi\sqrt{-1}(my_1+ny_2)}}{(n+\beta)^{k_2}(m+n+\gamma)^{k_3}}.
\end{equation}
By the restriction $m+\alpha=0$, this sum is $0$ if $\alpha\notin\mathbb{Z}$. If $\alpha\in\mathbb{Z}$, then the sum reduces to the one-dimensional sum
\begin{equation}
  S((0,k_2,k_3),(y_1,y_2))
  =-\lim_{N\to\infty}
  \sum_{\substack{n\in\mathbb{Z}\\|n|\leq N\\n+\beta,n+\gamma-\alpha\neq 0}}
  \frac{e^{2\pi\sqrt{-1}(-\alpha y_1+ny_2)}}{(n+\beta)^{k_2}(n+\gamma-\alpha)^{k_3}}.
\end{equation}
In the other cases, the sum is similarly modified.
Then the special values $S((k_1,k_2,k_3),(y_1,y_2))$ 
for all $k_1,k_2,k_3\in\mathbb{N}_0$
are
explicitly given by coefficients of a generating function, which will be given in Example \ref{ex:intro}.

In the above arguments the sums are taken over all integers.
However in some special cases, it is possible to treat sums running over only positive integers
(Examples \ref{Exam-A1}, \ref{Exam-A2}), which may be regarded as zeta-functions associated with
hyperplane arrangements.

In the next section we will introduce more general lattice sums, and their generating functions.

% where $\sideset{}{'}\sum$ means the sum with the following restrictions
% \begin{equation}
%   \begin{cases}
%     m\neq 0\qquad (k_1\neq 0)\\
%   \end{cases}
% \end{equation}

%%%%%%%%%%%%%%%%%%%%%%%%%%%%%%%%%%%%%%%%%%%%%%%%%%%%%%%%%%%%%%%%%%%%%%%%%%%%%%%%%%%%%%%%%%%%%%%%%%%
\section{Notations and statement of main results}\label{sec-2}
%%%%%%%%%%%%%%%%%%%%%%%%%%%%%%%%%%%%%%%%%%%%%%%%%%%%%%%%%%%%%%%%%%%%%%%%%%%%%%%%%%%%%%%%%%%%%%%%%%%
We fix a positive integer $r$.
Let $V=\mathbb{R}^r$ be a real vector space equipped with the standard inner product $\langle\cdot,\cdot\rangle$.
We regard $f=(\vec{f},\const{f})\in V\times\mathbb{C}$ with
$\vec{f}\in V$ and $\const{f}\in\mathbb{C}$ as an affine
linear functional on $V$ by
$f(\mathbf{v})=\langle \vec{f},\mathbf{v}\rangle+\const{f}$ for
$\mathbf{v}\in V$.  

We use the following notation:
For $X\subset V$, put $\langle X\rangle=\sum_{\mathbf{v}\in X}\mathbb{Z}\mathbf{v}$.
For $Y\subset V\times\mathbb{C}$,
put $\vec{Y}=\{\vec{f}~|~f=(\vec{f},\const{f})\in Y\}$.
Let $\Lambda\subset(\mathbb{Z}^{r}\setminus\{\vec{0}\})\times\mathbb{C}$ 
with $\sharp \Lambda<\infty$ such that $\rank\langle\vec{\Lambda}\rangle=r$. 
Put
$\widetilde{\Lambda}=\{f\in\Lambda~|~\rank\langle\vec{\Lambda}\setminus\{\vec{f}\}\rangle\neq r\}$.
For each $f\in\Lambda$ we associate a number $k_f\in\mathbb{N}_0$, and put
$\mathbf{k}=(k_f)_{f\in\Lambda}\in\mathbb{N}_{0}^{\sharp \Lambda}$.
For $k\in\mathbb{N}_0$, define
\begin{align*}
  \Lambda_k&=\Lambda_k(\mathbf{k})=\{f\in\Lambda~|~k_f=k\},\\
  \Lambda_+&=\Lambda_+(\mathbf{k})=\{f\in\Lambda~|~k_f>0\}.
\end{align*}
Obviously
\begin{align*}
\Lambda_+=\bigcup_{k\geq 1}\Lambda_k\quad{\rm and}\quad\Lambda=\Lambda_+\cup\Lambda_0.
\end{align*}
For $H\subset\Lambda$ such that $\rank\langle\vec{H}\rangle=r-1$,
let
$\mathfrak{H}_{H}=\sum_{g\in H}\mathbb{R}\,\vec{g}$ 
be the hyperplane
 passing through $\vec{H}\cup\{\vec{0}\}$. 

The following is the main object in the present paper, a lattice sum
over the hyperplane arrangement given by linear functionals belonging
to $\Lambda$.

\begin{definition}
\label{def:S}
For $\mathbf{k}=(k_f)_{f\in\Lambda}\in\mathbb{N}_{0}^{\sharp \Lambda}$ and
$\mathbf{y}\in V\setminus\bigcup_{f\in \widetilde{\Lambda}\cap\Lambda_1}(\mathfrak{H}_{\Lambda\setminus\{f\}}+\mathbb{Z}^r)$,
we define
\begin{equation}
\label{eq:S}                                                                            
  S(\mathbf{k},\mathbf{y};\Lambda)=                                                     
  \lim_{N\to\infty}                                                                     
  Z(N;\mathbf{k},\mathbf{y};\Lambda),
\end{equation}
where
\begin{gather}
  Z(N;\mathbf{k},\mathbf{y};\Lambda)=
  (-1)^{\sharp\Lambda_0}
  \sum_{\substack{\mathbf{v}=(v_1,\ldots,v_r)\in\mathbb{Z}^{r}\\
      |v_j|\leq N\quad(1\leq j\leq r)\\
      f(\mathbf{v})\neq 0\quad(f\in\Lambda_+)\\
      f(\mathbf{v})=0\quad(f\in\Lambda_0)}}
  e^{2\pi\sqrt{-1}\langle \mathbf{y},\mathbf{v}\rangle}
  \prod_{f\in\Lambda_+}
  \frac{1}{f(\mathbf{v})^{k_f}}%\\
%  \label{eq:S}
%  S(\mathbf{k},\mathbf{y};\Lambda)=
%  \lim_{N\to\infty}
%  Z(N).
\end{gather}
% \begin{equation}
%   \label{eq:cf}
%   S(\mathbf{k},\mathbf{y};\Lambda)=
%   \lim_{N\to\infty}
%   (-1)^{\sharp\Lambda_0}
%   \sum_{\substack{\mathbf{v}=(v_1,\ldots,v_r)\in\mathbb{Z}^{r}\\
%       |v_j|\leq N\quad(1\leq j\leq r)\\
%       f(\mathbf{v})\neq 0\quad(f\in\Lambda_+)\\
%       f(\mathbf{v})=0\quad(f\in\Lambda_0)}}
%   e^{2\pi\sqrt{-1}\langle \mathbf{y},\mathbf{v}\rangle}
%   \prod_{f\in\Lambda_+}
%   \frac{1}{f(\mathbf{v})^{k_f}}.
% \end{equation}
% where 
% \begin{equation}
% \label{eq:convcond}
%   \mathbf{y}\notin\bigcup_{f\in \widetilde{\Lambda}\cap\Lambda_1}(\mathfrak{H}_{\Lambda\setminus\{f\}}+\mathbb{Z}^r).
% \end{equation}
for $N>0$.
\end{definition}

This $S(\mathbf{k},\mathbf{y};\Lambda)$ is a generalization of the notion of "Weyl group symmetric"
linear combinations of zeta-functions of root systems $S(\mathbf{s},\mathbf{y};\Delta)$ mentioned in
the Introduction (in the case $\mathbf{s}=\mathbf{k}$); cf.\ \cite[(3.3)]{KM5}, \cite[(110)]{KM3}.
The first main result in the present paper is as follows.

\begin{theorem}
\label{thm:main0}
The series $S(\mathbf{k},\mathbf{y};\Lambda)$ 
converges and
is continuous in $\mathbf{y}$
on $V\setminus\bigcup_{f\in \widetilde{\Lambda}\cap\Lambda_1}(\mathfrak{H}_{\Lambda\setminus\{f\}}+\mathbb{Z}^r)$.
\end{theorem}

In order to define the generating function of $S(\mathbf{k},\mathbf{y};\Lambda)$, we need some more
notations.
Let $\mathscr{B}=\mathscr{B}(\Lambda)$ be the set of all subsets
$B=\{f_1,\ldots,f_r\}\subset \Lambda$ 
such that $\vec{B}$ forms a basis of $V$.
For $B\in\mathscr{B}$, 
let $\vec{B}^{*}=\{\vec{f}_1^B,\ldots,\vec{f}_r^B\}$ be
the dual basis of $\vec{B}=\{\vec{f}_1,\ldots,\vec{f}_r\}$ in $V$. It should be noted that for each $B\in\mathscr{B}$, we have
\begin{equation}
\label{eq:LsB}
\widetilde{\Lambda}\subset B,
\end{equation}
because all elements of $\widetilde{\Lambda}$ are indispensable for
constructing a basis.

Next we define a multi-dimensional generalization of fractional part $\{\cdot\}$ for real numbers,
which was first introduced in \cite[Section 4]{KM5}.
Let $\mathscr{R}=\mathscr{R}(\Lambda)$ be the set of all subsets
$R=\{g_1,\ldots,g_{r-1}\}\subset \Lambda$ 
such that $\vec{R}=\{\vec{g}_1,\ldots,\vec{g}_{r-1}\}$ is linearly independent set.
We need to fix a vector 
\begin{equation}
  \phi\in V\setminus\bigcup_{R\in\mathscr{R}}\mathfrak{H}_R
\end{equation}
so that $\langle\phi,\vec{f}^B\rangle\neq 0$ for all
$B\in\mathscr{B}$ and $f\in B$
(because if $\langle\phi,\vec{f}^B\rangle=0$ for some 
% $B=\{f_1,\ldots,f_r\}$, then $\phi\in\mathfrak{H}_R$ with
% $R=\{f_1,\ldots,f_{r-1}\}\in\mathscr{R}$).
$B\in\mathscr{B}$ and $f\in B$,  then
$\phi\in\mathfrak{H}_R$ with
$R=B\setminus\{f\}\in\mathscr{R}$).

For
$\mathbf{y}\in V$, $B\in\mathscr{B}$ and
$f\in B$, we define the multi-dimensional fractional part by
\begin{equation}
  \{\mathbf{y}\}_{B,f}=
  \begin{cases}
    \{\langle\mathbf{y},\vec{f}^B\rangle\}\qquad&(\langle\phi,\vec{f}^B\rangle>0),\\
    1-\{-\langle\mathbf{y},\vec{f}^B\rangle\}\qquad&(\langle\phi,\vec{f}^B\rangle<0).
  \end{cases}
\end{equation}
It should be noted that
\begin{equation}
\label{2-6}
  \{a\}=1-\{-a\}
\end{equation}
for $a\in\mathbb{R}\setminus\mathbb{Z}$.

Now we define
the generating function of 
$S(\mathbf{k},\mathbf{y};\Lambda)$ % at non-negative integers
% the generating function of special values of
% $S(\mathbf{s},\mathbf{y};\Lambda)$ at non-negative integers
and state its properties.

\begin{definition}
\label{def:gen}
  For $\mathbf{y}\in V$ and $\mathbf{t}=(t_f)_{f\in\Lambda}\in\mathbb{C}^{\sharp\Lambda}$,
  we define
  \begin{equation}
    \label{eq:exp_F}
    \begin{split}
      F(\mathbf{t},\mathbf{y};\Lambda)&=
      \sum_{B\in\mathscr{B}(\Lambda)}
      \Bigl(
      \prod_{g\in \Lambda\setminus B}
      \frac{t_g}
      {t_g-2\pi\sqrt{-1}\const{g}-\sum_{f\in B}(t_f-2\pi\sqrt{-1}\const{f})\langle \vec{g},\vec{f}^B\rangle}
      \Bigr)
      \\
      &\qquad\times
      \frac{1}{\sharp(\mathbb{Z}^r/\langle \vec{B}\rangle)}
      \sum_{\mathbf{w}\in \mathbb{Z}^r/\langle \vec{B}\rangle}
      \Bigl(
      \prod_{f\in B}\frac{t_f\exp
        ((t_f-2\pi\sqrt{-1}\const{f})\{\mathbf{y}+\mathbf{w}\}_{B,f})}{\exp(t_f-2\pi\sqrt{-1}\const{f})-1}
      \Bigr).
    \end{split}
  \end{equation}
\end{definition}

%\begin{theorem}
%\label{thm:main0}
%The series $S(\mathbf{k},\mathbf{y};\Lambda)$ 
%converges and
%is continuous in $\mathbf{y}$
%on $V\setminus\bigcup_{f\in \widetilde{\Lambda}\cap\Lambda_1}(\mathfrak{H}_{\Lambda\setminus\{f\}}+\mathbb{Z}^r)$.
%\end{theorem}

\begin{theorem}
\label{thm:main1}
{\rm (i)}
The function
$F(\mathbf{t},\mathbf{y};\Lambda)$
has one-sided continuity in $\mathbf{y}\in V$ in the direction $\phi$,
that is
\begin{equation}
  \lim_{c\to 0+}F(\mathbf{t},\mathbf{y}+c\phi;\Lambda)
=F(\mathbf{t},\mathbf{y};\Lambda).
\end{equation}

{\rm (ii)}
$F(\mathbf{t},\mathbf{y};\Lambda)$ is continuous in $\mathbf{y}$ on
$V\setminus\bigcup_{f\in \widetilde{\Lambda}}(\mathfrak{H}_{\Lambda\setminus\{f\}}+\mathbb{Z}^r)$.
In particular 
if $\widetilde{\Lambda}$ is empty, then
$F(\mathbf{t},\mathbf{y};\Lambda)$ is continuous on the whole $V$
and
is independent of the choice of $\phi$.

{\rm (iii)}
$F(\mathbf{t},\mathbf{y};\Lambda)$ is holomorphic in the neighborhood of the origin in $\mathbf{t}$.
\end{theorem}

Write the Taylor expansion of $F(\mathbf{t},\mathbf{y};\Lambda)$
around the origin in $\mathbf{t}$ as
\begin{equation}
\label{F_Taylor_exp}
  F(\mathbf{t},\mathbf{y};\Lambda)=\sum_{\mathbf{k}\in\mathbb{N}_0^{\sharp \Lambda}}C(\mathbf{k},\mathbf{y};\Lambda)\prod_{f\in\Lambda}\frac{t_f^{k_f}}{k_f!}.
\end{equation}

\begin{theorem}
\label{thm:main1b}
We have
\begin{equation}
\label{eq:SeqC}
  S(\mathbf{k},\mathbf{y};\Lambda)
  =
%  \frac{1}{|\mathbb{Z}^r/\langle\vec{B}\rangle|}
  \Bigl(
  \prod_{f\in\Lambda}
  -\frac{(2\pi\sqrt{-1})^{k_f}}{k_f!}
  \Bigr)
  C(\mathbf{k},\mathbf{y};\Lambda)
\end{equation}
for $\mathbf{k}=(k_f)_{f\in \Lambda}\in\mathbb{N}_0^{\sharp \Lambda}$
and
$
\mathbf{y}\in V\setminus\bigcup_{f\in \widetilde{\Lambda}\cap\Lambda_1}(\mathfrak{H}_{\Lambda\setminus\{f\}}+\mathbb{Z}^r)$.
  % \mathbf{y}\notin\bigcup_{f\in \Lambda\quad(f\in B_0\cap\Lambda_1)}(\mathfrak{H}_{B_0\setminus\{f\}}+\mathbb{Z}^r).
%Moreover 
%$F(\mathbf{t},\mathbf{y};\Lambda)$ is continuous in $\mathbf{y}$ on
%$V\setminus\bigcup_{f\in \widetilde{\Lambda}}(\mathfrak{H}_{\Lambda\setminus\{f\}}+\mathbb{Z}^r)$.
%In particular 
%if $\widetilde{\Lambda}$ is empty, then
%$F(\mathbf{t},\mathbf{y};\Lambda)$ is continuous on the whole $V$
%and
%is independent of the choice of $\phi$.
% if 
% $\rank\langle\vec{\Lambda}\setminus\{\vec{f}\}\rangle=r$
% for all $f\in\Lambda$, then
% $F(\mathbf{t},\mathbf{y};\Lambda)$ is continuous in $\mathbf{y}$ and
% is independent of the choice of $\phi$.
\end{theorem}

The above results are again generalizations of the results proved in \cite{KM5}, \cite{KM3}.
In fact, the form of $F(\mathbf{t},\mathbf{y};\Lambda)$ in Definition \ref{def:gen} is the
generalization of \cite[Theorem 4.1]{KM5}, Theorem \ref{thm:main1} is the generalization of
the facts mentioned in \cite[p.252]{KM3}, and Theorem \ref{thm:main1b} is the generalization of
\cite[(3.10)]{KM5}.

Before going into the proofs of the main theorems, in the next section we will give several
examples.    Then we will start the proofs of main theorems from Section \ref{sec-4}.
Sections \ref{sec-4} is %and \ref{sec-5} are 
devoted to the proof of Theorem \ref{thm:main0}.
Then from Section \ref{sec-6} to Section \ref{sec-9} we will describe the proof of 
Theorem \ref{thm:main1} and Theorem \ref{thm:main1b}.    In the final section we will mention that
there is some hierarchy among generating functions.

%%%%%%%%%%%%%%%%%%%%%%%%%%%%%%%%%%%%%%%%%%%%%%%%%%%%%%%%%%%%%%%%%%%%%%%%%%%%%%%%%%%%%%%%%%%%%%%%%%
\section{Examples}\label{sec-3}
%%%%%%%%%%%%%%%%%%%%%%%%%%%%%%%%%%%%%%%%%%%%%%%%%%%%%%%%%%%%%%%%%%%%%%%%%%%%%%%%%%%%%%%%%%%%%%%%%%%

In this section we apply our theorems to some special cases, and to state explicit expressions of
$F(\mathbf{t},\mathbf{y};\Lambda)$, $C(\mathbf{k},\mathbf{y};\Lambda)$ and
$S(\mathbf{k},\mathbf{y};\Lambda)$ for those examples.

\begin{example}
\label{ex:intro} 
Let $V=\mathbb{R}^2$.
Let $\alpha,\beta,\gamma\in\mathbb{C}$,
\begin{align}
  \Lambda&=\{f_1=((1,0),\alpha),f_2=((0,1),\beta),f_3=((1,1),\gamma)\},\\
  \mathscr{B}&=\{\{f_1,f_2\},\{f_1,f_3\},\{f_2,f_3\}\},
\end{align}
%where $\alpha\in\mathbb{C}\setminus\{0\}$,
which corresponds to the series in \eqref{eq:introS1}, \eqref{eq:introS2} and so on.
% \begin{equation}
%   S(k_a,k_b,k_c,y_1,y_2;\Lambda)=\lim_{N\to\infty}
% \sum_{\substack{m,n\in\mathbb{Z}\\|m|,|n|\leq N\\a_1m+a_2n+\alpha\neq0\\b_1m+b_2n+\beta\neq0\\c_1m+c_2n+\gamma\neq0}}\frac{e^{2\pi\sqrt{-1}(my_1+ny_2)}}{(a_1m+a_2n+\alpha)^{k_a}(b_1m+b_2n+\beta)^{k_b}(c_1m+c_2n+\gamma)^{k_c}},
% \end{equation}
% where $k_a,k_b,k_c\in\mathbb{N}$. 
Then the generating function is given by
\begin{equation}
  \begin{split}
  &F((t_1,t_2,t_3),(y_1,y_2);\Lambda)=
  \\
&
\frac{t_3}{t_3-2\pi\sqrt{-1}\gamma-(t_1-2\pi\sqrt{-1}\alpha)-(t_2-2\pi\sqrt{-1}\beta)}
\frac{t_1e^{(t_1-2 \pi \sqrt{-1} \alpha)\{y_1\}}}{e^{(t_1-2 \pi \sqrt{-1} \alpha)}-1}
\frac{t_2e^{(t_2-2 \pi \sqrt{-1} \beta)\{y_2\}}}{e^{(t_2-2 \pi \sqrt{-1} \beta)}-1}
\\
&+
\frac{t_2}{t_2-2\pi\sqrt{-1}\beta+(t_1-2\pi\sqrt{-1}\alpha)-(t_3-2\pi\sqrt{-1}\gamma)}
\frac{t_1e^{(t_1-2 \pi \sqrt{-1} \alpha)\{y_1-y_2\}}}{e^{(t_1-2 \pi \sqrt{-1} \alpha)}-1}
\frac{t_3e^{(t_3-2 \pi \sqrt{-1} \gamma)\{y_2\}}}{e^{(t_3-2 \pi \sqrt{-1} \gamma)}-1}
\\
&+
\frac{t_1}{t_1-2\pi\sqrt{-1}\alpha+(t_2-2\pi\sqrt{-1}\beta)-(t_3-2\pi\sqrt{-1}\gamma)}
\frac{t_2e^{(t_2-2 \pi \sqrt{-1} \beta)(1-\{y_1-y_2\})}}{e^{(t_2-2 \pi \sqrt{-1} \beta)}-1}
\frac{t_3e^{(t_3-2 \pi \sqrt{-1} \gamma)\{y_1\}}}{e^{(t_3-2 \pi \sqrt{-1} \gamma)}-1}.
\end{split}
\end{equation}
In particular, if $\alpha,\beta,\gamma\notin\mathbb{Z}$ with $\alpha+\beta\neq\gamma$,
we have
\begin{equation}
  \begin{split}
    C((2,1,1),(y_1,y_2);\Lambda)=
% \\
&
-\frac{16\sqrt{-1}\pi ^3 \{y_1-y_2\} e^{2\sqrt{-1}\pi(\alpha -\alpha  \{y_1-y_2\}-\gamma  \{y_2\})}}{\left(-1+e^{2\sqrt{-1}\pi  \alpha
   }\right)^2 \left(-1+e^{-2\sqrt{-1}\pi  \gamma }\right) (\alpha +\beta -\gamma )}
\\
&
+\frac{16\sqrt{-1}\pi ^3
   \{y_1-y_2\} e^{2\sqrt{-1}\pi  (2\alpha-\alpha  \{y_1-y_2\}-\gamma  \{y_2\})}}{\left(-1+e^{2\sqrt{-1}\pi  \alpha
   }\right)^2 \left(-1+e^{-2\sqrt{-1}\pi  \gamma }\right) (\alpha +\beta -\gamma )}
\\
&
+\frac{16\sqrt{-1}\pi ^3 e^{2\sqrt{-1}\pi(\alpha -\alpha  \{y_1-y_2\}-\gamma \{y_2\})}}{\left(-1+e^{2\sqrt{-1}\pi  \alpha }\right)^2 \left(-1+e^{-2\sqrt{-1}\pi  \gamma }\right) (\alpha
   +\beta -\gamma )}
\\
&
-\frac{8 \pi ^2 e^{2\sqrt{-1}\pi(-\beta +\beta  \{y_1-y_2\}-\gamma  \{y_1\})}}{\left(-1+e^{-2\sqrt{-1}\pi  \beta }\right) \left(-1+e^{-2 \sqrt{-1}\pi  \gamma }\right) (\alpha +\beta -\gamma )^2}
\\
&
+\frac{16\sqrt{-1}\pi ^3 \{y_1\} e^{2\sqrt{-1}\pi(\alpha -\alpha  \{y_1\}-\beta  \{y_2\})}}{\left(-1+e^{2\sqrt{-1}\pi  \alpha }\right)^2 \left(-1+e^{-2\sqrt{-1}\pi  \beta }\right) (\alpha +\beta
   -\gamma )}
\\
&
-\frac{16\sqrt{-1}\pi ^3 \{y_1\} e^{2\sqrt{-1}\pi (2\alpha -\alpha  \{y_1\}-\beta  \{y_2\})}}{\left(-1+e^{2\sqrt{-1}\pi  \alpha }\right)^2
   \left(-1+e^{-2\sqrt{-1}\pi  \beta }\right) (\alpha +\beta -\gamma )}
\\
&
-\frac{16\sqrt{-1}\pi ^3 e^{2\sqrt{-1}\pi(\alpha -\alpha  \{y_1\}-\beta  \{y_2\})}}{\left(-1+e^{2
  \sqrt{-1}\pi  \alpha }\right)^2 \left(-1+e^{-2\sqrt{-1}\pi  \beta }\right) (\alpha +\beta -\gamma )}
\\
&
+\frac{8 \pi ^2 e^{-2\sqrt{-1}\pi(\alpha  \{y_1\}+\beta  \{y_2\})}}{\left(-1+e^{-2\sqrt{-1}\pi  \alpha }\right) \left(-1+e^{-2\sqrt{-1}\pi  \beta }\right) (\alpha +\beta
   -\gamma )^2}
\\
&
-\frac{8 \pi ^2 e^{-2\sqrt{-1}\pi(\alpha  \{y_1-y_2\}+\gamma \{y_2\})}}{\left(-1+e^{-2\sqrt{-1}\pi  \alpha }\right) \left(-1+e^{-2\sqrt{-1}\pi  \gamma
   }\right) (\alpha +\beta -\gamma )^2}
\end{split}
\end{equation}
and
\begin{equation}
  \begin{split}
    S((2,1,1),(y_1,y_2);\Lambda)&=\lim_{N\to\infty}
    \sum_{\substack{m,n\in\mathbb{Z}\\%m+\alpha,n+\beta,m+n+\gamma\neq 0\\
|m|,|n|\leq N}}
    \frac{e^{2\pi\sqrt{-1}(my_1+ny_2)}}{(m+\alpha)^{2}(n+\beta)^{1}(m+n+\gamma)^{1}}\\
    &=\frac{-(2\pi\sqrt{-1})^2}{2!}\frac{-(2\pi\sqrt{-1})^1}{1!}\frac{-(2\pi\sqrt{-1})^1}{1!}    
    C((2,1,1),(y_1,y_2);\Lambda).
  \end{split}
\end{equation}
If $\alpha=0$ and $\beta,\gamma\notin\mathbb{Z}$ with $\beta\neq \gamma$, we have
\begin{equation}
  \begin{split}
    C((0,1,2),(y_1,y_2);\Lambda)=
% \\
&
\frac{\sqrt{-1} \{y_2\} e^{-2 \sqrt{-1} \pi  \gamma  \{y_2\}}}{2 \left(-1+e^{-2 \sqrt{-1} \pi  \gamma }\right) \pi  (\beta - \gamma )}
% \\
% &
-\frac{\sqrt{-1} e^{-2 \sqrt{-1} \pi  \gamma (1+\{y_2\})}}{2 \left(-1+e^{-2 \sqrt{-1} \pi  \gamma }\right)^2 \pi  (\beta -\gamma )}
\\
&
+\frac{e^{-2 \sqrt{-1} \pi  \beta  \{y_2\}}}{4 \left(-1+e^{-2 \sqrt{-1} \pi  \beta }\right) \pi^2  (\beta -\gamma )^2}
% \\
% &
-\frac{e^{-2 \sqrt{-1} \pi  \gamma  \{y_2\}}}{4 \left(-1+e^{-2 \sqrt{-1} \pi  \gamma }\right) \pi^2 (  \beta - \gamma )^2}
\end{split}
\end{equation}
and
\begin{equation}
  \begin{split}
    S((0,1,2),(y_1,y_2);\Lambda)&=
    -\lim_{N\to\infty}
  \sum_{\substack{m,n\in\mathbb{Z}\\|m|,|n|\leq N\\m=0}}
%  \sum_{\substack{m,n\in\mathbb{Z}\\|m|,|n|\leq N\\n+\beta,m+n+\gamma\neq 0\\m+\alpha=0}}
  \frac{e^{2\pi\sqrt{-1}(my_1+ny_2)}}{(n+\beta)^{1}(m+n+\gamma)^{2}}
\\
&
=
-\lim_{N\to\infty}
%  \sum_{\substack{n\in\mathbb{Z}\\|n|\leq N\\n+\beta,n+\gamma-\alpha\neq 0}}
  \sum_{\substack{n\in\mathbb{Z}\\|n|\leq N}}
  \frac{e^{2\pi\sqrt{-1}ny_2}}{(n+\beta)^{1}(n+\gamma)^{2}}
\\
    &=\frac{-(2\pi\sqrt{-1})^0}{0!}\frac{-(2\pi\sqrt{-1})^1}{1!}\frac{-(2\pi\sqrt{-1})^2}{2!}    
    C((0,1,2),(y_1,y_2);\Lambda).
  \end{split}
\end{equation}
\end{example}

\begin{example}\label{Exam-A1}
Let $V=\mathbb{R}$.
Let 
\begin{align}
  \Lambda&=\Lambda_{\alpha}=\{f_{-1}=(-1,\alpha),f_0=(1,0),f_1=(1,\alpha)\},\\
  \mathscr{B}&=\{\{f_{-1}\},\{f_0\},\{f_1\}\},
% \\
%   \mathfrak{H}_{\mathscr{R}}&=\mathbb{Z},
\end{align}
where $\alpha\in\mathbb{C}\setminus\{0\}$,
which corresponds to the series
\begin{equation}
  S(\mathbf{k},y;\Lambda_{\alpha})=\lim_{N\to\infty}
\sum_{\substack{m\in\mathbb{Z}\setminus\{0,\pm\alpha\}\\|m|\leq N}}\frac{e^{2\pi\sqrt{-1}my}}{(-m+\alpha)^{k_{-1}}m^{k_0}(m+\alpha)^{k_1}},
\end{equation}
where $k_{-1},k_0,k_1\in\mathbb{N}$. Then the generating function is given by
\begin{equation}
  \begin{split}
  &F((t_{-1},t_0,t_1),y;\Lambda_{\alpha})=
  \\
  &\quad\frac{t_0}{t_0 + (t_{-1} - 2 \pi \sqrt{-1}\alpha)}\frac{t_1}{t_1 - 2 \pi \sqrt{-1}\alpha + (t_{-1} - 2 \pi \sqrt{-1}\alpha)}\frac{t_{-1}e^{(t_{-1} - 2 \pi \sqrt{-1}\alpha)(1-\{y\})}}{e^{t_{-1} - 2 \pi \sqrt{-1}\alpha} - 1} 
\\
&\quad+
 \frac{t_{-1}}{t_{-1} - 2 \pi \sqrt{-1}\alpha + t_0}\frac{t_1}{t_1 - 2 \pi \sqrt{-1}\alpha - t_0}\frac{t_0e^{t_0\{y\}}}{e^{t_0} - 1} 
\\
&\quad+
\frac{t_{-1}}{t_{-1} - 2 \pi \sqrt{-1}\alpha + (t_1 - 2 \pi \sqrt{-1}\alpha)}\frac{t_0}{t_0 - (t_1 - 2 \pi \sqrt{-1}\alpha)}
 \frac{t_1e^{(t_1 - 2 \pi\sqrt{-1}\alpha)\{y\}}}{e^{t_1 - 2 \pi\sqrt{-1}\alpha} - 1}.
 \end{split}
\end{equation}
% \begin{equation}
%   \begin{split}
%   &F(t_{-1},t_0,t_1,y;\Lambda)=
%   \\
%   &\quad\frac{t_0}{t_0 - (t_{-1} + 2 \pi \sqrt{-1}\alpha)}\frac{t_1}{t_1 - 2 \pi \sqrt{-1}\alpha - (t_{-1} + 2 \pi \sqrt{-1}\alpha)}\frac{t_{-1}e^{(t_{-1} + 2 \pi \sqrt{-1}\alpha)\{y\}}}{e^{t_{-1} + 2 \pi \sqrt{-1}\alpha} - 1} 
% \\
% &\quad+
%  \frac{t_{-1}}{t_{-1} + 2 \pi \sqrt{-1}\alpha - t_0}\frac{t_1}{t_1 - 2 \pi \sqrt{-1}\alpha - t_0}\frac{t_0e^{t_0\{y\}}}{e^{t_0} - 1} 
% \\
% &\quad+
% \frac{t_{-1}}{t_{-1} + 2 \pi \sqrt{-1}\alpha - (t_1 - 2 \pi \sqrt{-1}\alpha)}\frac{t_0}{t_0 - (t_1 - 2 \pi \sqrt{-1}\alpha)}
%  \frac{t_1e^{(t_1 - 2 \pi\sqrt{-1}\alpha)\{y\}}}{e^{t_1 - 2 \pi\sqrt{-1}\alpha} - 1}.
%  \end{split}
% \end{equation}
Then for $\alpha\notin\mathbb{Z}$,
\begin{equation}
  \begin{split}
    C((2,2,2),y;\Lambda_{\alpha})&=
  -\frac{1}{4 \pi ^6 \alpha^6}
  +\frac{1}{24 \pi ^4 \alpha^4}
  -\frac{\{y\}}{4 \pi ^4 \alpha^4}
  +\frac{\{y\}^2}{4 \pi ^4 \alpha^4}
  -\frac{3 \sqrt{-1} e^{2 \pi\sqrt{-1}  \alpha \{y\}}}{16 \pi ^5 \alpha^5 \left(-1+e^{2 \pi\sqrt{-1}  \alpha}\right)^2}
  \\
  &
  -\frac{3 \sqrt{-1} e^{2 \pi\sqrt{-1}  \alpha(1-\{y\})}}{16 \pi ^5 \alpha^5\left(-1+e^{2 \pi\sqrt{-1}  \alpha}\right)^2}
  +\frac{3 \sqrt{-1} e^{2 \pi\sqrt{-1}  \alpha(2-\{y\})}}{16 \pi ^5 \alpha^5 \left(-1+e^{2 \pi\sqrt{-1}  \alpha}\right)^2}
  +\frac{3 \sqrt{-1} e^{2 \pi\sqrt{-1}  \alpha (\{y\}+1)}}{16 \pi ^5 \alpha^5 \left(-1+e^{2 \pi\sqrt{-1}  \alpha}\right)^2}
  \\
  &
  -\frac{\{y\} e^{2 \pi\sqrt{-1}  \alpha \{y\}}}{8 \pi ^4 \alpha^4 \left(-1+e^{2 \pi\sqrt{-1}  \alpha}\right)^2}
  +\frac{\{y\} e^{2 \pi\sqrt{-1}  \alpha(1-\{y\})}}{8 \pi ^4 \alpha^4 \left(-1+e^{2 \pi\sqrt{-1}  \alpha}\right)^2}
  -\frac{\{y\} e^{2 \pi\sqrt{-1}  \alpha(2-\{y\})}}{8 \pi ^4 \alpha^4 \left(-1+e^{2 \pi\sqrt{-1}  \alpha}\right)^2}
  \\
  &
  +\frac{\{y\} e^{2 \pi\sqrt{-1}  \alpha (\{y\}+1)}}{8 \pi ^4 \alpha^4 \left(-1+e^{2 \pi\sqrt{-1}  \alpha}\right)^2}
  -\frac{e^{2 \pi\sqrt{-1}  \alpha(1-\{y\})}}{8 \pi ^4 \alpha^4 \left(-1+e^{2 \pi\sqrt{-1}  \alpha}\right)^2}
  -\frac{e^{2 \pi\sqrt{-1}  \alpha (\{y\}+1)}}{8 \pi ^4 \alpha^4 \left(-1+e^{2 \pi\sqrt{-1}  \alpha}\right)^2}
\end{split}
\end{equation}
and for $\alpha\in\mathbb{Z}$,
\begin{equation}
  \begin{split}
    C((2,2,2),y;\Lambda_{\alpha})&=
-\frac{1}{4 \pi ^6 \alpha^6}
+\frac{1}{24 \pi ^4 \alpha^4}
-\frac{\{y\}}{4 \pi ^4 \alpha^4}
+\frac{\{y\}^2}{4 \pi ^4 \alpha^4}
-\frac{3 \sqrt{-1} \{y\} e^{-2 \pi\sqrt{-1}  \alpha \{y\}}}{16 \pi ^5 \alpha^5}
\\
&
+\frac{3 \sqrt{-1} \{y\} e^{2 \pi\sqrt{-1}  \alpha \{y\}}}{16 \pi ^5 \alpha^5}
+\frac{3 \sqrt{-1} e^{-2 \pi\sqrt{-1}  \alpha \{y\}}}{32 \pi ^5 \alpha^5}
-\frac{3 \sqrt{-1} e^{2 \pi\sqrt{-1}  \alpha \{y\}}}{32 \pi ^5 \alpha^5}
\\
&+\frac{\{y\}^2 e^{-2 \pi\sqrt{-1}  \alpha \{y\}}}{16 \pi ^4 \alpha^4}
+\frac{\{y\}^2 e^{2 \pi\sqrt{-1}  \alpha \{y\}}}{16 \pi ^4 \alpha^4}
-\frac{\{y\} e^{-2 \pi\sqrt{-1}  \alpha \{y\}}}{16 \pi ^4 \alpha^4}
-\frac{\{y\} e^{2 \pi\sqrt{-1}  \alpha \{y\}}}{16 \pi ^4 \alpha^4}
\\
&-\frac{23 e^{-2 \pi\sqrt{-1}  \alpha \{y\}}}{128 \pi ^6 \alpha^6}
-\frac{23 e^{2 \pi\sqrt{-1}  \alpha \{y\}}}{128 \pi ^6 \alpha^6}
+\frac{e^{-2 \pi\sqrt{-1}  \alpha \{y\}}}{96 \pi ^4 \alpha^4}
+\frac{e^{2 \pi\sqrt{-1}  \alpha \{y\}}}{96 \pi ^4 \alpha^4}.
\end{split}
\end{equation}
For example, setting $y=0$ and $\alpha=1,2,3$, we obtain
\begin{align*}
S((2,2,2),0;\Lambda_1) & =\sum_{\substack{m\in\mathbb{Z}\setminus\{0,\pm 1\}}}\frac{1}{(-m+1)^{2}m^{2}(m+1)^{2}}=\frac{1}{2}\pi^2 - \frac{39}{8},\\
S((2,2,2),0;\Lambda_2) & =\sum_{\substack{m\in\mathbb{Z}\setminus\{0,\pm 2\}}}\frac{1}{(-m+2)^{2}m^{2}(m+2)^{2}}=\frac{1}{32}\pi^2 - \frac{39}{512},\\
S((2,2,2),0;\Lambda_3) & =\sum_{\substack{m\in\mathbb{Z}\setminus\{0,\pm 3\}}}\frac{1}{(-m+3)^{2}m^{2}(m+3)^{2}}=\frac{1}{162}\pi^2 - \frac{13}{1944}.
\end{align*}
Similarly, computing $C((2k,2k,2k),0;\Lambda_{\alpha})$, we can obtain
\begin{align*}
S((4,4,4),0;\Lambda_1)&=\frac{1}{40}\pi^4 + \frac{35}{16}\,\pi^2 - \frac{3075}{128},\\
S((6,6,6),0;\Lambda_2)&=\frac{11}{20643840}\pi^6 + \frac{21}{2097152}\pi^4 + \frac{3003}{16777216}\pi^2 - \frac{137067}{268435456},\\
S((8,8,8),0;\Lambda_3)&=\frac{43}{8678218953600}\pi^8 + \frac{367}{7810397058240}\pi^6 + \frac{581}{1983592903680}\pi^4\\
&\quad + \frac{46189}{21422803359744}\pi^2 - \frac{2864587}{1028294561267712}.
\end{align*}
Here we define the zeta-function associated with $\Lambda_{\alpha}$ by 
\begin{equation}\label{A1-zeta}
  \zeta((s_1,s_2,s_3);\Lambda_{\alpha})
=\sum_{\substack{m=1 \\ m\neq \pm \alpha}}^{\infty}\frac{1}{(-m+\alpha)^{s_1}m^{s_2}(m+\alpha)^{s_3}},
\end{equation}
which can be regarded as a Hurwitz-type analogue of the Riemann zeta-function, that is, with a shifting parameter $\alpha$. 
We can easily check that $S((2k,2k,2k),0;\Lambda_{\alpha})=2\zeta((2k,2k,2k);\Lambda_{\alpha})$ for $k\in \mathbb{N}$. Therefore we obtain from the above results that, for example, 
\begin{align*}
\zeta((2,2,2);\Lambda_1) & =\frac{1}{4}\pi^2 - \frac{39}{16},\\
\zeta((4,4,4);\Lambda_1) & =\frac{1}{80}\pi^4 + \frac{35}{32}\,\pi^2 - \frac{3075}{256},\\
\zeta((6,6,6);\Lambda_2) & =\frac{11}{41287680}\pi^6 + \frac{21}{4194304}\pi^4 + \frac{3003}{33554432}\pi^2 - \frac{137067}{536870912}.
\end{align*}
\end{example}

\begin{example}\label{Exam-A2}
Let $V=\mathbb{R}^2$, and $\alpha\in\mathbb{C}\setminus\{0\}$.
Let 
\begin{align}
  \Lambda&=\Lambda_{\alpha}=\{\{f_{1j}\}_j,\{ f_{2j} \}_j,\{f_{3j}\}_j\}\\
  &=\{\{(-1,0,\alpha),(1,0,0), (1,0,\alpha)\},\ 
  \{(0,-1,\alpha),(0,1,0),(0,1,\alpha)\},\ \{(-1,-1,\alpha),(1,1,0),(1,1,\alpha)\}\},\notag\\
  \mathscr{B}&=\{\{f_{1j},f_{2l}\}_{j,l},\{f_{1j},f_{3l}\}_{j,l},\{f_{2j},f_{3l}\}_{j,l}\}.
\end{align}
%where $\alpha\in\mathbb{C}\setminus\{0\}$,
Set $\mathbf{y}=0$ and 
\begin{align*}
S\left( \{k_j \}_{1\leq j\leq 9},0;\Lambda_{\alpha} \right)&=\lim_{N\to\infty} \sum_{\substack{m,n\in\mathbb{Z}\setminus\{0,\pm\alpha\}\\ m+n\neq 0,\pm \alpha\\ |m|,|n|\leq N}}\frac{1}{(-m+\alpha )^{k_1}m^{k_2}(m+\alpha )^{k_3}}\\
& \quad \times \frac{1}{(-n+\alpha)^{k_4}n^{k_5}(n+\alpha )^{k_6}(-(m+n)+\alpha )^{k_7}(m+n)^{k_8}(m+n+\alpha )^{k_9}}.
\end{align*}
Then, computing $C(\{k_j\},0;\Lambda_{\alpha})$, we obtain, for example,
\begin{align*}
S((1,2,2,2,1,1,1,2,2),0;\Lambda_1)&=\frac{1}{1890}\pi^6 + \frac{701}{2160}\pi^4 - \frac{1841}{108}\pi^2 + \frac{2822557}{20736},\\
S((2,2,2,2,2,2,2,2,2),0;\Lambda_2)&=\frac{11}{15482880}\pi^6 + \frac{4901}{70778880}\pi^4 -\frac{ 26747}{28311552}\pi^2 + \frac{20643217}{10871635968},\\
S((1,1,1,2,2,2,1,1,1),0;\Lambda_3)&=\frac{2}{295245}\pi^4 - \frac{227}{6377292}\pi^2 + \frac{14183}{459165024}.
\end{align*}
Similarly to Example \ref{Exam-A1}, we define the zeta-function associated with $\Lambda$ by
\begin{align}
\zeta_2(\{s_j\}_{1\leq j\leq 9};\Lambda_{\alpha})& =\sum_{\substack{m,n=1 \\ m\neq \pm \alpha \\ n\neq \pm \alpha \\ m+n\neq \pm \alpha}}^\infty \frac{1}{(-m+\alpha )^{s_1}m^{s_2}(m+\alpha )^{s_3}} \label{A2-zeta}\\
& \quad \times \frac{1}{(-n+\alpha )^{s_4}n^{s_5}(n+\alpha )^{s_6}(-(m+n)+\alpha )^{s_7}(m+n)^{s_8}(m+n+\alpha )^{s_9}}, \notag
\end{align}
which can be regarded as a Hurwitz-type analogue of the zeta-function of the root system of type $A_2$ 
defined by 
\begin{equation}
\zeta_2((s_1,s_2,s_3);A_2)=\sum_{m,n=1}^\infty \frac{1}{m^{s_1}n^{s_2}(m+n)^{s_3}} \label{A2}
\end{equation}
(see \cite[Section 2]{KM2} \cite[Section 11.7, Example 2]{KM3}). Note that \eqref{A2} is also called the Tornheim double sum or the Mordell-Tornheim double zeta-function (see, for example, \cite{M02,To}). 
We already studied certain Hurwitz-type analogues of zeta-functions of root systems in 
\cite[Section 8]{KM5}. 
From the viewpoint of root systems, we can regard 
$S\left(\{2k\}_{1\leq j\leq 9},0;\Lambda_{\alpha} \right)$ is the sum of zeta values 
$\zeta_2(\{2k\}_{1\leq j\leq 9};\Lambda_{\alpha})$ under the action of the Weyl group of type 
$A_2$ $(\simeq S_3)$. This implies that
$$S\left(\{2k\}_{1\leq j\leq 9},0;\Lambda_{\alpha} \right)=6\zeta_2(\{2k\}_{1\leq j\leq 9};\Lambda_{\alpha})\quad (k\in \mathbb{N}). $$
Therefore, as an analogue of $\zeta_2((2,2,2);A_2)=\pi^6/2835$, we obtain from the above result that
\begin{align*}
\zeta_2((2,2,2,2,2,2,2,2,2);\Lambda_2)=\frac{11}{92897280}\pi^6 + \frac{4901}{424673280}\pi^4 -\frac{ 26747}{ 169869312 }\pi^2 + \frac{20643217}{65229815808}.
\end{align*}

\end{example}

\begin{remark}
We give another interpretation of
 the series \eqref{A1-zeta} and \eqref{A2-zeta}, that is,
we regard
each term of these series
 as a product of positive roots of affine root system $A_1^{(1)}$ and $A_2^{(1)}$ respectively 
(for the theory of affine root systems, see \cite{Kac}).
Since there are infinitely many positive roots in affine root systems,
the product consists of infinitely many factors.
 In order for the infinite product to make sense, we understand that infinitely many variables are set to be zero and hence the product is truncated.
\end{remark}

%%%%%%%%%%%%%%%%%%%%%%%%%%%%%%%%%%%%%%%%%%%%%%%%%%%%%%%%%%%%%%%%%%
\section{Proof of Theorem \ref{thm:main0}}\label{sec-4}
%%%%%%%%%%%%%%%%%%%%%%%%%%%%%%%%%%%%%%%%%%%%%%%%%%%%%%%%%%%%%%%%%

Now we start the proofs of the main theorems.
First of all, in this section, we prove Theorem \ref{thm:main0}.
The main body of the argument is the proof of
an evaluation formula (Proposition \ref{prop:main}) for 
$S(\mathbf{k},\mathbf{y};\Lambda)$.

For $t,b\in\mathbb{C}$ and $y\in\mathbb{R}$
let
\begin{equation}\label{F_def_exp}
  F(t,y;b)=\frac{t e^{(t-2\pi \sqrt{-1}b)y}}{e^{t-2\pi \sqrt{-1}b}-1}=\sum_{k=0}^\infty C(k,y;b)\frac{t^k}{k!},
\end{equation}
where the right-hand side converges when $|t|$ is sufficiently small.

It is to be noted that $F(t,\{y\};b)$ (resp.~$C(k,\{y\};b)$) is just the special case $r=1$,
$\Lambda=\{(1,b)\}=B$ of $F(\mathbf{t},\mathbf{y};\Lambda)$ defined by \eqref{eq:exp_F}
(resp.~$C(\mathbf{k},\mathbf{y};\Lambda)$ defined by \eqref{F_Taylor_exp}).
%We will show the following proposition, which a first step of the main theorem.
%the series in Definition \ref{def:S} converges under the given condition and their values are given as in Proposition \ref{prop:main}.

\begin{prop}
\label{prop:main}
For $\mathbf{k}=(k_f)_{f\in\Lambda}\in\mathbb{N}_{0}^{\sharp \Lambda}$,
$\mathbf{y}\in V\setminus\bigcup_{f\in \widetilde{\Lambda}\cap\Lambda_1}(\mathfrak{H}_{\Lambda\setminus\{f\}}+\mathbb{Z}^r)$,
the series \eqref{eq:S} converges.
For a fixed decomposition $\Lambda=B_0\cup L_0$ with $B_0=\{f_1,\ldots,f_r\}\in\mathscr{B}$, we have
\begin{equation}
\label{eq:S_int}
  \begin{split}
    S(\mathbf{k},\mathbf{y};\Lambda)
    &=
    \frac{1}{\sharp(\mathbb{Z}^r/\langle\vec{B}_0\rangle)}
    \prod_{f\in \Lambda}
    \Bigl(
    -\frac{(2\pi\sqrt{-1})^{k_f}}{k_f!}   
    \Bigr)
    \sum_{\mathbf{w}\in\mathbb{Z}^r/\langle\vec{B}_0\rangle}  
    \int_0^1\dots\int_0^1
    \Bigl(\prod_{g\in L_0}    C(k_g,x_g;\const{g}) dx_g\Bigr)
    \\
    &\qquad\times
    \prod_{f\in B_0}
    \Bigl(
%    C(k_f,\{\langle\mathbf{y}+\mathbf{w}-\sum_{g\in L_0}x_g\vec{g},\vec{f}^{B_0}\rangle\};\const{f})
    C(k_f,\{\mathbf{y}+\mathbf{w}-\sum_{g\in L_0}x_g\vec{g}\}_{B_0,f};\const{f})
    \Bigr).
  \end{split}
\end{equation}
\end{prop}

This is a generalization of \cite[Theorem 6]{KM3} (for integral values of $\mathbf{k}$).
Only the case in the domain of absolute convergence was considered in \cite[Theorem 6]{KM3},
so there was no problem of convergence.
In our present situation,
if $k_f\geq2$ for all $f\in B$ with some fixed $B\in\mathscr{B}$, then the matter of convergence
is again obvious, so it is easy to prove our claims.
%see that the corresponding series converges absolutely uniformly in $\mathbf{y}$. 
However if $k_f=1$ for sufficiently many $f\in \Lambda$, then 
there are subtle problems on convergence, and the proof becomes much more complicated.
It should be remarked that the key of the convergence of $S(\mathbf{k},\mathbf{y};\Lambda)$ is the condition $\rank\langle\vec{\Lambda}\rangle=r$.

Since it is difficult to treat \eqref{eq:S} directly,
in the following we consider a little modified sum
\begin{equation}
 \label{eq:S_1}
  S_1(\mathbf{k},\mathbf{y};\Lambda;B_0)=
  \lim_{N\to\infty}
  Z_1(N;\mathbf{k},\mathbf{y};\Lambda;B_0),
  \end{equation}
where
\begin{gather}
\label{Z_1_def}
  Z_1(N;\mathbf{k},\mathbf{y};\Lambda;B_0)=
  (-1)^{\sharp\Lambda_0}
  \sum_{\substack{\mathbf{v}\in\mathbb{Z}^{r}\\
      |\Re f(\mathbf{v})|\leq N\quad(f\in B_0)\\
      % |v_j|\leq N\quad(1\leq j\leq r)\\
      f(\mathbf{v})\neq 0\quad(f\in\Lambda_+)\\
      f(\mathbf{v})=0\quad(f\in\Lambda_0)}}
  e^{2\pi\sqrt{-1}\langle \mathbf{y},\mathbf{v}\rangle}
  \prod_{f\in\Lambda_+}
  \frac{1}{f(\mathbf{v})^{k_f}}.
%  \\
%  \label{eq:S_1}
%  S_1(\mathbf{k},\mathbf{y};\Lambda)=
%  \lim_{N\to\infty}
%  Z_1(N),
\end{gather}
That is, the condition
$|v_j|\leq N$ for $1\leq j\leq r$ 
in the definition of $S(\mathbf{k},\mathbf{y};\Lambda)$
is replaced by
$|\Re f(\mathbf{v})|\leq N$ for $f\in B_0$.
At the last stage of the proof we will show that 
$S(\mathbf{k},\mathbf{y};\Lambda)=S_1(\mathbf{k},\mathbf{y};\Lambda;B_0)$.
In particular, we will find that $S_1(\mathbf{k},\mathbf{y};\Lambda;B_0)$ actually does not depend
on the choice of $B_0$.

The proof of Proposition \ref{prop:main} consists of three steps.
\bigskip

%%%%%%%%%%%%%%%%%%%%%%%%%%%%%%%%%%%%
{\it The first step}.
%%%%%%%%%%%%%%%%%%%%%%%%%%%%%%%%%%%%
We first consider the simplest case of \eqref{eq:S_1}, which corresponds to $r=1$ and $\Lambda=\{(1,b)\}=B$ with $\mathscr{B}=\{B\}$, in Lemmas \ref{lm:convkn0} and \ref{lm:convk1}. 
%It should be noted that in the case $k=0$, the series \eqref{eq:C0} consists of only one term or no term.
\begin{lemma}
\label{lm:convkn0}
%  Theorem \ref{thm:main1} is true in the case $r=1$ and $\Lambda=\{(1,b)\}=B$ with $\mathscr{B}=\{B\}$.
%Assume that $r=1$ and $\Lambda=\{1,b\}$. 
For $b\in\mathbb{C}$, $y\in\mathbb{R}$ and $k\in\mathbb{N}$ with $k\geq 2$, 
we have
\begin{align}
%\label{eq:Cauchy}
  \label{eq:Ck}
  \lim_{N\to\infty}
  \Biggl(\sum_{\substack{n\in\mathbb{Z}\\|n+\Re b|\leq N\\n+b\neq 0}}\frac{e^{2\pi \sqrt{-1} ny}}{(n+b)^k}\Biggr)
  &=-\frac{(2\pi\sqrt{-1})^k}{k!}C(k,\{y\};b),
  \\
  \label{eq:C0}
  \lim_{N\to\infty}
  \Biggl(-\sum_{\substack{n\in\mathbb{Z}\\|n+\Re b|\leq N\\n+b=0}}e^{2\pi\sqrt{-1}ny}\Biggr)
  &=-C(0,\{y\};b).
\end{align}
(Actually the sum on the left-hand side of \eqref{eq:C0} consists of at
most one term.) 
The series above 
converge absolutely uniformly in $y$ and hence $C(k,\{y\};b)$ and $C(0,\{y\};b)$ are continuous in $y$.
% \begin{equation}
%   S(k,y;\Lambda)=-\frac{(2\pi\sqrt{-1})^k}{k!}C(k,y;b).
% \end{equation}
\end{lemma}
\begin{proof}
%Note that $M=B$ and $N=\emptyset$ in the case.
Let $\gamma_{X,Y}$ be the counterclockwise rectangle contour with vertices at $\pm X\pm 2\pi\sqrt{-1}Y$.
Applying the Cauchy theorem to the integral 
\begin{equation}
%\label{eq:repCauchy}
  \lim_{N\to \infty}\lim_{M\to \infty}\int_{\gamma_{M,N+\epsilon}} F(t,\{y\};b) t^{-k-1}dt\qquad(k\geq 2)
\end{equation}
with a sufficiently small $\epsilon>0$,
% $0<\epsilon<1$ such that
% $\min_{m\in\mathbb{Z}}|b+m\pm\epsilon-\frac{1}{2}|<\frac{1}{4}$,
we see that the sum of all the residues vanishes, namely,
\begin{equation}
%\label{eq:Cauchy}
  \frac{C(k,\{y\};b)}{k!}+\frac{1}{(2\pi\sqrt{-1})^k}
  \lim_{N\to\infty}
  \sum_{\substack{n\in\mathbb{Z}\\|n+\Re b|\leq N\\n+b\neq 0}}\frac{e^{2\pi \sqrt{-1} ny}}{(n+b)^k}=0,
\end{equation}
and hence \eqref{eq:Ck}.

The left-hand side of \eqref{eq:C0} consists of only one term $e^{-2\pi\sqrt{-1}by}$ if $b\in\mathbb{Z}$, and vanishes if $b\not\in\mathbb{Z}$, %Thus we have
%\begin{equation}
%\label{eq:Cauchy4}
%  \lim_{N\to\infty}
%  \sum_{\substack{n\in\mathbb{Z}\\|n+\Re b|\leq N\\n+b=0}}e^{2\pi\sqrt{-1}ny}
%=
%  \begin{cases}
%    e^{-2\pi\sqrt{-1}by}&\qquad (b\in\mathbb{Z}),\\
%    0&\qquad (b\notin\mathbb{Z}),
%  \end{cases}
%\end{equation}
while
\begin{equation}
  C(0,\{y\};b)=F(0,\{y\};b)=
  \begin{cases}
    e^{-2\pi\sqrt{-1}by}&\qquad (b\in\mathbb{Z}),\\
    0&\qquad (b\notin\mathbb{Z})
  \end{cases}
\end{equation}
and hence \eqref{eq:C0}.

% Later we use the following. 
% For $k\neq 1$, 
Both in \eqref{eq:Ck} and \eqref{eq:C0},
the absolute uniform convergence of 
the series in $y$ is clear.
% by \eqref{eq:Cauchy} and \eqref{eq:Cauchy4}.
% The series are absolutely uniformly convergent in $y$
%  by \eqref{eq:Cauchy} and \eqref{eq:Cauchy4}.
\end{proof}

The case $k=1$ is more subtle.   We first prepare the following

\begin{lemma}
\label{lm:Ei}
For $\mu>0$ there exists $K>0$ such that
for $a,z>0$
  \begin{equation}
    \int_0^\infty \frac{e^{-xz}}{\sqrt{x^2+a^2}}dx\leq K (az)^{-\frac{1}{\mu+1}}.
  \end{equation}
\end{lemma}
\begin{proof}
Rewrite
  \begin{equation}
    \int_0^\infty \frac{e^{-xz}}{\sqrt{x^2+a^2}}dx
=
    \int_0^\infty \frac{e^{-x}}{\sqrt{x^2+(az)^2}}dx.
  \end{equation}  
By the inequality of weighted arithmetic and geometric means 
\begin{equation}
  \mu A+B
  \geq(\mu+1)(A^\mu B)^{\frac{1}{\mu+1}}\qquad(\mu,A,B>0),
\end{equation}
% \begin{equation}
%   A^2+B^2=
%   w\frac{A^2}{w}+B^2
% \geq(w+1)\Bigl(\Bigl(\frac{A^2}{w}\Bigr)^wB^2\Bigr)^{\frac{1}{w+1}}
% =(w+1)w^{-\frac{w}{w+1}}A^{\frac{2w}{w+1}}B^{\frac{2}{w+1}},
% \end{equation}
we have
\begin{equation}
  \sqrt{x^2+(az)^2}\geq
\sqrt{\mu+1}\mu^{-\frac{\mu}{2(\mu+1)}}x^{\frac{\mu}{\mu+1}}(az)^{\frac{1}{\mu+1}}
\end{equation}
and
  \begin{equation}
    \begin{split}
      \int_0^\infty \frac{e^{-x}}{\sqrt{x^2+(az)^2}}dx
      &\leq
      \frac{\mu^{\frac{\mu}{2(\mu+1)}}}{\sqrt{\mu+1}}(az)^{-\frac{1}{\mu+1}}\int_0^\infty e^{-x} x^{-\frac{\mu}{\mu+1}}dx
      \\
      &\leq
      \Bigl(\frac{\mu^{\frac{\mu}{2(\mu+1)}}}{\sqrt{\mu+1}}\Gamma\Bigl(\frac{1}{\mu+1}\Bigr)\Bigr)(az)^{-\frac{1}{\mu+1}}.
%      \Bigl(\frac{w^{\frac{w}{2(w+1)}}}{\sqrt{w+1}}\int_0^\infty \frac{e^{-x}}{x^{\frac{w}{w+1}}}dx\Bigr)(az)^{-\frac{1}{w+1}}.
    \end{split}
  \end{equation}
% We calculate the integral in the right-hand side.
% Since
% \begin{equation}
% \int_y^\infty\frac{e^{-x}}{x^{\frac{w}{w+1}}}dx
% =
% \int_0^\infty\frac{e^{-x}}{x^{\frac{w}{w+1}}}dx
% -\int_0^y\frac{e^{-x}}{x^{\frac{w}{w+1}}}dx
% -(w+1)y^{\frac{1}{w+1}}
% +\int_0^{y}\frac{1}{x^{\frac{w}{w+1}}}dx
% \end{equation}
% we have
% \begin{equation}
%     \int_y^\infty\frac{e^{-x}}{x^{\frac{w}{w+1}}}dx
% =-(w+1)y^{\frac{1}{w+1}}+K'+\int_0^{y}\frac{1-e^{-x}}{x^{\frac{w}{w+1}}}dx
% \end{equation}
% for some $K'\in\mathbb{R}$
% % \begin{equation}
% %   \frac{d}{dy}\Bigl(\int_y^\infty\frac{e^{-x}}{x^{\frac{w}{w+1}}}dx+(w+1)y^{\frac{1}{w+1}}\Bigr)
% % =\frac{1-e^{-y}}{y^{\frac{w}{w+1}}},
% % \end{equation}
% % we have
% % \begin{equation}
% %     \int_y^\infty\frac{e^{-x}}{x^{\frac{w}{w+1}}}dx
% % =-(w+1)y^{\frac{1}{w+1}}+K'+\int_0^{y}\frac{1-e^{-x}}{x^{\frac{w}{w+1}}}dx
% % \end{equation}
% % for some $K'\in\mathbb{R}$
%  and hence
% \begin{equation}
%     \int_R^\infty \frac{e^{-xz}}{\sqrt{x^2+a^2}}dx
% % \frac{m^{\frac{m}{2(m+1)}}}{(az)^{\frac{1}{m+1}}}\int_{Rz}^\infty \frac{e^{-x}}{x^{\frac{m}{m+1}}}dx
% \leq K (az)^{-\frac{1}{w+1}}(1+z^{\frac{1}{w+1}}).
% \end{equation}

\end{proof}

Let 
\begin{equation}
  \delta_{N<0<M}=
  \begin{cases}
    1\qquad&(N<0<M),\\
    0\qquad&\text{otherwise}.
  \end{cases}
\end{equation}

\begin{lemma}
\label{lm:convk1}
% Assume that $r=1$ and $\Lambda=\{1,b\}$. For $y\in\mathbb{R}\setminus\mathbb{Z}$ and $k=1$, the series \eqref{eq:cf} converges conditionally and
For $b\in\mathbb{C}$ and $y\in\mathbb{R}\setminus\mathbb{Z}$
% and $k\in\mathbb{N}_0$ with $k\neq 1$, 
we have
\begin{equation}
\label{eq:C1}
  \lim_{N\to\infty}
  \Biggl(\sum_{\substack{n\in\mathbb{Z}\\|n+\Re b|\leq N\\n+b\neq 0}}\frac{e^{2\pi \sqrt{-1} ny}}{n+b}\Biggr)
  =-2\pi\sqrt{-1}C(1,\{y\};b)
\end{equation}
and $C(1,\{y\};b)$ is continuous in $y$.
% \begin{equation}
%   S(1,y;\Lambda)=-2\pi\sqrt{-1}C(1,y;b).
% \end{equation}
Moreover 
for any $\mu>0$
there exists $K>0$ such that for all $y\in\mathbb{R}\setminus\mathbb{Z}$ and
all $M,N\in\mathbb{R}$ with sufficiently large $|M|,|N|$ and $M\geq|N|$
\begin{equation}\label{est:C1}
  \Bigl| \sum_{\substack{n\in\mathbb{Z}\\N\leq n+\Re b\leq M\\n+b\neq 0}}\frac{e^{2\pi\sqrt{-1}ny}}{n+b}\Bigr|
\leq K |N|^{-\frac{1}{\mu+1}}((1-\{y\})^{-\frac{1}{\mu+1}}+\{y\}^{-\frac{1}{\mu+1}})+\delta_{N<0<M}K.
\end{equation}
\end{lemma}
\begin{proof}
From \eqref{F_def_exp} we can easily see that
\begin{equation}
  C(1,\{y\};b)=
  \begin{cases}
    \Bigl(\{y\}-\dfrac{1}{2}\Bigr)e^{-2\pi\sqrt{-1}b\{y\}}
    \qquad&(b\in\mathbb{Z}),
    \\[2mm]
    \dfrac{e^{-2\pi\sqrt{-1}b\{y\}}}{e^{-2\pi\sqrt{-1}b}-1}\qquad&(b\notin\mathbb{Z}),
  \end{cases}
\end{equation}
from which the continuity of $C(1,\{y\};b)$ follows.

Let $\gamma_{Y}$ be the horizontal path from $-\infty+2\pi\sqrt{-1}Y$ to
$\infty+2\pi\sqrt{-1}Y$. Then for all $y\in\mathbb{R}\setminus\mathbb{Z}$, 
we have
% and all sufficiently large $N\in\mathbb{N}$,
\begin{equation}\label{pr-3-4-a}
  \begin{split}
    &\frac{1}{2\pi\sqrt{-1}}\Bigl(-\int_{\gamma_{M+\epsilon}} F(t,\{y\};b) t^{-2}dt+
  \int_{\gamma_{N-\epsilon}} F(t,\{y\};b) t^{-2}dt\Bigr)-
\delta_{N<0<M}C(1,\{y\};b)
\\
&=\frac{1}{2\pi\sqrt{-1}}
  \sum_{\substack{n\in\mathbb{Z}\\ N\leq n+\Re b\leq M\\n+b\neq 0}}\frac{e^{2\pi\sqrt{-1}ny}}{n+b}.
\end{split}
\end{equation}
On the other hand, for $L\in\mathbb{Z}$ 
\begin{equation}\label{pr-3-4-b}
  \begin{split}
    \Bigl|\int_{\gamma_{L\pm\epsilon}} F(t,\{y\};b) t^{-2}dt\Bigr|
    &\leq
    \int_{-\infty}^\infty 
    \Bigl|\frac{e^{(x+2\pi\sqrt{-1}(L\pm\epsilon)-2\pi \sqrt{-1}b)\{y\}}}{e^{x+2\pi\sqrt{-1}(L\pm\epsilon)-2\pi \sqrt{-1}b}-1}
    \frac{1}{x+2\pi\sqrt{-1}(L\pm\epsilon)}
    \Bigr|dx
    \\
    &\leq
    \int_{-\infty}^\infty 
    \frac{e^{x\{y\}}e^{2\pi\Im b\{y\}}}{|e^{x-2\pi\sqrt{-1}(b\pm \epsilon)}-1|}
    \frac{dx}{|x+2\pi\sqrt{-1}(L\pm\epsilon)|}
    \\
    &=
    \int_{-\infty}^\infty 
%    \frac{e^{x\{y\}}}{e^{x}\cos(2\pi(b\pm\epsilon-\frac{1}{2}))+1}
    \frac{e^{x\{y\}}e^{2\pi|\Im b|}}{|e^{x-2\pi\sqrt{-1}(b\pm \epsilon)}-1|}
%    \frac{e^{x\{y\}}}{e^x+1}
    \frac{dx}{\sqrt{x^2+4\pi^2(L\pm\epsilon)^2}}.
  \end{split}
\end{equation}
It is easy to see that
there exists $K'>0$ independent of $y$ such that
\begin{gather}\label{pr-3-4-c}
  \frac{e^{x\{y\}}e^{2\pi|\Im b|}}{|e^{x-2\pi\sqrt{-1}(b\pm \epsilon)}-1|}
  \leq g(x,y):=
\begin{cases}
  K' e^{x\{y\}} \qquad&(x<0),\\
  K' e^{x(\{y\}-1)} \qquad&(x\geq0).
\end{cases}
\end{gather}
Applying \eqref{pr-3-4-c} and Lemma \ref{lm:Ei} to \eqref{pr-3-4-b},
we have 
\begin{equation}
  \begin{split}
    \Bigl|\int_{\gamma_{L\pm\epsilon}} F(t,\{y\};b) t^{-2}dt\Bigr|
    &\leq 
    \int_{-\infty}^\infty \frac{g(x,y)}{\sqrt{x^2+4\pi^2(L\pm\epsilon)^2}}dx%+K'
    \\
    &\leq K'' 
|L|^{-\frac{1}{\mu+1}}((1-\{y\})^{-\frac{1}{\mu+1}}+\{y\}^{-\frac{1}{\mu+1}})%+\delta_{N<0<M}K'
  \end{split}
\end{equation}
for some $K''>0$.
Therefore, choosing $N=-L$ and $M=L$ in \eqref{pr-3-4-a} and taking the
limit $L\to\infty$, we obtain \eqref{eq:C1}.   Moreover, 
since $C(1,\{y\};b)$ is bounded in $y$, we obtain \eqref{est:C1}.
\end{proof}

%%%%%%%%%%%%%%%%%%%%%%%%%%%%%%%%%%%%%%
{\it The second step}.
%%%%%%%%%%%%%%%%%%%%%%%%%%%%%%%%%%%%%
Secondly we consider the higher rank case of \eqref{eq:S_1} under the special condition $\Lambda=B$ with $\mathscr{B}=\{B\}$ in Lemmas \ref{lm:calcS}, \ref{lm:bddS} and \ref{lm:ZZdiff}.
We first prepare the following algebraic lemma.   This statement is 
included in \cite[Chapitre 6, Section 1, 9]{Bour}, but here we supply a proof.

\begin{lemma}
\label{lm:characteristic}
Let $Q, P$ be free $\mathbb{Z}$-modules of rank $r$ with $Q\subset P$ so that $P/Q$ is a finite abelian group.
Then 
\begin{equation}\label{characteristic:f1}
\Hom(P/Q,\mathbb{Q}/\mathbb{Z})\simeq
\Hom(Q,\mathbb{Z})/\Hom(P,\mathbb{Z})
\end{equation}
and for $\Bar{\lambda}\in P/Q$, we have
\begin{equation}\label{characteristic:f2}
  \frac{1}{\sharp(P/Q)}\sum_{\Bar{f}\in\Hom(Q,\mathbb{Z})/\Hom(P,\mathbb{Z})} e^{2\pi\sqrt{-1}\Bar{f}(\Bar{\lambda})}=\delta_{\Bar{\lambda},0},
\end{equation}
where the right-hand side denotes Kronecker's delta.
\end{lemma}
\begin{proof}
First we note that an element of $\Hom(P,\mathbb{Z})$ can be naturally
regarded as an element of $\Hom(Q,\mathbb{Z})$.   Denote this injection 
by $\iota$.
Next, let $f\in\Hom(Q,\mathbb{Z})$.  
It is well-known that
there exist a basis $\{\lambda_i\}_{i=1}^r$ of $P$
and a basis $\{\lambda_i'\}_{i=1}^r$ of $Q$ such that $\lambda_i'=k_i\lambda_i$ with $k_i\in\mathbb{N}$ and hence
\begin{equation}
  P/Q=\langle\Bar{\lambda}_1\rangle\oplus\cdots\oplus\langle\Bar{\lambda}_r\rangle,
\end{equation}
where each $\langle\Bar{\lambda}_i\rangle$ is a cyclic group of order $k_i$ with $\Bar{\lambda}_i\in P/Q$.
Define
$\varphi(f)$ by the linear extension of
\begin{equation}
  \varphi(f)(\Bar{\lambda}_i)=p(f(\lambda_i')/k_i),
\end{equation}
where $p$ denotes the natural projection $\mathbb{Q}\to\mathbb{Q}/\mathbb{Z}$.
Then $\varphi(f)\in\Hom(P/Q,\mathbb{Q}/\mathbb{Z})$ is well-defined.

We show that the sequence
\begin{equation}\label{exact}
  0\to
  \Hom(P,\mathbb{Z}) \xrightarrow{\iota}
  \Hom(Q,\mathbb{Z}) \xrightarrow{\varphi}%{\Hat{\cdot}} 
  \Hom(P/Q,\mathbb{Q}/\mathbb{Z}) \to
  0
\end{equation}
is exact.
%Then the second assertion follows from the orthogonality relations of group characters.
First show the surjectivity of $\varphi$.
Let $g\in\Hom(P/Q,\mathbb{Q}/\mathbb{Z})$. Then
choose a representative $a_i\in\mathbb{Q}$ of $g(\Bar{\lambda}_i)$.
Since $k_i g(\Bar{\lambda}_i)=g(\Bar{\lambda_i'})=g(0)=0$, we have 
$k_ia_i\in\mathbb{Z}$.
Define $f\in\Hom(Q,\mathbb{Z})$ by the linear extension of 
$f(\lambda_i')=k_ia_i$.
Then $\varphi(f)(\Bar{\lambda_i})=p(a_i)=g(\Bar{\lambda_i})$, so
$g=\varphi(f)$.   Therefore $\varphi$ is surjective.
Next, let $f\in\ker\varphi$. Then 
$\varphi(f)(\Bar{\lambda_i})=p(f(\lambda_i')/k_i)=0$ and so 
$f(k_i\lambda_i)\in k_i\mathbb{Z}$, that is, 
$f(\lambda_i)\in\mathbb{Z}$.   This implies
$f\in\Hom(P,\mathbb{Z})$, and hence the exactness at
$\Hom(Q,\mathbb{Z})$ is proved.   The assertion 
\eqref{characteristic:f1} immediately follows from \eqref{exact},
and \eqref{characteristic:f2} follows from the orthogonality relations 
of group characters (cf.\ Apostol \cite[Theorem 6.13]{Apos}).
\end{proof}

\begin{lemma}
  \label{lm:setH}
  Let $B\in\mathscr{B}$ and $f\in B$. Then
  \begin{equation}
    \mathfrak{H}_{B\setminus\{f\}}+\mathbb{Z}^r=
\{\mathbf{y}\in V~|~\langle\mathbf{y}+\mathbf{w},\vec{f}^B\rangle\in\mathbb{Z}\text{ for some }\mathbf{w}\in\mathbb{Z}^r\}.
  \end{equation}
\end{lemma}
\begin{proof}
  We denote the left-hand side by $P$ and the right-hand side by $Q$ respectively. 
If $\mathbf{y}\in P$, then $\mathbf{y}=\mathbf{y}_0+\mathbf{w}_0$ with $\langle\mathbf{y}_0,\vec{f}^B\rangle=0$  and $\mathbf{w}_0\in\mathbb{Z}^r$. 
By setting $\mathbf{w}=-\mathbf{w}_0$ we have
\begin{equation}
  \langle\mathbf{y}+\mathbf{w},\vec{f}^B\rangle=
  \langle\mathbf{y}_0,\vec{f}^B\rangle=0\in\mathbb{Z},
\end{equation}
and $\mathbf{y}\in Q$.
Conversely, if $\mathbf{y}\in Q$, then
\begin{equation}%\label{3-32}
  \mathbf{y}+\mathbf{w}\in\mathfrak{H}_{B\setminus\{f\}}+\mathbb{Z}\vec{f}\subset
  \mathfrak{H}_{B\setminus\{f\}}+\mathbb{Z}^r
\end{equation}
because $\mathfrak{H}_{B\setminus\{f\}}$ is orthogonal to $\vec{f}^B$.
Hence we have $\mathbf{y}\in P$.
\end{proof}

\begin{lemma}
\label{lm:calcS}
Assume
 $\Lambda=B$ with $\mathscr{B}=\{B\}$.
% % $\Lambda=\{f_1,\ldots,f_r\}=B$ with $\mathscr{B}=\{B\}$.
For $\mathbf{k}=(k_f)_{f\in\Lambda}\in\mathbb{N}_{0}^{r}$ and
$\mathbf{y}\in V\setminus\bigcup_{f\in\Lambda_1}(\mathfrak{H}_{\Lambda\setminus\{f\}}+\mathbb{Z}^r)$, 
the limit \eqref{eq:S_1} (with $B_0=B$) converges, and
we have
  \begin{equation}
  \label{lem4_7formula}
    \begin{split}
      &
     S_1(\mathbf{k},\mathbf{y};\Lambda;B)
%     \lim_{N\to\infty}
%      (-1)^{\sharp\Lambda_0}
%      \sum_{\substack{\mathbf{v}\in\mathbb{Z}^{r}\\
%      |v_j|\leq N\quad(1\leq j\leq r)\\
%      |\Re f(\mathbf{v})|\leq N\quad(f\in B)\\
%      f(\mathbf{v})\neq 0\quad(f\in\Lambda_+)\\
%      f(\mathbf{v})=0\quad(f\in\Lambda_0)}}
%      \sum_{\substack{\mathbf{v}\in\mathbb{Z}^{r}\\f(\mathbf{v})\neq 0\quad(f\in\Lambda)}}
%      e^{2\pi\sqrt{-1}\langle \mathbf{y},\mathbf{v}\rangle}
%      \prod_{f\in\Lambda_+}
%      \frac{1}{f(\mathbf{v})^{k_f}}
%      \frac{1}{(\langle \vec{f},\mathbf{v}\rangle+\const{f})^{k_f}}
=
      \frac{1}{\sharp(\mathbb{Z}^r/\langle\vec{B}\rangle)}
      \sum_{\mathbf{w}\in\mathbb{Z}^r/\langle\vec{B}\rangle}  
      \prod_{f\in\Lambda}
      \Bigl(
      -\frac{(2\pi\sqrt{-1})^{k_f}}{k_f!}
%      C(k_f,\{\langle\mathbf{y}+\mathbf{w},\vec{f}^B\rangle\};\const{f})
      C(k_f,\{\mathbf{y}+\mathbf{w}\}_{B,f};\const{f})
      \Bigr).
    \end{split}
  \end{equation}
\end{lemma}

\begin{proof}
%Note that $M=B$ and $N=\emptyset$ in the case.
%  We abbreviate $k_i=k_{f_i}$.
  Let $A={}^t(\vec{f})_{f\in B}$ be a regular matrix, where $\vec{f}$ are regarded as column vectors.
  Then $A^{-1}=(\vec{f}^B)_{f\in B}$.
For $\mathbf{v}\in\mathbb{Z}^r$, 
write $\mathbf{u}=(u_f)_{f\in B}=A\mathbf{v}$ so that 
$u_f=\langle\vec{f},\mathbf{v}\rangle$ and
$\mathbf{v}=A^{-1}\mathbf{u}=\sum_{f\in B} \vec{f}^Bu_f$.
%  By writing $\mathbf{u}=A\mathbf{v}$,
  We have
  \begin{equation}
    \label{eq:Cauchy3}
    \begin{split}
    Z_1(N;\mathbf{k},\mathbf{y};\Lambda;B)
      &
=
 %     \lim_{N\to\infty}
      (-1)^{\sharp\Lambda_0}
      \sum_{\substack{\mathbf{v}\in\mathbb{Z}^{r}\\
%      |v_j|\leq N\quad(1\leq j\leq r)\\
      |\Re f(\mathbf{v})|\leq N\quad(f\in B)\\
      f(\mathbf{v})\neq 0\quad(f\in\Lambda_+)\\
      f(\mathbf{v})=0\quad(f\in\Lambda_0)}}
%      \sum_{\substack{\mathbf{v}\in\mathbb{Z}^{r}\\f(\mathbf{v})\neq 0\quad(f\in\Lambda)}}
      e^{2\pi\sqrt{-1}\langle \mathbf{y},\mathbf{v}\rangle}
      \prod_{f\in\Lambda_+}
      \frac{1}{(\langle \vec{f},\mathbf{v}\rangle+\const{f})^{k_f}}
      \\
      &=
%      \lim_{N\to\infty}
      (-1)^{\sharp\Lambda_0}
      \sum_{\substack{u_f\in\mathbb{Z} \\ 
          |u_f+\Re \const{f}|\leq N \quad (f\in B) \\
          u_f+\const{f}\neq 0 \quad (f\in\Lambda_+) \\
          u_f+\const{f}=0 \quad (f\in\Lambda_0)}}
      \iota(\mathbf{u})
      e^{2\pi\sqrt{-1}\langle \mathbf{y},A^{-1}\mathbf{u}\rangle}
      \prod_{f\in\Lambda_+}
      \frac{1}{(u_f+\const{f})^{k_f}},
    \end{split}
  \end{equation}
  where
  \begin{equation}
    \iota(\mathbf{u})=
    \begin{cases}
      1\qquad&
      (A^{-1}\mathbf{u}\in\mathbb{Z}^r),\\  
      0\qquad&
      (A^{-1}\mathbf{u}\notin\mathbb{Z}^r).\\  
    \end{cases}
  \end{equation}
We prove
  \begin{equation}\label{delta_delta}
    \iota(\mathbf{u})
    =
    \frac{1}{\sharp(\mathbb{Z}^r/\langle\vec{B}\rangle)}
    \sum_{\mathbf{w}\in\mathbb{Z}^r/\langle\vec{B}\rangle}  
    e^{2\pi\sqrt{-1}\langle\mathbf{w},A^{-1}\mathbf{u}\rangle}.
  \end{equation}
In fact, using Lemma \ref{lm:characteristic}
with $Q=\mathbb{Z}^r$, $P=\langle\vec{B}^*\rangle$ and noting
$\Hom(\mathbb{Z}^r,\mathbb{Z})\simeq\mathbb{Z}^r$,
$\Hom(\langle\vec{B}^*\rangle,\mathbb{Z})\simeq
\langle\vec{B}\rangle$ and
$\sharp(\langle\vec{B}^*\rangle/\mathbb{Z}^r)=\sharp
(\mathbb{Z}^r/\langle\vec{B}\rangle)$, we have
  \begin{equation*}
    \frac{1}{\sharp(\mathbb{Z}^r/\langle\vec{B}\rangle)}
    \sum_{\mathbf{w}\in\mathbb{Z}^r/\langle\vec{B}\rangle}  
    e^{2\pi\sqrt{-1}\langle\mathbf{w},\Bar{\lambda}\rangle}
  =\delta_{\Bar{\lambda},0}
  \end{equation*}
for $\Bar{\lambda}\in\langle\vec{B}^*\rangle/\mathbb{Z}^r$.
Since $A^{-1}\mathbf{u}=\sum_{f\in B} \vec{f}^Bu_f\in\langle\vec{B}^*\rangle$,
choosing $\Bar{\lambda}=\overline{A^{-1}\mathbf{u}}\in
\langle\vec{B}^*\rangle/\mathbb{Z}^r$ in the above, we obtain
\eqref{delta_delta}.

%  See Appendix for the details. 
Using \eqref{delta_delta} we find that \eqref{eq:Cauchy3} is equal to
  \begin{equation}\label{4-32}
    \begin{split}
      &
      \frac{1}{\sharp(\mathbb{Z}^r/\langle\vec{B}\rangle)}
      (-1)^{\sharp\Lambda_0}
      \sum_{\mathbf{w}\in\mathbb{Z}^r/\langle\vec{B}\rangle}  
      \sum_{\substack{u_f\in\mathbb{Z} \\ 
          |u_f+\Re \const{f}|\leq N \quad (f\in B) \\
          u_f+\const{f}\neq 0 \quad (f\in\Lambda_+) \\
          u_f+\const{f}=0 \quad (f\in\Lambda_0)}}
        e^{2\pi\sqrt{-1}\langle \mathbf{y}+\mathbf{w},A^{-1}\mathbf{u}\rangle}
            \Bigl(
      \prod_{f\in\Lambda_+}
      \frac{1}{(u_f+\const{f})^{k_f}}
      \Bigr)
      \\
      &=
      \frac{1}{\sharp(\mathbb{Z}^r/\langle\vec{B}\rangle)}
      \sum_{\mathbf{w}\in\mathbb{Z}^r/\langle\vec{B}\rangle}  
      \prod_{f\in\Lambda_0}\Bigl(-
      \sum_{\substack{u_f\in\mathbb{Z} \\ 
          |u_f+\Re \const{f}|\leq N\\
          u_f+\const{f}=0}}
      e^{2\pi\sqrt{-1}\langle\mathbf{y}+\mathbf{w},\vec{f}^B\rangle u_f}\Bigr)
      \\
      &\qquad\times
      \prod_{f\in\Lambda_+}\Bigl(
      \sum_{\substack{u_f\in\mathbb{Z} \\ 
          |u_f+\Re \const{f}|\leq N\\
          u_f+\const{f}\neq 0}}
      \frac{e^{2\pi\sqrt{-1}\langle\mathbf{y}+\mathbf{w},\vec{f}^B\rangle u_f}}{(u_f+\const{f})^{k_f}}
      \Bigr),
    \end{split}
  \end{equation}
where we have used $A^{-1}\mathbf{u}=\sum_{f\in B} \vec{f}^Bu_f$.
(Note that at present $\Lambda=\Lambda_0\cup\Lambda_+=B$.)

By Lemma \ref{lm:setH} we see that the assumption $\mathbf{y}\in V\setminus\bigcup_{f\in\Lambda_1}(\mathfrak{H}_{\Lambda\setminus\{f\}}+\mathbb{Z}^r)$ implies
that $\langle\mathbf{y}+\mathbf{w},\vec{f}^B\rangle\notin\mathbb{Z}$ for all $f\in \Lambda_1$ and $\mathbf{w}\in\mathbb{Z}^r$. 
% Next we show that $\langle\mathbf{y}+\mathbf{w},\vec{f}^B\rangle\notin\mathbb{Z}$ for all $f\in \Lambda_1$ and $\mathbf{w}\in\mathbb{Z}^r$. 
% If
% $\langle\mathbf{y}+\mathbf{w},\vec{f}^B\rangle\in\mathbb{Z}$ for some $f\in \Lambda_1$ and $\mathbf{w}\in\mathbb{Z}^r$, then
% \begin{equation}\label{3-32}
%   \mathbf{y}+\mathbf{w}\in\mathfrak{H}_{\Lambda\setminus\{f\}}+\mathbb{Z}\vec{f}\subset
%   \mathfrak{H}_{\Lambda\setminus\{f\}}+\mathbb{Z}^r
% \end{equation}
% because $\mathfrak{H}_{\Lambda\setminus\{f\}}$ is orthogonal to $\vec{f}^B$.
% This contradicts with the assumption
% $\mathbf{y}\in V\setminus\bigcup_{f\in\Lambda_1}(\mathfrak{H}_{\Lambda\setminus\{f\}}+\mathbb{Z}^r)$.
Therefore the condition of Lemma \ref{lm:convk1} is satisfied
(for $y=\langle\mathbf{y}+\mathbf{w},\vec{f}^B\rangle$).
Therefore letting $N\to\infty$ on the right-hand side of \eqref{4-32}
and applying Lemmas \ref{lm:convkn0} and \ref{lm:convk1},
we obtain 
   \begin{equation}
     S_1(\mathbf{k},\mathbf{y};\Lambda;B)
=
      \frac{1}{\sharp(\mathbb{Z}^r/\langle\vec{B}\rangle)}
      \sum_{\mathbf{w}\in\mathbb{Z}^r/\langle\vec{B}\rangle}  
      \prod_{f\in\Lambda}
      \Bigl(
      -\frac{(2\pi\sqrt{-1})^{k_f}}{k_f!}
      C(k_f,\{\langle\mathbf{y}+\mathbf{w},\vec{f}^B\rangle\};\const{f})
      \Bigr).
  \end{equation}

Lastly we note that the factor 
$\{\langle\mathbf{y}+\mathbf{w},\vec{f}^B\rangle\}$
on the right-hand side of the above can be replaced by 
$\{\mathbf{y}+\mathbf{w}\}_{B,f}$.    This is because 
$C(k,\{y\};b)=C(k,1-\{-y\};b)$ for $k\geq2$ or $k=0$,
while $y\notin\mathbb{Z}$ when $k=1$ (cf.\ \eqref{2-6}).
This completes the proof of the lemma.
\end{proof}

Lemma \ref{lm:calcS} gives the right-hand side of \eqref{eq:S_int} in the special case $\Lambda=B$
and $\mathscr{B}=\{B\}$, but for $S_1(\mathbf{k},\mathbf{y};\Lambda;B)$ instead of
$S(\mathbf{k},\mathbf{y};\Lambda)$.    In order to use Lemma \ref{lm:calcS} in the proof of the
general case, we need the following inequality.

\begin{lemma}
\label{lm:bddS}
Assume $\Lambda=B$ with $\mathscr{B}=\{B\}$.
%Let $r\in\mathbb{N}$, $w>0$ and $\Lambda=\{f_1,\ldots,f_r\}=B$ with $\mathscr{B}=\{B\}$.
For $\mu>0$ and
$\mathbf{k}=(k_f)_{f\in\Lambda}\in\mathbb{N}_{0}^{r}$
there exists $K>0$ such that
for all 
$\mathbf{y}\in V\setminus\bigcup_{f\in\Lambda_1}(\mathfrak{H}_{\Lambda\setminus\{f\}}+\mathbb{Z}^r)$ and all sufficiently large $N>0$,
  \begin{equation}
    \begin{split}
      |Z_1(N;\mathbf{k},\mathbf{y};\Lambda;B)|
%      \Bigl|%(-1)^{\sharp\Lambda_0}
%      \sum_{\substack{\mathbf{v}\in\mathbb{Z}^{r}\\
%      |v_j|\leq N\quad(1\leq j\leq r)\\
%      |\Re f(\mathbf{v})|\leq N\quad(f\in B)\\
%      f(\mathbf{v})\neq 0\quad(f\in\Lambda_+)\\
%      f(\mathbf{v})=0\quad(f\in\Lambda_0)}}
%      \sum_{\substack{\mathbf{v}\in\mathbb{Z}^{r}\\f(\mathbf{v})\neq 0\quad(f\in\Lambda)}}
%      e^{2\pi\sqrt{-1}\langle \mathbf{y},\mathbf{v}\rangle}
%      \prod_{f\in\Lambda_+}
%      \frac{1}{f(\mathbf{v})^{k_f}}\Bigr|
%      \frac{1}{(\langle \vec{f},\mathbf{v}\rangle+\const{f})^{k_f}}\Bigr|
      \leq K \sum_{\mathbf{w}\in\mathbb{Z}^r/\langle\vec{B}\rangle}  
      \prod_{f\in\Lambda_1}(1+(1-\{\langle\mathbf{y}+\mathbf{w},\vec{f}^B\rangle\})^{-\frac{1}{\mu+1}}+\{\langle\mathbf{y}+\mathbf{w},\vec{f}^B\rangle\}^{-\frac{1}{\mu+1}}).
    \end{split}
  \end{equation}
\end{lemma}

\begin{proof}
From the proof of Lemma \ref{lm:calcS}, we have
  \begin{equation}\label{spsp}
    \begin{split}
      |Z_1(N;\mathbf{k},\mathbf{y};\Lambda;B)|
%      \Bigl|%(-1)^{\sharp\Lambda_0}
%      \sum_{\substack{\mathbf{v}\in\mathbb{Z}^{r}\\
%      |\Re f(\mathbf{v})|\leq N\quad(f\in B)\\
%      f(\mathbf{v})\neq 0\quad(f\in\Lambda_+)\\
%      f(\mathbf{v})=0\quad(f\in\Lambda_0)}}
%      e^{2\pi\sqrt{-1}\langle \mathbf{y},\mathbf{v}\rangle}
%      \prod_{f\in\Lambda_+}
%      \frac{1}{(\langle \vec{f},\mathbf{v}\rangle+\const{f})^{k_f}}
%      \Bigr|
%      \\
      &\leq
      \frac{1}{\sharp(\mathbb{Z}^r/\langle\vec{B}\rangle)}
      \sum_{\mathbf{w}\in\mathbb{Z}^r/\langle\vec{B}\rangle}  
      \prod_{f\in\Lambda_0}\Bigl|
      -\sum_{\substack{u_f\in\mathbb{Z} \\ 
          |u_f+\Re\const{f}|\leq N\\
          u_f+\const{f}=0}}
      e^{2\pi\sqrt{-1}\langle\mathbf{y}+\mathbf{w},\vec{f}^B\rangle u_f}
      \Bigr|
      \\
      &\qquad\times
      \prod_{f\in\Lambda_+}\Bigl|
      \sum_{\substack{u_f\in\mathbb{Z} \\ 
          |u_f+\Re\const{f}|\leq N\\
          u_f+\const{f}\neq 0}}
      \frac{e^{2\pi\sqrt{-1}\langle\mathbf{y}+\mathbf{w},\vec{f}^B\rangle u_f}}{(u_f+\const{f})^{k_f}}
      \Bigr|.
     % \\
     \end{split}
   \end{equation}
On the right-hand side of \eqref{spsp}, each sum corresponding to
$f\in\Lambda_0$ consists of just one term, and each sum corresponding
to $k_f\geq 2$ is convergent absolutely as $N\to\infty$, so all of them
are bounded.   For $f\in\Lambda_1$, we apply Lemma \ref{lm:convk1} to
obtain that the right-hand side of \eqref{spsp} is
\begin{equation}
      \leq K \sum_{\mathbf{w}\in\mathbb{Z}^r/\langle\vec{B}\rangle}  
      \prod_{f\in\Lambda_1}(1+(1-\{\langle\mathbf{y}+\mathbf{w},\vec{f}^B\rangle\})^{-\frac{1}{\mu+1}}+\{\langle\mathbf{y}+\mathbf{w},\vec{f}^B\rangle\}^{-\frac{1}{\mu+1}}).
    %\end{split}
\end{equation}
\end{proof}

To evaluate the difference between $Z(N;\mathbf{k},\mathbf{y};\Lambda;B)$ and
$Z_1(N;\mathbf{k},\mathbf{y};\Lambda;B)$ in the final step of the proof, the following two
lemmas are necessary.
For $B\in\mathscr{B}$ and 
$R>0$, let
\begin{align}
  % U_R&=\{\mathbf{x}\in\mathbb{R}^{r}~|~|\langle\vec{f},\mathbf{x}\rangle|\leq R\quad(f\in B)\}\\
  &U_R=U_R(B)=\{\mathbf{x}\in\mathbb{R}^{r}~|~|\Re f(\mathbf{x})|\leq R\quad(f\in B)\},\\
  &W_R=\{\mathbf{x}\in\mathbb{R}^{r}~|~|x_j|\leq R\quad(1\leq j\leq r)\}.
\end{align}
  
\begin{lemma}
\label{lm:UWU}
There exist positive numbers $c,d$ with $c>d>0$ such that
  \begin{equation}
    U_{dR}(B)\subset W_{R}\subset U_{cR}(B)
  \end{equation}  
  for all sufficiently large 
$R>0$.  
\end{lemma}
\begin{proof}
We show the first inclusion. 
Each vertex $\mathbf{x}$ of the parallelotope $U_R(B)$ satisfies
one of the following equations
\begin{equation}
  A\mathbf{x}+(\Re \const{f})_{f\in B}
=(R_f)_{f\in B} \qquad (R_f\in\{R,-R\}),
\end{equation}
where we regard $(\Re \const{f})_{f\in B}$,
$(R_f)_{f\in B}$ and $\vec{f}$ as column vectors respectively
and
$A={}^t(\vec{f})_{f\in B}$.
Hence we see that the Euclid norm $\|\mathbf{x}\|$ of each vertex $\mathbf{x}=A^{-1}(R_f-\Re\const{f})_{f\in B}$ satisfies
\begin{equation}
  \begin{split}
    \|\mathbf{x}\|
    &=\|A^{-1}\|\|(R_f-\Re\const{f})_{f\in B}\|
    \leq \|A^{-1}\|\max_{f\in B}\sqrt{r}|R\pm\Re\const{f}|
\\
    &\leq R\sqrt{r}\|A^{-1}\|\max_{f\in B}\Bigl|1\pm\frac{\Re\const{f}}{R}\Bigr|
    \leq 2\sqrt{r}\|A^{-1}\| R
%  \|\mathbf{x}\|\leq\frac{\sqrt{r}R}{\|A\|}\leq R\max_{f\in B}\frac{\sqrt{r}}{\|\vec{f}\|}
\end{split} 
\end{equation}
% \begin{equation}
%   \|\mathbf{x}\|\leq \max_{f\in B}\frac{\sqrt{r}|R\pm\const{f}|}{\|A\|}\leq R\max_{f\in B}\frac{\sqrt{r}}{2\|\vec{f}\|},
% %  \|\mathbf{x}\|\leq\frac{\sqrt{r}R}{\|A\|}\leq R\max_{f\in B}\frac{\sqrt{r}}{\|\vec{f}\|}
% \end{equation}
for all sufficiently large $R>0$,
where $\|A\|$ denotes the matrix norm of $A$.
Thus choosing
\begin{equation}
  d=(2\sqrt{r}\|A^{-1}\|)^{-1}
\end{equation}
we find that the vertex $\mathbf{x}$ of $U_{dR}(B)$ satisfies 
$\|\mathbf{x}\|\leq R$, so $U_{dR}(B)\subset W_R$.

We show the second inclusion.   Denote by $K(R)$ the ball of radius $R$
whose center is the origin.
The distance between the origin and the hyperplane $\{\mathbf{x}\in V~|~\langle\vec{f},\mathbf{x}\rangle+\Re\const{f}=\pm R\}$ 
is given by $\frac{|R\pm\Re\const{f}|}{\|\vec{f}\|}$.
Since
  \begin{equation}
 R\min_{f\in B} \frac{1}{2\|\vec{f}\|}
 \leq
  \min_{f\in B} \frac{|R\pm\Re\const{f}|}{\|\vec{f}\|}
  \end{equation}
holds for all sufficiently large $R>0$, we find that
$K(R\min(2\|\vec{f}\|)^{-1})\subset U_R(B)$.
Therefore, choosing
\begin{equation}
  c^{-1}=\min_{f\in B} \frac{1}{2\sqrt{r}\|\vec{f}\|},
\end{equation}
we obtain
\begin{equation}
W_R\subset K(\sqrt{r}R)=K(cR\min(2\|\vec{f}\|)^{-1})
\subset U_{cR}(B).
\end{equation}
\end{proof}

\begin{lemma}
\label{lm:ZZdiff}
Assume $\Lambda=B=\{f_1,\ldots,f_r\}$ with $\mathscr{B}=\{B\}$. 
Let $c,d$ be as in Lemma \ref{lm:UWU}.
For $\mu>0$ and $\mathbf{k}=(k_f)_{f\in\Lambda}\in\mathbb{N}_{0}^{r}$
%Let $r\in\mathbb{N}$, $w>0$ and $\Lambda=\{f_1,\ldots,f_r\}=B$ with $\mathscr{B}=\{B\}$. For $\mathbf{k}=(k_f)_{f\in\Lambda}\in\mathbb{N}_{0}^{r}$,
there exists $K>0$ such that for all
%$\mathbf{y}\in V$, 
$\mathbf{y}\in V\setminus\bigcup_{f\in\Lambda_1}(\mathfrak{H}_{\Lambda\setminus\{f\}}+\mathbb{Z}^r)$
and all sufficiently large $N\in\mathbb{N}$,
\begin{equation}
\label{eq:limGZ}
\begin{split}
  % \lim_{N\to\infty}
  &\Bigl|%(-1)^{\sharp\Lambda_0}
  \sum_{\substack{\mathbf{v}\in\mathbb{Z}^{r}\cap (W_N\setminus U_{dN}(B))\\
      f(\mathbf{v})\neq 0\quad(f\in\Lambda_+)\\
      f(\mathbf{v})=0\quad(f\in\Lambda_0)}}
  e^{2\pi\sqrt{-1}\langle \mathbf{y},\mathbf{v}\rangle}
  \prod_{f\in\Lambda_+}
  \frac{1}{f(\mathbf{v})^{k_f}}
%  G(\mathbf{y},\mathbf{v})
  \Bigr|
  \\
& \qquad \leq
  K N^{-\frac{1}{\mu+1}}(\log N)^{r}\Bigl(1+
  \sum_{\mathbf{w}\in\mathbb{Z}^r/\langle\vec{B}\rangle}  
  \sum_{f\in\Lambda_1}
  \Bigl((1-\{\langle\mathbf{y}+\mathbf{w},\vec{f}^B\rangle\})^{-\frac{1}{\mu+1}}+\{
  \langle\mathbf{y}+\mathbf{w},\vec{f}^B\rangle\}^{-\frac{1}{\mu+1}}\Bigr)\Bigr).
\end{split}
\end{equation}
% \begin{equation}
%   \lim_{N\to\infty}
%   (-1)^{\sharp\Lambda_0}
%   \sum_{\substack{\mathbf{v}\in\mathbb{Z}^{r}\\
%       |f(\mathbf{v})|\leq N\quad(f\in B)\\
%       f(\mathbf{v})\neq 0\quad(f\in\Lambda_+)\\
%       f(\mathbf{v})=0\quad(f\in\Lambda_0)}}
%   G(\mathbf{y},\mathbf{v})
%   =
%   \lim_{N\to\infty}
%   (-1)^{\sharp\Lambda_0}
%   \sum_{\substack{\mathbf{v}\in\mathbb{Z}^{r}\\
%       |v_j|\leq N\quad(1\leq j\leq r)\\
%       f(\mathbf{v})\neq 0\quad(f\in\Lambda_+)\\
%       f(\mathbf{v})=0\quad(f\in\Lambda_0)}}
%   G(\mathbf{y},\mathbf{v}),
% \end{equation}
%where
\end{lemma}

\begin{proof}
For brevity, we put
\begin{equation}
  G(\mathbf{y},\mathbf{v})=
  e^{2\pi\sqrt{-1}\langle \mathbf{y},\mathbf{v}\rangle}
  \prod_{f\in\Lambda_+}
  \frac{1}{(\langle \vec{f},\mathbf{v}\rangle+\const{f})^{k_f}}.
\end{equation}
% By Lemma \ref{lm:UWU}, it is sufficient to show
% \begin{equation}
% \label{eq:limGZ}
%   \lim_{N\to\infty}
%   \Bigl|(-1)^{\sharp\Lambda_0}
%   \sum_{\substack{\mathbf{v}\in\mathbb{Z}^{r}\cap (W_N\setminus U_{dN})\\
%       f(\mathbf{v})\neq 0\quad(f\in\Lambda_+)\\
%       f(\mathbf{v})=0\quad(f\in\Lambda_0)}}
%   G(\mathbf{y},\mathbf{v})
%   \Bigr|=0.
% \end{equation}
We rearrange $\{f_1,\ldots,f_l\}=\Lambda_+$ and $\{f_{l+1},\ldots,f_r\}=\Lambda_0$ and
decompose 
\begin{equation}
  \Biggl\{\mathbf{v}\in\mathbb{Z}^{r}\cap(W_N\setminus U_{dN}(B))~\Biggm|~
  \begin{aligned}
    &f(\mathbf{v})\neq 0\quad(f\in\Lambda_+),\\
    &f(\mathbf{v})=0 \quad(f\in\Lambda_0)
  \end{aligned}
  \Biggr\}=\bigcup_{j=1}^l X_j(N)
\end{equation}
with
\begin{equation}
  X_j(N)=\left\{\mathbf{v}\in\mathbb{Z}^r~\middle|~
  \begin{aligned}
    &|\Re f_1(\mathbf{v})|,\ldots,|\Re f_{j-1}(\mathbf{v})|\leq dN,|\Re f_{j}(\mathbf{v})|>dN,\\
    &f_i(\mathbf{v})\neq 0\quad(1\leq i\leq l),\;
    f_i(\mathbf{v})=0\quad(l+1\leq i\leq r),\\
    & |v_k|\leq N \qquad(1\leq k\leq r)
  \end{aligned}
  \right\}
\end{equation}
for $1\leq j\leq l$.
Let
$A={}^t(\vec{f})_{f\in B}$ with $\vec{f}$ regarded as column vectors.
We rewrite the series in terms of $\mathbf{u}=A\mathbf{v}$. Let
\begin{equation}
  Y_j(N)=\left\{\mathbf{u}\in\mathbb{Z}^r~\middle|~
  \begin{aligned}
    &|u_1+\Re\const{f_1}|,\ldots,|u_{j-1}+\Re\const{f_{j-1}}|\leq dN,|u_j+\Re\const{f_j}|>dN,\\
    &u_i+\const{f_i}\neq 0\quad(1\leq i\leq l),\;
    u_i+\const{f_i}=0 \quad(l+1\leq i\leq r), \\
    &
  |\text{each entry of $A^{-1}\mathbf{u}$}|\leq N
  \end{aligned}
  \right\}
\end{equation}
for $1\leq j\leq l$. 
In $Y_j(N)$ for a fixed $(u_1,\ldots,u_{j-1},u_{j+1},\ldots,u_r)$ with sufficiently large $N\in\mathbb{N}$, 
we see that $u_j$ runs over all integers such that 
$dN-\Re\const{f_j}<u_j<H_j^+$ 
and
$-H_j^-<u_j<-dN-\Re\const{f_j}$
% $dN<u_j+\Re\const{f_j}<H_j^+$ 
% and
% $-H_j^-<u_j+\Re\const{f_j}<-dN$
for some $H_j^\pm=H_j^\pm(u_1,\ldots,u_{j-1},u_{j+1},\ldots,u_r)\geq 0$,
where $H_j^\pm$ are determined by the intersection point of the half line
 $\{(u_1,\ldots,u_{j-1},x,u_{j+1},\ldots,u_r)~|~\pm x>0\}$
and the boundaries 
$\partial W_N=\bigcup_{1\leq l\leq r}\{A\mathbf{v}~|~|v_k|\leq N \quad (k\neq l), v_l=\pm N\}$ of $W_N$.
If there is no intersection point, then we put $H^\pm_j=0$ accordingly.
See Figure \ref{fig:UWU} for these sets and parameters.
\begin{figure}[h]
  \centering
  \includegraphics{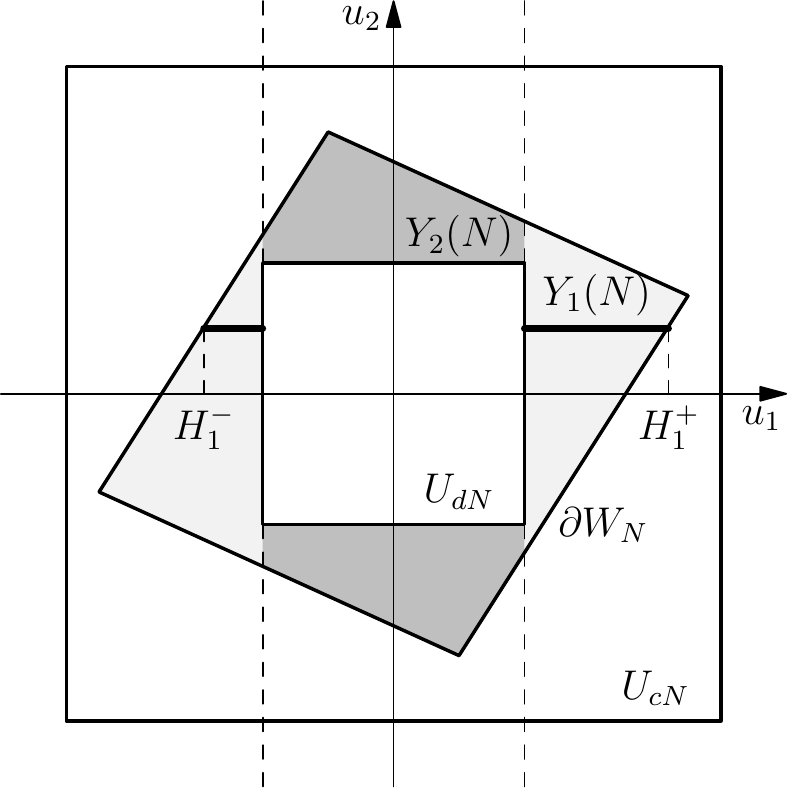}
  \caption{}
  \label{fig:UWU}
\end{figure}

From the proof of Lemma \ref{lm:calcS} %\eqref{eq:Cauchy2}
we evaluate
\begin{equation}
\label{4-51}
  \begin{split}
    \Bigl|\sum_{\mathbf{v}\in X_j(N)}
    G(\mathbf{y},\mathbf{v})
    \Bigr|
    &\leq
    \frac{1}{\sharp(\mathbb{Z}^r/\langle\vec{B}\rangle)}
    % \Bigl|
    % \sum_{\substack{u_f\in\mathbb{Z} \\ 
    %     u_f+\const{f}=0 \quad (f\in\Lambda_0)}}
    % \Bigl(-\prod_{f\in\Lambda_0}
    % e^{2\pi\sqrt{-1}\langle\mathbf{y}+\mathbf{w},\vec{f}^B\rangle u_f}\Bigr)
    % \Bigr|
    % \\
    % &\qquad\times
    \sum_{\mathbf{w}\in\mathbb{Z}^r/\langle\vec{B}\rangle}  
    \Bigl|\sum_{\mathbf{u}\in Y_j(N)}
    \Bigl(
    \prod_{i=1}^l
    \frac{e^{2\pi\sqrt{-1}\langle\mathbf{y}+\mathbf{w},\vec{f}_i^B\rangle u_i}}{(u_i+\const{f_i})^{k_i}}
    \Bigr)
    \Bigr|.
  \end{split}
\end{equation}
Further for each $j$ and $\mathbf{w}$, we have
\begin{equation}
  \begin{split}
&    \Bigl|\sum_{\mathbf{u}\in Y_j(N)}
    \Bigl(
    \prod_{i=1}^l
    \frac{e^{2\pi\sqrt{-1}\langle\mathbf{y}+\mathbf{w},\vec{f}_i^B\rangle u_i}}{(u_i+\const{f_i})^{k_i}}
    \Bigr)
    \Bigr|
    \\
    &\qquad =\Bigl|
    % \sum_{\substack{u_i\in\mathbb{Z}\quad(1\leq i\leq r) \\ 
    %     0<|u_i+\const{f_i}|\leq dN \quad (1\leq i\leq j-1)\\
    %     |u_j+\const{f_j}|>dN \\
    %     \mathbf{v}((u_i)_{1\leq i\leq l})\in X_j(N)
    %   }}
    \sum_{\substack{u_i\in\mathbb{Z},\;
        u_i+\const{f_i}\neq 0\quad(1\leq i\leq l) \\ 
        |u_i+\Re\const{f_i}|\leq dN \quad (1\leq i\leq j-1)\\
        dN-\Re\const{f_j}<u_j<H_j^+\text{ or} \\
        -H_j^-<u_j<-dN-\Re\const{f_j}
      }}
    \Bigl(
    \prod_{i=1}^l
    \frac{e^{2\pi\sqrt{-1}\langle\mathbf{y}+\mathbf{w},\vec{f}_i^B\rangle u_i}}{(u_i+\const{f_i})^{k_i}}
    \Bigr)
    \Bigr|
    \\
    &\qquad =
    \Bigl|
    \sum_{\substack{u_i\in\mathbb{Z},\;u_i+\const{f_i}\neq 0
    \quad(1\leq i\leq l,i\neq j) \\ 
        |u_i+\Re\const{f_i}|\leq dN \quad (1\leq i\leq j-1)
      }}
    \Bigl(
    \prod_{\substack{i=1\\i\neq j}}^l
    \frac{e^{2\pi\sqrt{-1}\langle\mathbf{y}+\mathbf{w},\vec{f}_i^B\rangle u_i}}{(u_i+\const{f_i})^{k_i}}
    \Bigr)
    \sum_{\substack{
        dN-\Re\const{f_j}<u_j<H_j^+ \text{ or}\\
        -H_j^-<u_j<-dN-\Re\const{f_j}
      }}
    \frac{e^{2\pi\sqrt{-1}\langle\mathbf{y}+\mathbf{w},\vec{f}_j^B\rangle u_j}}{(u_j+\const{f_j})^{k_j}}
    \Bigr|
    \\
    &\qquad\leq
    \sum_{\substack{u_i\in\mathbb{Z},\;u_i+\const{f_i}\neq 0
    \quad(1\leq i\leq l,i\neq j) \\ 
        |u_i+\Re\const{f_i}|\leq dN \quad (1\leq i\leq j-1)\\
        |u_i+\Re\const{f_i}|\leq cN \quad (j+1\leq i\leq l)
      }}
    \Bigl(
    \prod_{\substack{i=1\\i\neq j}}^l
    \frac{1}{|u_i+\const{f_i}|^{k_i}}
    \Bigr)
    \Bigl|
    \sum_{\substack{
        dN-\Re\const{f_j}<u_j<H_j^+ \text{ or}\\
        -H_j^-<u_j<-dN-\Re\const{f_j}
      }}
    \frac{e^{2\pi\sqrt{-1}\langle\mathbf{y}+\mathbf{w},\vec{f}_j^B\rangle u_j}}{(u_j+\const{f_j})^{k_j}}
    \Bigr|,
  \end{split}
\end{equation}
where in the last member, we added the extra conditions
$|u_i+\Re\const{f_i}|\leq cN$ for $j+1\leq i\leq l$, which comes from Lemma \ref{lm:UWU}. %, $\mathbf{u}\in Y_j(N)$ satisfies $|u_1+\const{f_1}|,\ldots,|u_r+\const{f_r}|\leq cN$.
If $k_j=1$, then
by Lemma \ref{lm:convk1}, and if $k_j\geq2$, then directly we obtain
\begin{equation}
  \begin{split}
    &\Bigl|
    \sum_{\substack{
        dN-\Re\const{f_j}<u_j<H_j^+ \\
        -H_j^-<u_j<-dN-\Re\const{f_j}
      }}
    \frac{e^{2\pi\sqrt{-1}\langle\mathbf{y}+\mathbf{w},\vec{f}_j^B\rangle u_j}}{(u_j+\const{f_j})^{k_j}}
    \Bigr|
    \\
    &
    \leq
    \begin{cases}
      K N^{-\frac{1}{\mu+1}}((1-\{\langle\mathbf{y}+\mathbf{w},\vec{f}_j^B\rangle\})^{-\frac{1}{\mu+1}}+\{
      \langle\mathbf{y}+\mathbf{w},\vec{f}_j^B\rangle\}^{-\frac{1}{\mu+1}})
      \qquad&(k_j=1) \\
      KN^{-1}\leq KN^{-\frac{1}{\mu+1}}\qquad&(k_j\geq2) 
    \end{cases}
    \\
    &=:Q_j(N,\mathbf{y},\mathbf{w})
  \end{split}
\end{equation}
for some $K>0$.
Hence 
\begin{equation}
\label{4-54}
  \begin{split}
    \Bigl|\sum_{\mathbf{u}\in Y_j(N)}
    \Bigl(
    \prod_{i=1}^l
    \frac{e^{2\pi\sqrt{-1}\langle\mathbf{y}+\mathbf{w},\vec{f}_i^B\rangle u_i}}{(u_i+\const{f_i})^{k_i}}
    \Bigr)
    \Bigr|
    &\leq
    Q_j(N,\mathbf{y},\mathbf{w})
    \sum_{\substack{u_i\in\mathbb{Z},\;u_i+\const{f_i}\neq 0
    \quad(1\leq i\leq l,i\neq j) \\ 
        |u_i+\Re\const{f_i}|\leq dN \quad (1\leq i\leq j-1) \\
        |u_i+\Re\const{f_i}|\leq cN \quad (j+1\leq i\leq l) 
      }}
    \Bigl(
    \prod_{\substack{i=1\\i\neq j}}^l
    \frac{1}{|u_i+\const{f_i}|^{k_i}}
    \Bigr)
    \\
    &\leq K' Q_j(N,\mathbf{y},\mathbf{w}) (\log dN)^{j-1}(\log cN)^{l-j}
    \\
    &\leq K''
    Q_j(N,\mathbf{y},\mathbf{w}) (\log N)^{r}
  \end{split}
\end{equation}
for some $K',K''>0$.
Substituting \eqref{4-54} into \eqref{4-51}, we complete the proof.
%\eqref{eq:limGZ}.
% by Lemma \ref{lm:bddk1}.
%By Lemma \ref{lm:UWU}, we have $H_j^{-1}\leq$
\end{proof}

%%%%%%%%%%%%%%%%%%%%%%%%%%%%%%
{\it The third step.}
%%%%%%%%%%%%%%%%%%%%%%%%%%%%%%
Lastly we consider the general case.    First we prove several preparatory lemmas.
% For $\mathbf{y}\in V$, a decomposition $\Lambda=B_0\cup L_0$ with $B_0\in\mathscr{B}$ and $f\in B_0$,
% let 
% \begin{equation}
%   H(f,\mathbf{y})=
% \{\mathbf{x}=(x_g)_{g\in L_0}\in\mathbb{R}^{\sharp L_0}~|~\langle\mathbf{y}-\sum_{g\in L_0}x_g\vec{g},\vec{f}^{B_0}\rangle\in\mathbb{Z}\}.
% \end{equation}

\begin{lemma}
\label{lm:locfin}
Fix a decomposition $\Lambda=B_0\cup L_0$ with $B_0\in\mathscr{B}$. 
Let $\mathbf{y}\in V$ and $f\in B_0$.
If $f\notin \widetilde{\Lambda}$, then
there exists $g\in L_0$ such that
$\langle\vec{g},\vec{f}^{B_0}\rangle\neq 0$.
If $f\in \widetilde{\Lambda}$ and
$\mathbf{y}\notin\mathfrak{H}_{B_0\setminus\{f\}}+\mathbb{Z}^r$,
then there exists $c\in\mathbb{R}\setminus\mathbb{Z}$ such that
\begin{equation}
\langle\mathbf{y}-\sum_{g\in L_0}x_g\vec{g},\vec{f}^{B_0}\rangle
=c
\end{equation}
for all $\mathbf{x}=(x_g)_{g\in L_0}\in\mathbb{R}^{\sharp L_0}$.
\end{lemma}
\begin{proof}
The first assertion directly follows from the definition.
Assume that
$f\in\widetilde{\Lambda}$. Then $\langle\vec{g},\vec{f}^{B_0}\rangle=0$
for all $g\in L_0$ and we have
\begin{equation}
\langle\mathbf{y}-\sum_{g\in L_0}x_g\vec{g},\vec{f}^{B_0}\rangle
=
  \langle\mathbf{y},\vec{f}^{B_0}\rangle, %=n.
\end{equation}
which is a constant function in $\mathbf{x}$.
By Lemma \ref{lm:setH},
we find
$\langle\mathbf{y},\vec{f}^{B_0}\rangle\notin\mathbb{Z}$.
This implies
the second assertion.
% If
% $\langle\mathbf{y},\vec{f}^{B_0}\rangle\in\mathbb{Z}$,
% then similarly to \eqref{3-32} we find that 
% $\mathbf{y}\in\mathfrak{H}_{B_0\setminus\{f\}}+\mathbb{Z}\vec{f}\subset\mathfrak{H}_{B_0\setminus\{f\}}+\mathbb{Z}^r$.
% This implies
% the second assertion.
\end{proof}

For $\mathbf{y}\in V$, a decomposition $\Lambda=B_0\cup L_0$ with $B_0\in\mathscr{B}$ and $f\in B_0$,
let 
\begin{equation}
  H(f,\mathbf{y})=
\Bigl\{\mathbf{x}=(x_g)_{g\in L_0}\in\mathbb{R}^{\sharp L_0}~\Bigm|~\langle\mathbf{y}-\sum_{g\in L_0}x_g\vec{g},\vec{f}^{B_0}\rangle\in\mathbb{Z}\Bigr\}.
\end{equation}

\begin{lemma}
\label{lm:locfin2}
Let $\mathbf{y}\in V$, $f\in B_0$, and assume that 
$\mathbf{y}\notin\mathfrak{H}_{B_0\setminus\{f\}}+\mathbb{Z}^r$ 
if $f\in \widetilde{\Lambda}$. 
Then the set $H(f,\mathbf{y})$ is empty, or a collection of equally spaced parallel hyperplanes.
\end{lemma}
\begin{proof}
Let $U=(\vec{g})_{g\in L_0}$ be an $r\times\sharp L_0$ matrix and
$\mathbf{x}=(x_g)_{g\in L_0}$ be a column vector.
Consider the equation 
%for $\mathbf{x}\in\mathbb{R}^{\sharp L_0}$
\begin{equation}
\label{eq:Ux}
  \langle\mathbf{y}-\sum_{g\in L_0}x_g\vec{g},\vec{f}^{B_0}\rangle=n
\end{equation}
for $n\in\mathbb{Z}$,
which is equivalent to
\begin{equation}
\label{eq:Ux_bis}
  \langle U\mathbf{x},\vec{f}^{B_0}\rangle
=\langle \mathbf{y},\vec{f}^{B_0}\rangle-n.
%=\langle \mathbf{y}-n\vec{f},\vec{f}^{B_0}\rangle.
\end{equation}
%By Lemma \ref{lm:nondeg},
%By the assumption of the statement, for each $f\in B_0$,
%The first assertion directly follows from the definition.

Assume that $f\notin\widetilde{\Lambda}$. Then
there exists $g\in L_0$ such that
%$\vec{g}\notin\mathfrak{H}_{B_0\setminus\{f\}}$, namely,
$\langle\vec{g},\vec{f}^{B_0}\rangle\neq 0$.
We see that \eqref{eq:Ux_bis}
has a solution $\mathbf{x}=\mathbf{x}_0-n\mathbf{a}$ with
\begin{equation}
  (\mathbf{x}_0)_{h}=
  \begin{cases}
    \langle \mathbf{y},\vec{f}^{B_0}\rangle
    /\langle\vec{g},\vec{f}^{B_0}\rangle\qquad&(h=g)\\
    0\qquad&(h \neq g)\\
  \end{cases},
\qquad
  (\mathbf{a})_{h}=
  \begin{cases}
    1/\langle\vec{g},\vec{f}^{B_0}\rangle\qquad&(h=g)\\
    0\qquad&(h \neq g)\\
  \end{cases},
\end{equation}
and so the equation \eqref{eq:Ux_bis} is rewritten as
\begin{equation}
  \langle U(\mathbf{x}-\mathbf{x}_0+n\mathbf{a}),\vec{f}^{B_0}\rangle=0.
\end{equation}
The condition 
$\langle\vec{g},\vec{f}^{B_0}\rangle\neq 0$
also implies that $\dim\ker \mathscr{U}=\sharp L_0-1$ for the linear functional 
$\mathscr{U}(\mathbf{v})=\langle U\mathbf{v},\vec{f}^{B_0}\rangle$, and
\begin{equation}
H(f,\mathbf{y})
%\{\mathbf{x}=(x_g)_{g\in L_0}\in\mathbb{R}^{\sharp L_0}~|~\langle\mathbf{y}-\sum_{g\in L_0}x_g\vec{g},\vec{f}^{B_0}\rangle\in\mathbb{Z}\}
=(\ker\mathscr{U}+\mathbf{x}_0)+\mathbb{Z}\mathbf{a}
%=\bigcup_{n\in\mathbb{Z}}(\ker h+\mathbf{x}_0-n\mathbf{a})
\end{equation}
is a collection of equally spaced parallel hyperplanes.

Assume that
$f\in\widetilde{\Lambda}$. Then by Lemma \ref{lm:locfin},
%$\langle\vec{g},\vec{f}^{B_0}\rangle=0$
%for all $g\in L_0$ and we have
%\eqref{eq:Ux} reduces to
\begin{equation}
\langle\mathbf{y}-\sum_{g\in L_0}x_g\vec{g},\vec{f}^{B_0}\rangle
\in\mathbb{R}\setminus\mathbb{Z},
%  \langle\mathbf{y},\vec{f}^{B_0}\rangle, %=n.
\end{equation}
and hence
%which is a constant function in $\mathbf{x}$.
%If
%\begin{equation}
%  \langle\mathbf{y},\vec{f}^{B_0}\rangle\in\mathbb{Z},
%\end{equation}
%then $\mathbf{y}\in\mathfrak{H}_{B_0\setminus\{f\}}+\mathbb{Z}%\vec{f}\subset\mathfrak{H}_{B_0\setminus\{f\}}+\mathbb{Z}^r$, 
%which contradicts. 
%Therefore 
$H(f,\mathbf{y})=\emptyset$.
\end{proof}

\begin{lemma}
  \label{lm:aex}
  Fix a decomposition $\Lambda=B_0\cup L_0$ with $B_0\in\mathscr{B}$. 
  Assume $D\subset B_0$.
%  For a fixed $\mathbf{y}\in V\setminus\bigcup_{f\in\widetilde{\Lambda}\cap \Lambda_1}(\mathfrak{H}_{\Lambda\setminus\{f\}}+\mathbb{Z}^r)$, the measure of
  For a fixed $\mathbf{y}\in V\setminus\bigcup_{f\in\widetilde{\Lambda}\cap D}(\mathfrak{H}_{\Lambda\setminus\{f\}}+\mathbb{Z}^r)$, the measure of
  \begin{equation}
    \label{eq:setX}
    M(\mathbf{y})=
    \Bigl\{(x_g)_{g\in L_0}\in \mathbb{R}^{\sharp L_0}~\Bigm|~\mathbf{y}-\sum_{g\in L_0}x_g\vec{g}\in \bigcup_{f\in D}(\mathfrak{H}_{B_0\setminus\{f\}}+\mathbb{Z}^r)\Bigr\}
  \end{equation}
  is zero.
\end{lemma}
\begin{proof}
By Lemma \ref{lm:setH}, we have $M(\mathbf{y})=\bigcup_{f\in D}\bigcup_{\mathbf{w}\in\mathbb{Z}^r} H(f,\mathbf{y}+\mathbf{w})$. 
% Further
% by \eqref{eq:LsB} we have
% \begin{equation}
%   \bigcup_{f\in \widetilde{\Lambda}\cap\Lambda_1}(\mathfrak{H}_{\Lambda\setminus\{f\}}+\mathbb{Z}^r)
% =
% \bigcup_{f\in \widetilde{\Lambda}\cap\Lambda_1\cap B_0}(\mathfrak{H}_{B_0\setminus\{f\}}+\mathbb{Z}^r),
% \end{equation}
% so from the assumption
Further
by \eqref{eq:LsB} we have
\begin{equation}
  \bigcup_{f\in \widetilde{\Lambda}\cap D}(\mathfrak{H}_{\Lambda\setminus\{f\}}+\mathbb{Z}^r)
=
\bigcup_{f\in \widetilde{\Lambda}\cap D}(\mathfrak{H}_{B_0\setminus\{f\}}+\mathbb{Z}^r),
\end{equation}
so from the assumption
% From the assumption
$\mathbf{y}\in V\setminus\bigcup_{f\in \widetilde{\Lambda}\cap D}(\mathfrak{H}_{\Lambda\setminus\{f\}}+\mathbb{Z}^r)$
we see that
$\mathbf{y}+\mathbf{w}\notin\bigcup_{f\in \widetilde{\Lambda}\cap D}(\mathfrak{H}_{B_0\setminus\{f\}}+\mathbb{Z}^r)$ for any $\mathbf{w}\in\mathbb{Z}^r$.
Therefore we can apply 
Lemma \ref{lm:locfin2} to find that
for each $f\in D$ and $\mathbf{w}\in\mathbb{Z}^r$
the measure of $H(f,\mathbf{y}+\mathbf{w})$ is zero.
\end{proof}

\begin{lemma}
  \label{lm:integrable}
Let $n\in\mathbb{N}$ and $P,Q\in\mathbb{N}_0$ with $P\geq Q$. Let $a_{ki}\in\mathbb{R}$ for $1\leq k\leq P$ and $0\leq i\leq n$ such that for each $k=1,\ldots,P$ there exists $i\geq 1$ such that $a_{ki}\neq 0$.
 If $\mu\geq P$, then
  \begin{equation}
    \int_0^1dx_1\cdots
    \int_0^1dx_n
    \Bigl(\prod_{k=1}^Q(1-\{a_{k0}+\sum_{i=1}^n a_{ki}x_i\})^{-\frac{1}{\mu+1}}\Bigr)
    \Bigl(\prod_{k=Q+1}^P\{a_{k0}+\sum_{i=1}^n a_{ki}x_i\}^{-\frac{1}{\mu+1}}\Bigr)<\infty.
  \end{equation}
\end{lemma}
\begin{proof}
%Let  
%  \begin{equation}
%    H=\bigcup_{1\leq k\leq P}\{\mathbf{x}=(x_1,\ldots,x_n)\in\mathbb{R}^n~|~a_{k0}+\sum_{i=1}^n a_{ki}x_i\in\mathbb{Z}\}
%  \end{equation}
%be a union of hyperplanes.
For $1\leq k\leq P$
put
\begin{equation}
  L_k(\mathbf{x})=\sum_{i=1}^n a_{ki}x_i.
\end{equation}
Since $[0,1]^n$ is compact,
by considering a neighborhood of each point $\mathbf{x}_0$ in $[0,1]^n$ and shifting the point $\mathbf{x}_0$ to the origin, we see that
it is sufficient to show that 
\begin{equation}
  I=\int_{[-\epsilon,\epsilon]^n}dx_1\cdots dx_n
  \Bigl(\prod_{k=1}^q(1-\{L_k(\mathbf{x})\})^{-\frac{1}{\mu+1}}\Bigr)
  \Bigl(\prod_{k=q+1}^p\{L_k(\mathbf{x})\}^{-\frac{1}{\mu+1}}\Bigr)
\end{equation}
is finite for a sufficiently small $\epsilon>0$,
where $0\leq q\leq Q$ and $q\leq p\leq P$.
This is estimated as
\begin{equation}
\label{eq:I_rhs}
    I\leq
%     \int_{\|\mathbf{x}\|\leq\epsilon}dx_1\cdots dx_n
%     \Bigl(\prod_{k=1}^p|\sum_{i=1}^n a_{ki}x_i|^{-\frac{1}{\mu+1}}\Bigr)\\
% &=
    \int_{[-\epsilon,\epsilon]^n}dx_1\cdots dx_n
    \Bigl(\prod_{k=1}^p|L_k(\mathbf{x})|^{-\frac{1}{\mu+1}}\Bigr).
\end{equation}
For this integral, 
we decompose the region 
\begin{equation}
[-\epsilon,\epsilon]^n=\bigcup_{k=1}^p U_k,
\end{equation} 
where
\begin{gather}
U_k=\{\mathbf{x}\in[-\epsilon,\epsilon]^n~|~|L_k(\mathbf{x})|\leq |L_m(\mathbf{x})|\text{ for any }m\neq k\}.
\end{gather}
We show that the integral on each $U_k$ is finite.
Since on $U_k$
\begin{equation}
|L_m(\mathbf{x})|^{-\frac{1}{\mu+1}} \leq |L_k(\mathbf{x})|^{-\frac{1}{\mu+1}}
\end{equation}
for any $m\neq k$, we have
\begin{equation}
\Bigl(\prod_{m=1}^p|L_m(\mathbf{x})|^{-\frac{1}{\mu+1}}\Bigr)
\leq
|L_k(\mathbf{x})|^{-\frac{p}{\mu+1}}.
\end{equation}
Fix $i$ such that $a_{ki}\neq 0$.
Then
by changing variables as $y_i=L_k(\mathbf{x})$ and $y_j=x_j$ for $j\neq i$, we obtain
\begin{equation}
  \begin{split}
    \int_{U_k}dx_1\cdots dx_n
    \Bigl(\prod_{m=1}^p|L_m(\mathbf{x})|^{-\frac{1}{\mu+1}}\Bigr)
    &\leq
    \int_{U_k}dx_1\cdots dx_n
|L_k(\mathbf{x})|^{-\frac{p}{\mu+1}}
    \\
    &\leq
    |a_{ki}|
    \int_{-r}^r |y_i|^{-\frac{p}{\mu+1}}dy_i\int_{[-\epsilon,\epsilon]^{n-1}}\prod_{j\neq i}dy_j
  \end{split}
% \Bigl(\prod_{k=1}^p|\sum_{i=1}^n a_{ki}x_i|^{-\frac{1}{w+1}}\Bigr)
\end{equation}
for some $r>0$. % and $0\leq p_n\leq \cdots \leq p_2\leq p_1 \leq p\leq P$.
By the assumption $\mu\geq P$, the right-hand side is finite
because
\begin{equation}
  -\frac{p}{\mu+1}>
  -\frac{P}{P+1}>-1.
\end{equation}
\end{proof}

\begin{lemma}
\label{lm:intC}
For $k\in \mathbb{N}_0$, $m\in\mathbb{Z}$ and $b\in\mathbb{C}$,
\begin{equation}
  -\frac{(2\pi\sqrt{-1})^k}{k!}
  \int_0^1 C(k,x;b)e^{-2\pi\sqrt{-1}mx}dx=
  \begin{cases}
    -1\qquad&(m+b=0,k=0),\\
    0\qquad&(m+b\neq 0, k=0),\\
    0\qquad&(m+b=0,k\neq 0),\\
    \dfrac{1}{(m+b)^k}\qquad&(m+b\neq 0, k\neq 0).
  \end{cases}
\end{equation}  
\end{lemma}
\begin{proof}
  By definition \eqref{F_def_exp}, for $0\leq x<1$, we have
\begin{equation}
\label{4-80}
\frac{1}{k!}C(k,x;b)e^{-2\pi\sqrt{-1}mx}
=
\frac{1}{2\pi\sqrt{-1}}
\int_{|t|=\epsilon}
  \frac{t e^{(t-2\pi \sqrt{-1}(m+b))x}}{e^{t-2\pi \sqrt{-1}b}-1}
t^{-k-1}dt
\end{equation}
for sufficiently small $\epsilon>0$.
By integrating the both sides in the region $0\leq x<1$, we obtain
\begin{equation}
    \frac{1}{k!}\int_0^1 C(k,x;b)e^{-2\pi\sqrt{-1}mx}dx 
    =
    \frac{1}{2\pi\sqrt{-1}}
    \int_{|t|=\epsilon}
    \dfrac{t}{t-2\pi \sqrt{-1}(m+b)}t^{-k-1}dt.
\end{equation}
Since
\begin{equation}
    \dfrac{t}{t-2\pi \sqrt{-1}(m+b)}
    =
    \begin{cases}
      1\qquad&(m+b=0),\\
      -\displaystyle\sum_{l=1}^\infty \frac{1}{(2\pi\sqrt{-1}(m+b))^l}t^l\qquad&(m+b\neq0),
    \end{cases}
\end{equation}
we obtain the assertion.
\end{proof}

%%%%%%%%%%%%%%%%%%%%%%%%%%%%%%%%%%%%%
\begin{proof}[Proof of Proposition \ref{prop:main}]
%%%%%%%%%%%%%%%%%%%%%%%%%%%%%%%%%%%%%%
%Lastly we show the convergence of the series in Definition \ref{def:S}.
%We fix a decomposition $\Lambda=B_0\cup L_0$ with $B_0=\{f_1,\ldots,f_r\}\in\mathscr{B}$. % and $k_f\geq 2$ for all $f\in B_0$.
%Note that $M\supset B_0$ and $N\subset L_0$ in the case. 
Applying Lemma \ref{lm:intC} with $m=\langle \vec{g},\mathbf{v}\rangle$ and $b=\const{g}$
(for $g\in L_0$)
to \eqref{Z_1_def}, for $N>0$ we have
\begin{equation}
\label{4-71}
  \begin{split}
%    \label{eq:defZ1}
    % S(\mathbf{k},\mathbf{y};\Lambda)
    % &=
    %\lim_{N\to\infty}
    Z_1(N;\mathbf{k},\mathbf{y};\Lambda;B_0)&=
%    (-1)^{\sharp\Lambda_0}
%    \sum_{\substack{\mathbf{v}\in\mathbb{Z}^{r}\\
%      |\Re f(\mathbf{v})|\leq N\quad(f\in B_0)\\
%        f(\mathbf{v})\neq 0\quad(f\in\Lambda_+)\\
%        f(\mathbf{v})=0\quad(f\in\Lambda_0)}}
%    e^{2\pi\sqrt{-1}\langle \mathbf{y},\mathbf{v}\rangle}
%    \prod_{f\in\Lambda_+}
%    \frac{1}{f(\mathbf{v})^{k_f}}
%    \\
    % &=\lim_{N\to\infty}(-1)^{\sharp\Lambda_0}
    % \sum_{\substack{\mathbf{v}\in\mathbb{Z}^{r}\\
    %     |f(\mathbf{v})|\leq N\quad(f\in B_0)\\
    %     f(\mathbf{v})\neq 0\quad(f\in\Lambda_+)\\
    %     f(\mathbf{v})=0\quad(f\in\Lambda_0)}}
    % e^{2\pi\sqrt{-1}\langle \mathbf{y},\mathbf{v}\rangle}
    % \prod_{f\in\Lambda_+}
    % \frac{1}{f(\mathbf{v})^{k_f}}
    % \\
    %\lim_{N\to\infty}
    (-1)^{\sharp \Lambda_0}
      \sum_{\substack{\mathbf{v}\in\mathbb{Z}^{r}\cap U_N(B_0)\\
%      |\Re f(\mathbf{v})|\leq N\quad(f\in B_0)\\
%          |v_j|\leq N\quad(1\leq j\leq r)\\
      f(\mathbf{v})\neq 0\quad(f\in\Lambda_+\cap B_0)\\
      f(\mathbf{v})=0\quad(f\in\Lambda_0\cap B_0)}}
%    \sum_{\substack{\mathbf{v}\in\mathbb{Z}^{r}\\f(\mathbf{v})\neq 0\quad(f\in B_0)}}
    e^{2\pi\sqrt{-1}\langle \mathbf{y},\mathbf{v}\rangle}
    \prod_{f\in\Lambda_+\cap B_0}
    \frac{1}{f(\mathbf{v})^{k_f}}
    \\
    &\qquad\times
    \prod_{g\in L_0}(-1)^{\sharp (\Lambda_0\cap L_0)}
    \Bigl(
    -\frac{(2\pi\sqrt{-1})^{k_g}}{k_g!}\int_0^1C(k_g,x_g;\const{g})e^{-2\pi\sqrt{-1}\langle \vec{g},\mathbf{v}\rangle x_g}dx_g
    \Bigr)
    \\
    &=
    %\lim_{N\to\infty}
    \prod_{g\in L_0}
    \Bigl(
    -\frac{(2\pi\sqrt{-1})^{k_g}}{k_g!}
    \Bigr)
    \int_0^1\dots\int_0^1
    \Bigl(\prod_{g\in L_0}    C(k_g,x_g;\const{g})    dx_g\Bigr)
    \\
    &\qquad\times(-1)^{\sharp (\Lambda_0\cap B_0)}
      \sum_{\substack{\mathbf{v}\in\mathbb{Z}^{r}\cap U_N(B_0)\\
%      |\Re f(\mathbf{v})|\leq N\quad(f\in B_0)\\
%          |v_j|\leq N\quad(1\leq j\leq r)\\
      f(\mathbf{v})\neq 0\quad(f\in\Lambda_+\cap B_0)\\
      f(\mathbf{v})=0\quad(f\in\Lambda_0\cap B_0)}}
%        \sum_{\substack{\mathbf{v}\in\mathbb{Z}^{r}\\f(\mathbf{v})\neq 0\quad(f\in B_0)}}
    e^{2\pi\sqrt{-1}\langle \mathbf{y}-\sum_{g\in L_0}x_g\vec{g},\mathbf{v}\rangle}
    \Bigl(
    \prod_{f\in \Lambda_+\cap B_0}
    \frac{1}{f(\mathbf{v})^{k_f}}
    \Bigr)\\
  &=\prod_{g\in L_0}
    \Bigl(
    -\frac{(2\pi\sqrt{-1})^{k_g}}{k_g!}
    \Bigr)
    \int_0^1\dots\int_0^1
    \Bigl(\prod_{g\in L_0}    C(k_g,x_g;\const{g})    dx_g\Bigr)
    \\
 &\qquad\times Z_1(N;\mathbf{k}(B_0),\mathbf{y}-\sum_{g\in L_0}x_g\vec{g};B_0;B_0),
  \end{split}
\end{equation}
where $\mathbf{k}(B_0)=(k_f)_{f\in B_0}$.

We want to take the limit $N\to\infty$.
First we claim that it is possible to
exchange
the limit and the integrals.
% we can
% exchange
% the limit and the integrals by %\eqref{eq:seqbdd}
% the absolute uniform convergence in $y$ for $k_f\neq 1$ with
% Lemma \ref{lm:convkn0},
% %Combining \eqref{eq:Cauchy3} and \eqref{eq:Cauchy2}, and
% and by
% the dominated convergence theorem for $k_f=1$
% with Lemma
% \ref{lm:convk1}.
By Lemma \ref{lm:bddS} with $\mu=\sharp B_0=r$, we have
\begin{equation}
  \begin{split}
    &\Bigl|
    \Bigl(\prod_{g\in L_0}C(k_g,x_g;\const{g})\Bigr)
      \sum_{\substack{\mathbf{v}\in\mathbb{Z}^{r}\cap U_N(B_0)\\
%      |\Re f(\mathbf{v})|\leq N\quad(f\in B_0)\\
%          |v_j|\leq N\quad(1\leq j\leq r)\\
      f(\mathbf{v})\neq 0\quad(f\in\Lambda_+\cap B_0)\\
      f(\mathbf{v})=0\quad(f\in\Lambda_0\cap B_0)}}
%        \sum_{\substack{\mathbf{v}\in\mathbb{Z}^{r}\\f(\mathbf{v})\neq 0\quad(f\in B_0)}}
    e^{2\pi\sqrt{-1}\langle \mathbf{y}-\sum_{g\in L_0}x_g\vec{g},\mathbf{v}\rangle}
    \Bigl(
    \prod_{f\in \Lambda_+\cap B_0}
    \frac{1}{f(\mathbf{v})^{k_f}}
    \Bigr)\Bigr|
    \\
    &\leq K \sum_{\mathbf{w}\in\mathbb{Z}^r/\langle\vec{B}_0\rangle}  
    \prod_{f\in\Lambda_1\cap B_0}(1+(1-\{\langle\mathbf{y}+\mathbf{w}-\sum_{g\in L_0}x_g\vec{g},\vec{f}^{B_0}\rangle\})^{-\frac{1}{r+1}}+\{\langle\mathbf{y}+\mathbf{w}-\sum_{g\in L_0}x_g\vec{g},\vec{f}^{B_0}\rangle\}^{-\frac{1}{r+1}})
    \\
    &=K \sum_{\mathbf{w}\in\mathbb{Z}^r/\langle\vec{B}_0\rangle}  
    \prod_{f\in\Lambda_1\cap B_0}X_f,
  \end{split}
\end{equation}
say.
When $f\in\widetilde{\Lambda}$, then under the condition
$
  \mathbf{y}\notin\bigcup_{f\in \widetilde{\Lambda}\cap \Lambda_1\cap B_0}(\mathfrak{H}_{B_0\setminus\{f\}}+\mathbb{Z}^r),
$
we see that $X_f$ is just a constant because of the second assertion of 
Lemma \ref{lm:locfin}. 
When $f\notin\widetilde{\Lambda}$, by the first assertion of Lemma \ref{lm:locfin} we see that $X_f$
fulfills the assumption of Lemma \ref{lm:integrable}, and hence by the lemma
it is integrable 
since $r\geq\sharp(\Lambda_1\cap B_0)$.
% for $f\in \widetilde{\Lambda}\cap \Lambda_1\cap B_0$,
Therefore our claim follows form Lebesgue's dominated convergence theorem.
Note that 
%$\widetilde{\Lambda}\cap \Lambda_1\cap B_0=\Lambda_1\cap B_0$
$\widetilde{\Lambda}\cap \Lambda_1\cap B_0=\widetilde{\Lambda}\cap \Lambda_1$
because of \eqref{eq:LsB}.
%$f\in\widetilde{\Lambda}$ must be an element of each $B\in\mathscr{B}$.

Therefore from \eqref{4-71} we now obtain
\begin{equation}\label{4-85}
\begin{split}
  &\lim_{N\to\infty}Z_1(N;\mathbf{k},\mathbf{y};\Lambda;B_0)\\
  &=\prod_{g\in L_0}
    \Bigl(
    -\frac{(2\pi\sqrt{-1})^{k_g}}{k_g!}
    \Bigr)
    \int_0^1\dots\int_0^1
    \Bigl(\prod_{g\in L_0}    C(k_g,x_g;\const{g})    dx_g\Bigr)
    \\
 &\qquad\times \lim_{N\to\infty}Z_1(N;\mathbf{k}(B_0),\mathbf{y}-\sum_{g\in L_0}x_g\vec{g};B_0;B_0).
  \end{split}
\end{equation}
By Lemma \ref{lm:aex} with $D=\Lambda_1\cap B_0$ the measure of $M(\mathbf{y})$ is 0, and if
$(x_g)_{g\in L_0}\notin M(\mathbf{y})$, then by 
Lemma \ref{lm:calcS} with $B=B_0$ we see that 
$Z_1(N;\mathbf{k}(B_0),\mathbf{y}-\sum_{g\in L_0}x_g\vec{g};B_0;B_0)$ converges as $N\to\infty$.
That is, the integrand on the right-hand side of \eqref{4-85} converges almost everywhere, and
\eqref{lem4_7formula} of Lemma \ref{lm:calcS} implies
\begin{equation}
\label{eq:skylB}
  \begin{split}
    S_1(\mathbf{k},\mathbf{y};\Lambda;B_0)
    &=\lim_{N\to\infty}Z_1(N;\mathbf{k},\mathbf{y};\Lambda;B_0)\\
    &=
    \prod_{g\in L_0}
    \Bigl(
    -\frac{(2\pi\sqrt{-1})^{k_g}}{k_g!}   
    \Bigr)    \int_0^1\dots\int_0^1
    \Bigl(\prod_{g\in L_0}    C(k_g,x_g;\const{g}) dx_g\Bigr)
    \\
    &\qquad\times
    \frac{1}{|\mathbb{Z}^r/\langle\vec{B}_0\rangle|}
    \sum_{\mathbf{w}\in\mathbb{Z}^r/\langle\vec{B}_0\rangle}  
    \prod_{f\in B_0}
    \Bigl(
    -\frac{(2\pi\sqrt{-1})^{k_f}}{k_f!}
%    C(k_f,\{\langle\mathbf{y}+\mathbf{w}-\sum_{g\in L_0}x_g\vec{g},\vec{f}^{B_0}\rangle\};\const{f})
    C(k_f,\{\mathbf{y}+\mathbf{w}-\sum_{g\in L_0}x_g\vec{g}\}_{B_0,f};\const{f})
    \Bigr)
    \\
    &=
    \frac{1}{\sharp(\mathbb{Z}^r/\langle\vec{B}_0\rangle)}
    \prod_{f\in \Lambda}
    \Bigl(
    -\frac{(2\pi\sqrt{-1})^{k_f}}{k_f!}   
    \Bigr)
    \sum_{\mathbf{w}\in\mathbb{Z}^r/\langle\vec{B}_0\rangle}  
    \int_0^1\dots\int_0^1
    \Bigl(\prod_{g\in L_0}    C(k_g,x_g;\const{g}) dx_g\Bigr)
    \\
    &\qquad\times
    \prod_{f\in B_0}
    \Bigl(
%    C(k_f,\{\langle\mathbf{y}+\mathbf{w}-\sum_{g\in L_0}x_g\vec{g},\vec{f}^{B_0}\rangle\};\const{f})
    C(k_f,\{\mathbf{y}+\mathbf{w}-\sum_{g\in L_0}x_g\vec{g}\}_{B_0,f};\const{f})
    \Bigr).
  \end{split}
\end{equation}
%where the integrand converges almost everywhere by Lemma \ref{lm:aex} with $D=\Lambda_1\cap B_0$.
The right-hand side of this equation coincides with that of \eqref{eq:S_int}.
Therefore, to complete the proof of the proposition, the only remaining task is to show that
$S_1(\mathbf{k},\mathbf{y};\Lambda;B_0)=S(\mathbf{k},\mathbf{y};\Lambda)$.
% where $\langle\mathbf{y}+\mathbf{w},\vec{f}^B\rangle\notin\mathbb{Z}$ for all $f\in \Lambda_1\cap B_0$ and $\mathbf{w}\in\mathbb{Z}^r$. 

% Further we define $Z_2(N)$ %to be the same as $Z_1(N)$
% by \eqref{eq:defZ1}
%  with the condition $|f(\mathbf{v})|\leq N\quad(f\in B_0)$ replaced by $|v_j|\leq N\quad (1\leq j\leq r)$.
%Then 

The sum $Z(N;\mathbf{k},\mathbf{y};\Lambda)$ has the expression which is almost the same as
\eqref{4-71}, only the condition $\mathbf{v}\in\mathbb{Z}^r\cap U_N(B_0)$ is replaced by
$\mathbf{v}\in\mathbb{Z}^r\cap W_N$.
Therefore, 
by using Lemmas \ref{lm:UWU} and \ref{lm:ZZdiff}, we see that
the difference $Z(N;\mathbf{k},\mathbf{y};\Lambda)-Z_1(dN;\mathbf{k},\mathbf{y};\Lambda;B_0)$ is evaluated as
\begin{equation}
\begin{split}
    &|Z(N;\mathbf{k},\mathbf{y};\Lambda)-Z_1(dN;\mathbf{k},\mathbf{y};\Lambda;B_0)|\\
    &=
    \Bigl|(-1)^{\sharp (\Lambda_0\cap B_0)}
    \prod_{g\in L_0}
    \Bigl(
    -\frac{(2\pi\sqrt{-1})^{k_g}}{k_g!}
    \Bigr)
    \int_0^1\dots\int_0^1
    \Bigl(\prod_{g\in L_0}    C(k_g,x_g;\const{g})    dx_g\Bigr)
    \\
    &\qquad\times
      \sum_{\substack{%\mathbf{v}\in\mathbb{Z}^{r}\\
\mathbf{v}\in\mathbb{Z}^{r}\cap (W_N\setminus U_{dN}(B_0))\\
%      |\Re f(\mathbf{v})|\leq N\quad(f\in B_0)\\
%          |v_j|\leq N\quad(1\leq j\leq r)\\
      f(\mathbf{v})\neq 0\quad(f\in\Lambda_+\cap B_0)\\
      f(\mathbf{v})=0\quad(f\in\Lambda_0\cap B_0)}}
%        \sum_{\substack{\mathbf{v}\in\mathbb{Z}^{r}\\f(\mathbf{v})\neq 0\quad(f\in B_0)}}
    e^{2\pi\sqrt{-1}\langle \mathbf{y}-\sum_{g\in L_0}x_g\vec{g},\mathbf{v}\rangle}
    \Bigl(
    \prod_{f\in \Lambda_+\cap B_0}
    \frac{1}{f(\mathbf{v})^{k_f}}
    \Bigr)\Bigr|
\\
  &\leq
    % \Bigl|
    % \prod_{g\in L_0}
    % \Bigl(
    % -\frac{(2\pi\sqrt{-1})^{k_g}}{k_g!}
    % \Bigr)
  K
    \int_0^1\dots\int_0^1
%    \Bigl(
\prod_{g\in L_0}  %  C(k_g,x_g;\const{g})    
dx_g
%\Bigr)
    % \\
    % &\qquad\times
    \Bigl|%(-1)^{\sharp (\Lambda_0\cap B_0)}
      \sum_{\substack{%\mathbf{v}\in\mathbb{Z}^{r}\\
\mathbf{v}\in\mathbb{Z}^{r}\cap (W_N\setminus U_{dN}(B_0))\\
%      |\Re f(\mathbf{v})|\leq N\quad(f\in B_0)\\
%          |v_j|\leq N\quad(1\leq j\leq r)\\
      f(\mathbf{v})\neq 0\quad(f\in\Lambda_+\cap B_0)\\
      f(\mathbf{v})=0\quad(f\in\Lambda_0\cap B_0)}}
%        \sum_{\substack{\mathbf{v}\in\mathbb{Z}^{r}\\f(\mathbf{v})\neq 0\quad(f\in B_0)}}
    e^{2\pi\sqrt{-1}\langle \mathbf{y}-\sum_{g\in L_0}x_g\vec{g},\mathbf{v}\rangle}
    \Bigl(
    \prod_{f\in \Lambda_+\cap B_0}
    \frac{1}{f(\mathbf{v})^{k_f}}
    \Bigr)\Bigr|
    \\
    &\leq K' N^{-\frac{1}{r+1}}(\log N)^{r} \int_0^1\dots\int_0^1
    \prod_{g\in L_0} dx_g
    \\
    &\qquad\Bigl(1+
  \sum_{\mathbf{w}\in\mathbb{Z}^r/\langle\vec{B}_0\rangle}  
  \sum_{f\in\Lambda_1\cap B_0}
  \Bigl((1-\{\langle\mathbf{y}+\mathbf{w}-\sum_{g\in L_0}x_g\vec{g},\vec{f}^{B_0}\rangle\})^{-\frac{1}{r+1}}+\{
  \langle\mathbf{y}+\mathbf{w}-\sum_{g\in L_0}x_g\vec{g},\vec{f}^{B_0}\rangle\}^{-\frac{1}{r+1}}\Bigr)\Bigr)
  \\
% \Bigl|(-1)^{\sharp\Lambda_0}
%   \sum_{\substack{\mathbf{v}\in\mathbb{Z}^{r}\cap (W_N\setminus U_{dN})\\
%       f(\mathbf{v})\neq 0\quad(f\in\Lambda_+)\\
%       f(\mathbf{v})=0\quad(f\in\Lambda_0)}}
%   e^{2\pi\sqrt{-1}\langle \mathbf{y},\mathbf{v}\rangle}
%   \prod_{f\in\Lambda_+}
%   \frac{1}{f(\mathbf{v})^{k_f}}
% %  G(\mathbf{y},\mathbf{v})
%   \Bigr|
%   \\
% & \qquad \leq
%   K N^{-\frac{1}{\mu+1}}(\log N)^{r}\Bigl(1+
%   \sum_{\mathbf{w}\in\mathbb{Z}^r/\langle\vec{B}\rangle}  
%   \sum_{f\in\Lambda_1}
%   \Bigl((1-\{\langle\mathbf{y}+\mathbf{w},\vec{f}^B\rangle\})^{-\frac{1}{\mu+1}}+\{
%   \langle\mathbf{y}+\mathbf{w},\vec{f}^B\rangle\}^{-\frac{1}{\mu+1}}\Bigr)\Bigr).
\end{split}
\end{equation}
for some $K,K'>0$.
Again by Lemmas \ref{lm:locfin} and \ref{lm:integrable}, we obtain
\begin{equation}
|Z(N;\mathbf{k},\mathbf{y};\Lambda)-Z_1(dN;\mathbf{k},\mathbf{y};\Lambda;B_0)|
\leq K''N^{-\frac{1}{r+1}}(\log N)^{r}
%   \begin{split}
%     &|Z(N)-Z_1(dN)|\leq K N^{-\frac{1}{r+1}}(\log N)^{r} \int_0^1\dots\int_0^1
%     \prod_{g\in L_0} dx_g
%     \\
%     &\qquad\Bigl(1+
%   \sum_{\mathbf{w}\in\mathbb{Z}^r/\langle\vec{B}_0\rangle}  
%   \sum_{f\in\Lambda_1\cap B_0}
%   \Bigl((1-\{\langle\mathbf{y}+\mathbf{w}-\sum_{g\in L_0}x_g\vec{g},\vec{f}^{B_0}\rangle\})^{-\frac{1}{r+1}}+\{
%   \langle\mathbf{y}+\mathbf{w}-\sum_{g\in L_0}x_g\vec{g},\vec{f}^{B_0}\rangle\}^{-\frac{1}{r+1}}\Bigr)
%   \\
%   &\qquad\leq K'N^{-\frac{1}{r+1}}(\log N)^{r}.
% \end{split}
\end{equation}
for some $K''>0$.
Hence we have
\begin{equation}
S(\mathbf{k},\mathbf{y};\Lambda)
=\lim_{N\to\infty} Z(N;\mathbf{k},\mathbf{y};\Lambda)=
  \lim_{N\to\infty} Z_1(N;\mathbf{k},\mathbf{y};\Lambda;B_0)=S_1(\mathbf{k},\mathbf{y};\Lambda;B_0).
\end{equation}

%   \begin{equation}
%   F(t,y;b)=\frac{t e^{(t-2\pi \sqrt{-1}b)\{y\}}}{e^{t-2\pi \sqrt{-1}b}-1}=\sum_{k=0}^\infty C(k,y;b)\frac{t^k}{k!}.
% \end{equation}

% Consider
% \begin{equation}
% \label{eq:TF}
%   \begin{split}
%     \Tilde{F}(\mathbf{t},\mathbf{y};\Lambda)
%     &=
%     \frac{1}{\sharp(\mathbb{Z}^r/\langle \vec{B}_0\rangle)}
%     \sum_{\mathbf{w}\in \mathbb{Z}^r/\langle \vec{B}_0\rangle}
%     \int_0^1\dots\int_0^1
%     \Bigl(\prod_{g\in L_0}    dx_g\Bigr)
%     \Bigl(
%     \prod_{g\in L_0}\frac{t_g\exp
%       ((t_g-2\pi\sqrt{-1}\const{g}) x_{g})}{e^{t_g-2\pi\sqrt{-1}\const{g}}-1}
%     \Bigr)
%     \\
%     &\qquad
%     \times
%     \Bigl(
%     \prod_{f\in B_0}\frac{t_f\exp
%       ((t_f-2\pi\sqrt{-1}\const{f}) 
%       \{\langle\mathbf{y}+\mathbf{w}-\sum_{g\in L_0}x_g\vec{g},\vec{f}^{B_0}\rangle\}
% %      \{\mathbf{y}-\sum_{g\in L_0}x_g\vec{g}+\mathbf{w}\}_{B_0,f}
%       )}{e^{t_f-2\pi\sqrt{-1}\const{f}}-1}
%     \Bigr)
%     \\
%     &=
%     \Bigl(
%     \prod_{f\in \Lambda}\frac{t_f}{e^{t_f-2\pi\sqrt{-1}\const{f}}-1}
%     \Bigr)
%     \frac{1}{\sharp(\mathbb{Z}^r/\langle \vec{B}_0\rangle)}
%     \sum_{\mathbf{w}\in \mathbb{Z}^r/\langle \vec{B}_0\rangle}
%     \int_0^1\dots\int_0^1
%     \prod_{g\in L_0}    dx_g
%     \\
%     &\qquad\times
%     \exp\Bigl(\sum_{g\in L_0}(t_g-2\pi\sqrt{-1}\const{g}) x_{g})+
%     \sum_{f\in B_0}
%     (t_f-2\pi\sqrt{-1}\const{f}) 
%     \{\langle\mathbf{y}+\mathbf{w}-\sum_{g\in L_0}x_g\vec{g},\vec{f}^{B_0}\rangle\}
%     % \{\mathbf{y}-\sum_{g\in L_0}x_g\vec{g}+\mathbf{w}\}_{B_0,f}
%     \Bigr).
%   \end{split}
% \end{equation}
\end{proof}

% %%%%%%%%%%%%%%%%%%%%%%%%%%%%%%%%%%%%%%%%%%%%%%%%%%%%%%%%%%%%%%%%%%%%%%%%%%%%%%%%%%%%%%%%%%%%%%%%%%%
% \section{Completion of the proof of Theorem \ref{thm:main0}}\label{sec-5}
% %%%%%%%%%%%%%%%%%%%%%%%%%%%%%%%%%%%%%%%%%%%%%%%%%%%%%%%%%%%%%%%%%%%%%%%%%%%%%%%%%%%%%%%%%%%%%%%%%%%%

%In the previous section
We have shown the convergence of $S(\mathbf{k},\mathbf{y};\Lambda)$ in Proposition \ref{prop:main}. Therefore 
to complete the proof of Theorem \ref{thm:main0},
we have only to show 
the continuity of
$S(\mathbf{k},\mathbf{y};\Lambda)$ in $\mathbf{y}$
on $V\setminus\bigcup_{f\in \widetilde{\Lambda}\cap\Lambda_1}(\mathfrak{H}_{\Lambda\setminus\{f\}}+\mathbb{Z}^r)$.

Let $\mathbf{y}_0\in V\setminus\bigcup_{f\in \widetilde{\Lambda}\cap\Lambda_1}(\mathfrak{H}_{\Lambda\setminus\{f\}}+\mathbb{Z}^r)$, and
let $G(\mathbf{y},(x_g))$ be the integrand of \eqref{eq:S_int}.
Since $G(\mathbf{y},(x_g))$ is bounded, we have
\begin{equation}
\label{eq:limG}
  \lim_{\mathbf{y}\to\mathbf{y}_0}\int G(\mathbf{y},(x_g)) 
  \prod_{g\in L_0} dx_g
  =
  \int \lim_{\mathbf{y}\to\mathbf{y}_0}G(\mathbf{y},(x_g)) 
  \prod_{g\in L_0} dx_g.
\end{equation}
Thus it is sufficient to show that
\begin{equation}
\label{eq:limHS}
  \lim_{\mathbf{y}\to\mathbf{y}_0}G(\mathbf{y},(x_g))=G(\mathbf{y}_0,(x_g))
\end{equation}
almost everywhere in $(x_g)$.
By Lemmas \ref{lm:convkn0} and \ref{lm:convk1}, we see that $C(k,\{y\};b)$ is continuous in $y$ on $\mathbb{R}$ if $k\neq 1$, and on $\mathbb{R}\setminus\mathbb{Z}$ if $k=1$.
Hence if $(x_g)$ satisfies
\begin{equation}
\label{eq:xinZ}
  \langle\mathbf{y}_0+\mathbf{w}-\sum_{g\in L_0}x_g\vec{g},\vec{f}^{B_0}\rangle\notin\mathbb{Z}
\end{equation}
for all $f\in \Lambda_1\cap B_0$, then
\eqref{eq:limHS} holds.
Therefore it is sufficient to show that $(x_g)$ satisfies \eqref{eq:xinZ} almost everywhere.
Since 
$\mathbf{y}_0+\mathbf{w}\in V\setminus\bigcup_{f\in \widetilde{\Lambda}\cap\Lambda_1}(\mathfrak{H}_{\Lambda\setminus\{f\}}+\mathbb{Z}^r)$, we see that
the measure of $M(\mathbf{y}_0+\mathbf{w})$ is zero
by Lemma \ref{lm:aex} with $D=\Lambda_1\cap B_0$.

% If $(x_g)$ does not satisfy \eqref{eq:xinZ}, then $(x_g)\in H(f,\mathbf{y}_0+\mathbf{w})$ for some
% $f\in \Lambda_1\cap B_0$.
% By \eqref{eq:LsB} we have
% \begin{equation}
%   \bigcup_{f\in \widetilde{\Lambda}\cap\Lambda_1}(\mathfrak{H}_{\Lambda\setminus\{f\}}+\mathbb{Z}^r)
% =
% \bigcup_{f\in \widetilde{\Lambda}\cap\Lambda_1\cap B_0}(\mathfrak{H}_{B_0\setminus\{f\}}+\mathbb{Z}^r),
% \end{equation}
% so from the assumption
% $\mathbf{y}_0\in V\setminus\bigcup_{f\in \widetilde{\Lambda}\cap\Lambda_1}(\mathfrak{H}_{\Lambda\setminus\{f\}}+\mathbb{Z}^r)$
% we see that
% $\mathbf{y}_0+\mathbf{w}\in\bigcup_{f\in \widetilde{\Lambda}\cap\Lambda_1\cap B_0}(\mathfrak{H}_{B_0\setminus\{f\}}+\mathbb{Z}^r)$.
% Therefore we can apply 
% Lemma \ref{lm:locfin2} to obtain the desired claim.

The proof of Theorem \ref{thm:main0} is thus complete.

%%%%%%%%%%%%%%%%%%%%%%%%%%%%%%%%%%%%%%%%%%%%%%%%%%%%%%%%%%%%%%%%%%%%%%%%%%%%%%%%%%%%%%%%%%%%%%%
\section{The structure of the proof of Theorem \ref{thm:main1} and theorem \ref{thm:main1b}}\label{sec-6}
%%%%%%%%%%%%%%%%%%%%%%%%%%%%%%%%%%%%%%%%%%%%%%%%%%%%%%%%%%%%%%%%%%%%%%%%%%%%%%%%%%%%%%%%%%%%%%%%

Now we start the proof of Theorem \ref{thm:main1} and Theorem \ref{thm:main1b}.   We first prove the assertion (i) of
Theorem \ref{thm:main1}.
%First, let $\mathscr{R}=\mathscr{R}(\Lambda)$ be the set of all subsets
%$R=\{g_1,\ldots,g_{r-1}\}\subset \Lambda$ 
%such that $\vec{R}=\{\vec{g}_1,\ldots,\vec{g}_{r-1}\}$ is %linearly independent set,
Let
\begin{equation}
  \label{eq:def_H_R}
  \mathfrak{H}_{\mathscr{R}}:=
    \bigcup_{R\in\mathscr{R}}(\mathfrak{H}_{R}+\mathbb{Z}^r).
\end{equation}

\begin{lemma}
\label{lm:H}
  The set $\mathfrak{H}_{\mathscr{R}}$ is a locally finite collection of hyperplanes, that is, 
for any $\mathbf{y}\in V$ 
there exists a neighborhood $U$ of $\mathbf{y}$ such that $U$ intersects only finitely many hyperplanes.
\end{lemma}
\begin{proof}
Let $\mathbf{n}_{R}$ be a normal vector of $\mathfrak{H}_{R}$.
We may assume that $\mathbf{n}_R\in\mathbb{Z}^r$, because
$\vec{g}_1,\ldots,\vec{g}_{r-1}\in\mathbb{Z}^r$.
Then the hyperplane 
\begin{equation}
\mathfrak{H}_{R}+\mathbf{v}=
  \{\mathbf{y}+\mathbf{v}~|~\langle\mathbf{y},\mathbf{n}_R\rangle=0\}
\end{equation}
with $\mathbf{v}\in\mathbb{Z}^r$ can be rewritten as
\begin{equation}
  \{\mathbf{y}~|~\langle\mathbf{y},\mathbf{n}_R\rangle-
\langle\mathbf{v},\mathbf{n}_R\rangle=0\}=
  \{\mathbf{y}~|~\langle\mathbf{y}-m\mathbf{e}_R,\mathbf{n}_R\rangle=0\},
\end{equation}
where $m=\langle\mathbf{v},\mathbf{n}_R\rangle\in\mathbb{Z}$
and $\mathbf{e}_R=\mathbf{n}_R/\langle\mathbf{n}_R,\mathbf{n}_R\rangle$.
Therefore
\begin{equation}
    \mathfrak{H}_{R}+\mathbf{v}
    =
    \mathfrak{H}_{R}+m\mathbf{e}_R,
\end{equation}
and so
\begin{equation}
  \mathfrak{H}_{\mathscr{R}}\subset
    \bigcup_{R\in\mathscr{R}}(\mathfrak{H}_{R}+\mathbb{Z}\mathbf{e}_R).
\end{equation}
Hence the assertion follows from this expression and $\sharp\mathscr{R}<\infty$.
\end{proof}

% The one-sided continuity of $F(\mathbf{t},\mathbf{y};\Lambda)$ in Theorem \ref{thm:main1} is shown as follows.
% Since the generating function is continuous with respect to $\mathbf{y}\in V$,
% it is sufficient to check that
% function \eqref{eq:exp_F} 
% %$F$ 
% is the continuous extension of
% \eqref{eq:exp_F_H}.
%By Lemma \ref{lm:H},
%for $\mathbf{y}\in V$, we see that 
%$\mathbf{y}+c\phi\not\in\mathfrak{H}_{\mathscr{R}}$
%for all sufficiently small $c>0$.
%Thus
%to show
%the one-sided continuity of $F(\mathbf{t},\mathbf{y};\Lambda)$ in Theorem \ref{thm:main1}, it is sufficient to show the following lemma.
\begin{lemma}
  \label{lm:contfrac}
\begin{equation}
 \lim_{c\to0+}\{\mathbf{y}+c\phi\}_{B,f}
=\{\mathbf{y}\}_{B,f}
\end{equation}
for $\mathbf{y}\in V$.
\end{lemma}
\begin{proof}
By Lemma \ref{lm:H},
for any $\mathbf{y}\in V$, we see that
$\mathbf{y}+c\phi\not\in\mathfrak{H}_{\mathscr{R}}$
and so
$\langle\mathbf{y}+c\phi,\vec{f}^B\rangle\not\in\mathbb{Z}$
for all sufficiently small $c>0$.   Therefore,
if $\langle\mathbf{y},\vec{f}^B\rangle\not\in\mathbb{Z}$, then
% By Lemma \ref{lm:H},
% for $\mathbf{y}\in\mathfrak{H}_{\mathscr{R}}$ and $\mathbf{v}\in\mathbb{Z}^r$, we see that 
% $\mathbf{y}+\mathbf{v}+c\phi\not\in\mathfrak{H}_{\mathscr{R}}$
% for all sufficiently small $c>0$.
% If $\langle\mathbf{y}+\mathbf{v},\vec{f}^B\rangle\not\in\mathbb{Z}$, then
\begin{equation}
 \lim_{c\to0+}\{\mathbf{y}+c\phi\}_{B,f}=
 \lim_{c\to0+}\{\langle\mathbf{y}+c\phi,\vec{f}^B\rangle\}
=\{\langle\mathbf{y},\vec{f}^B\rangle\}
=\{\mathbf{y}\}_{B,f}
\end{equation}
by \eqref{2-6}.
%since for $a\in\mathbb{R}\setminus\mathbb{Z}$, we have $\{a\}=1-\{-a\}$.
If $\langle\mathbf{y},\vec{f}^B\rangle\in\mathbb{Z}$, then
\begin{multline}
 \lim_{c\to0+}\{\mathbf{y}+c\phi\}_{B,f}\\
=
\begin{cases}
\displaystyle
  \lim_{c\to0+}\{\langle\mathbf{y}+c\phi,\vec{f}^B\rangle\}
=  \lim_{c\to0+}\{c\langle\phi,\vec{f}^B\rangle\}
=0=
\{\langle\mathbf{y},\vec{f}^B\rangle\}
\quad&(\langle\phi,\vec{f}^B\rangle>0),\\
\displaystyle
  \lim_{c\to0+}1-\{-\langle\mathbf{y}+c\phi,\vec{f}^B\rangle\}
=  \lim_{c\to0+}1-\{-c\langle\phi,\vec{f}^B\rangle\}
=1
=
  1-\{-\langle\mathbf{y},\vec{f}^B\rangle\}
\quad&(\langle\phi,\vec{f}^B\rangle<0).
\end{cases}
\end{multline}
Hence we have the assertion.
% \begin{equation}
%    \lim_{c\to0+}\{\mathbf{y}+c\phi\}_{B,f}
% =\{\mathbf{y}\}_{B,f}.
% \end{equation}
\end{proof}

By this lemma we immediately obtain
\begin{equation}
  \lim_{c\to0+}F
(\mathbf{t},\mathbf{y}+c\phi;\Lambda)
=F(\mathbf{t},\mathbf{y};\Lambda).
\end{equation}
This shows the assertion (i) of Theorem \ref{thm:main1}.

Next, observe that the right-hand side of \eqref{eq:S_int} can be
defined for any $\mathbf{y}\in V$ (though \eqref{eq:S_int} itself is
valid only under the assumption of Proposition \ref{prop:main}).
Therefore, we can define
% \begin{equation}
% \widetilde{C}(\mathbf{k},\mathbf{y};\Lambda)=S(\mathbf{k},\mathbf{y};\Lambda)
% \left(\prod_{f\in\Lambda}-\frac{k_f!}{(2\pi\sqrt{-1})^{k_f}}\right),
% \end{equation}
$\widetilde{C}(\mathbf{k},\mathbf{y};\Lambda)$ 
for any $\mathbf{y}\in V$ as the
$\displaystyle\Bigl(\prod_{f\in\Lambda}-\frac{k_f!}{(2\pi\sqrt{-1})^{k_f}}\Bigr)$ multiple of
the right-hand side of
\eqref{eq:S_int},
and we introduce the generating function of 
$\widetilde{C}(\mathbf{k},\mathbf{y};\Lambda)$ of the form
%$S(\mathbf{k},\mathbf{y};\Lambda)$ in the form
\begin{equation}\label{tilde_F_def}
\widetilde{F}(\mathbf{t},\mathbf{y};\Lambda)=\sum_{\mathbf{k}\in \mathbb{N}_0^{\sharp\Lambda}}
\widetilde{C}(\mathbf{k},\mathbf{y};\Lambda)
\prod_{f\in\Lambda}\frac{t_f^{k_f}}{k_f!},
\end{equation}
where $\mathbf{t}=(t_f)_{f\in \Lambda}\in\mathbb{C}^{\sharp\Lambda}$.
A more explicit form of the generating function can be deduced by substituting
the formula of Proposition \ref{prop:main} into \eqref{tilde_F_def}.
In fact,
\begin{lemma}
For any $\mathbf{y}\in V$, 
the series on the right-hand side of \eqref{tilde_F_def} is absolutely and uniformly convergent in the neighborhood of the origin with respect to $\mathbf{t}\in\mathbb{C}^{\sharp\Lambda}$. Furthermore we have
\begin{equation}
\label{eq:TF}
  \begin{split}
    \widetilde{F}(\mathbf{t},\mathbf{y};\Lambda)
    &=
    \Bigl(
    \prod_{f\in \Lambda}\frac{t_f}{\exp(t_f-2\pi\sqrt{-1}\const{f})-1}
    \Bigr)
    \frac{1}{\sharp(\mathbb{Z}^r/\langle \vec{B}_0\rangle)}
    \sum_{\mathbf{w}\in \mathbb{Z}^r/\langle \vec{B}_0\rangle}
    \int_0^1\dots\int_0^1
    \prod_{g\in L_0}    dx_g
    \\
    &\qquad\times
    \exp\Bigl(\sum_{g\in L_0}(t_g-2\pi\sqrt{-1}\const{g}) x_{g}+
    \sum_{f\in B_0}
    (t_f-2\pi\sqrt{-1}\const{f}) 
 %   \{\langle\mathbf{y}+\mathbf{w}-\sum_{g\in L_0}x_g\vec{g},\vec{f}^{B_0}\rangle\}
     \{\mathbf{y}+\mathbf{w}-\sum_{g\in L_0}x_g\vec{g}\}_{B_0,f}
    \Bigr).
  \end{split}
\end{equation}
\end{lemma}
\begin{proof}
The following proof is similar to that of \cite[Lemma 7]{KM3}.
By \eqref{4-80} we see that,
for $b\in\mathbb{C}$,
there exists a sufficiently small $R_b>0$ such that %by the Cauchy integral formula 
% \begin{equation}\label{F_def_exp}
%   F(t,y;b)=\frac{t e^{(t-2\pi \sqrt{-1}b)y}}{e^{t-2\pi \sqrt{-1}b}-1}=\sum_{k=0}^\infty C(k,y;b)\frac{t^k}{k!}.
% \end{equation}
\begin{equation}
  \frac{C(k,y;b)}{k!}=%\frac{1}{2\pi\sqrt{-1}}\int_{|z|=R_b}
  %F(z,y;b)dz
  =
  \frac{1}{2\pi\sqrt{-1}}\int_{|z|=R_b}
  \frac{z e^{(z-2\pi \sqrt{-1}b)y}}{e^{z-2\pi \sqrt{-1}b}-1}
\frac{dz}{z^{k+1}}
\end{equation}
holds for $y\in\mathbb{R}$. Thus
we have for $0\leq y\leq 1$ 
\begin{equation}
\Bigl\lvert\frac{C(k,y;b)}{k!}\Bigr\rvert
\leq\frac{1}{2\pi}\int_{|z|=R_b}\Bigl\lvert
\frac{z e^{(z-2\pi \sqrt{-1}b)y}}{e^{z-2\pi \sqrt{-1}b}-1}
\Bigr\rvert\frac{|dz|}{R_b^{k+1}}
\leq\frac{C_{b}}{R_b^k},
\end{equation}
where
\begin{equation}
  C_{b}=\max\Bigl\{\Bigl\lvert
\frac{z e^{(z-2\pi \sqrt{-1}b)y}}{e^{z-2\pi \sqrt{-1}b}-1}\Bigr\rvert~\Bigm|~|z|=R_b,~0\leq y\leq1\Bigr\}.
\end{equation}
Fix $r$ such that $0<r<\min_{f\in\Lambda} R_{\const{f}}$. Then for $|t_f|<r$ ($f\in\Lambda$),
\begin{equation}
  \begin{split}
  \Bigl\lvert 
  \widetilde{C}(\mathbf{k},\mathbf{y};\Lambda)
  \prod_{f\in\Lambda}\frac{t_f^{k_f}}{k_f!}
  \Bigr\rvert
  &\leq
  \frac{1}{\sharp(\mathbb{Z}^r/\langle\vec{B}_0\rangle)}
  \sum_{\mathbf{w}\in\mathbb{Z}^r/\langle\vec{B}_0\rangle}  
  \Bigl(\prod_{f\in\Lambda}|t_f|^{k_f}\Bigr)
  \int_0^1\dots\int_0^1
  \Bigl\lvert
  \Bigl(\prod_{g\in L_0} \frac{C(k_g,x_g;\const{g})}{k_g!} dx_g\Bigr)
  \\
  &\qquad\times
    \prod_{f\in B_0}
    \Bigl(
    \frac{C(k_f,\{\mathbf{y}+\mathbf{w}-\sum_{g\in L_0}x_g\vec{g}\}_{B_0,f};\const{f})}{k_f!}
    \Bigr)
    \Bigr\rvert
    \\
    &\leq
    % C
    \frac{1}{\sharp(\mathbb{Z}^r/\langle\vec{B}_0\rangle)}
    \sum_{\mathbf{w}\in\mathbb{Z}^r/\langle\vec{B}_0\rangle}  
    \Bigl(\prod_{f\in\Lambda}r^{k_f}\Bigr)
    \int_0^1\dots\int_0^1
    \prod_{f\in\Lambda}
    \frac{C_{\const{f}}}{R_{\const{f}}^{k_f}}
    \prod_{g\in L_0} dx_g
    \\
    &=
    \prod_{f\in\Lambda}
    C_{\const{f}}
    \Bigl(\frac{r}{R_{\const{f}}}\Bigr)^{k_f}.
\end{split}
\end{equation}
Since
\begin{equation}
  \sum_{\mathbf{k}\in\mathbb{N}_0^{\sharp\Lambda}}
  \prod_{f\in\Lambda}
  C_{\const{f}}
  \Bigl(\frac{r}{R_{\const{f}}}\Bigr)^{k_f}
  =
  \prod_{f\in\Lambda}
  \Bigl(\frac{C_{\const{f}}}{1-r/R_{\const{f}}}\Bigr)<\infty,
\end{equation}
we have the uniform and absolute convergence of $\widetilde{F}(\mathbf{t},\mathbf{y};\Lambda)$, which implies the holomorphy of
$\widetilde{F}(\mathbf{t},\mathbf{y};\Lambda)$
in the neighborhood of the origin with respect to $\mathbf{t}\in\mathbb{C}^{\sharp\Lambda}$.

Furthermore by exchanging the sum and the integral and using \eqref{F_def_exp} we obtain
\begin{equation}
  \begin{split}
    \widetilde{F}(\mathbf{t},\mathbf{y};\Lambda)
    &=
    \frac{1}{\sharp(\mathbb{Z}^r/\langle \vec{B}_0\rangle)}
    \sum_{\mathbf{w}\in \mathbb{Z}^r/\langle \vec{B}_0\rangle}
    \int_0^1\dots\int_0^1
    \Bigl(\prod_{g\in L_0}    dx_g\Bigr)
    \Bigl(
    \prod_{g\in L_0}\frac{t_g\exp
      ((t_g-2\pi\sqrt{-1}\const{g}) x_{g})}{\exp(t_g-2\pi\sqrt{-1}\const{g})-1}
    \Bigr)
    \\
    &\qquad
    \times
    \Bigl(
    \prod_{f\in B_0}\frac{t_f\exp
      ((t_f-2\pi\sqrt{-1}\const{f}) 
%      \{\langle\mathbf{y}+\mathbf{w}-\sum_{g\in L_0}x_g\vec{g},\vec{f}^{B_0}\rangle\}
      \{\mathbf{y}+\mathbf{w}-\sum_{g\in L_0}x_g\vec{g}\}_{B_0,f}
      )}{\exp(t_f-2\pi\sqrt{-1}\const{f})-1}
    \Bigr),
 %    \\
 %    &=
 %    \Bigl(
 %    \prod_{f\in \Lambda}\frac{t_f}{\exp(t_f-2\pi\sqrt{-1}\const{f})-1}
 %    \Bigr)
 %    \frac{1}{\sharp(\mathbb{Z}^r/\langle \vec{B}_0\rangle)}
 %    \sum_{\mathbf{w}\in \mathbb{Z}^r/\langle \vec{B}_0\rangle}
 %    \int_0^1\dots\int_0^1
 %    \prod_{g\in L_0}    dx_g
 %    \\
 %    &\qquad\times
 %    \exp\Bigl(\sum_{g\in L_0}(t_g-2\pi\sqrt{-1}\const{g}) x_{g})+
 %    \sum_{f\in B_0}
 %    (t_f-2\pi\sqrt{-1}\const{f}) 
 % %   \{\langle\mathbf{y}+\mathbf{w}-\sum_{g\in L_0}x_g\vec{g},\vec{f}^{B_0}\rangle\}
 %     \{\mathbf{y}+\mathbf{w}-\sum_{g\in L_0}x_g\vec{g}\}_{B_0,f}
 %    \Bigr).
  \end{split}
\end{equation}
which yields \eqref{eq:TF}.
\end{proof}

\begin{lemma}
  \label{lm:cont}  
    $\widetilde{F}(\mathbf{t},\mathbf{y};\Lambda)$ is continuous in $\mathbf{y}$ %for
% for each $f\in B_0$,
% there exist $g\in L_0$ such that
% $\vec{g}\notin\mathfrak{H}_{B_0\setminus\{f\}}$.
% $\rank\langle\vec{\Lambda}\setminus\{\vec{f}\}\rangle=r$
% for all $f\in\Lambda$.
on
$V\setminus\bigcup_{f\in \widetilde{\Lambda}}(\mathfrak{H}_{\Lambda\setminus\{f\}}+\mathbb{Z}^r)$
and 
has one-sided continuity in $\mathbf{y}\in V$ in the direction $\phi$.
% $\mathbf{y}\notin\mathfrak{H}_{B_0\setminus\{f\}}+\mathbb{Z}^r$ 
% for $f\in \widetilde{\Lambda}\cap B_0$. 
% for $f\in B_0$,
% $\mathbf{y}\notin\mathfrak{H}_{B_0\setminus\{f\}}+\mathbb{Z}^r$ or
% there exist $g\in L_0$ such that
% $\vec{g}\notin\mathfrak{H}_{B_0\setminus\{f\}}$.
\end{lemma}
\begin{proof}
  The proof is almost the same as that of the continuity of $S(\mathbf{k},\mathbf{y};\Lambda)$ in \eqref{eq:limG}.
%Lemma \ref{lm:locfin2}.
  Let $G(\mathbf{y},(x_g))$ be the integrand of the last expression of \eqref{eq:TF}.
% Since $G(\mathbf{y},(x_g))$ is bounded, we have
% \begin{equation}
%   \lim_{\mathbf{y}\to\mathbf{y}_0}\int G(\mathbf{y},(x_g)) 
%   \prod_{g\in L_0} dx_g
%   =
%   \int \lim_{\mathbf{y}\to\mathbf{y}_0}G(\mathbf{y},(x_g)) 
%   \prod_{g\in L_0} dx_g.
% \end{equation}
% Thus it is sufficient to show that
% \begin{equation}
% \label{eq:limH}
%   \lim_{\mathbf{y}\to\mathbf{y}_0}G(\mathbf{y},(x_g))=G(\mathbf{y}_0,(x_g))
% \end{equation}
% almost everywhere in $(x_g)$.
In this case, the continuity comes from
%If $(x_g)$ satisfies
% \begin{equation}
% %\label{eq:xinZ}
%   \langle\mathbf{y}_0+\mathbf{w}-\sum_{g\in L_0}x_g\vec{g},\vec{f}^{B_0}\rangle\notin\mathbb{Z},
% \end{equation}
\eqref{eq:xinZ}
for all $f\in B_0$.
% , then
% \eqref{eq:limH} holds and the assertion follows 
% from
% Lemma \ref{lm:locfin2}
% because
Hence the first assertion follows from
Lemma \ref{lm:aex} with $D=B_0$.
% $\bigcup_{f\in \widetilde{\Lambda}}(\mathfrak{H}_{\Lambda\setminus\{f\}}+\mathbb{Z}^r)
% =
% \bigcup_{f\in \widetilde{\Lambda}}(\mathfrak{H}_{B_0\setminus\{f\}}+\mathbb{Z}^r)$.
%
%The first assertion obviously includes the second assertion in the case
%$\mathbf{y}\notin \bigcup_{f\in \widetilde{\Lambda}}(\mathfrak{H}_{\Lambda\setminus\{f\}}+\mathbb{Z}^r)$,
%so we have only to show the
%second assertion for 
%$\mathbf{y}\in\bigcup_{f\in \widetilde{\Lambda}}(\mathfrak{H}_{\Lambda\setminus\{f\}}+\mathbb{Z}^r)$
%to complete the proof.
%Let
%$\mathbf{y}_0\in \bigcup_{f\in \widetilde{\Lambda}}(\mathfrak{H}_{\Lambda\setminus\{f\}}+\mathbb{Z}^r)$,
%put $\mathbf{y}=\mathbf{y}_0+c\phi$,
%and consider the limit $c\to 0+$.
The second assertion immediately follows from Lemma \ref{lm:contfrac}.

% Fix $f\in B_0$ and
% let $U=(\vec{g})_{g\in L_0}$ be an $r\times\sharp L_0$ matrix and
% $\mathbf{x}=(x_g)_{g\in L_0}$ be a column vector.
% Then
% \begin{equation}
%     \langle\mathbf{y}_0+\mathbf{w}-\sum_{g\in L_0}x_g\vec{g},\vec{f}^{B_0}\rangle=n
% \end{equation}
% is rewritten as
% \begin{equation}
% \label{eq:Ux}
%   \langle U\mathbf{x}-(\mathbf{y}_0+\mathbf{w}-n\vec{f}),\vec{f}^{B_0}\rangle=0.
% \end{equation}
% %By Lemma \ref{lm:nondeg},
% By the assumption of the statement, for each $f\in B_0$,
% there exist $g\in L_0$ such that
% $\vec{g}\notin\mathfrak{H}_{B_0\setminus\{f\}}$, namely,
% $\langle\vec{g},\vec{f}^{B_0}\rangle\neq 0$.
% Thus we have a solution $\mathbf{x}_0$ of \eqref{eq:Ux} and
% \begin{equation}
%   \langle U(\mathbf{x}-\mathbf{x}_0),\vec{f}^{B_0}\rangle=0.
% \end{equation}
% The condition 
% $\langle\vec{g},\vec{f}^{B_0}\rangle\neq 0$
% also implies that $\dim\ker h=\sharp L_0-1$ for the linear functional $h(\mathbf{v})=\langle U\mathbf{v},\vec{f}^{B_0}\rangle$, and hence \eqref{eq:xinZ} holds almost everywhere.

% Then we show that
% \begin{equation}
%   \lim_{\mathbf{y}\to\mathbf{y}_0}\int G(\mathbf{y},(x_g)) 
%   \prod_{g\in L_0} dx_g
%   =
%   \int G(\mathbf{y}_0,(x_g)) 
%   \prod_{g\in L_0} dx_g.
% \end{equation}
\end{proof}

We have obtained the assertions, corresponding to (i), (ii) and (iii) of Theorem \ref{thm:main1}, for
$\widetilde{F}(\mathbf{t},\mathbf{y};\Lambda)$.
In the following sections, we will prove 
\begin{equation}
\label{tilde_F=F}
\widetilde{F}(\mathbf{t},\mathbf{y};\Lambda)=F(\mathbf{t},\mathbf{y};\Lambda)
\end{equation}
for $\mathbf{y}\in V\setminus\mathfrak{H}_{\mathscr{R}}$.
Then 
$\widetilde{F}(\mathbf{t},\mathbf{y};\Lambda)=F(\mathbf{t},\mathbf{y};\Lambda)$
on the whole $V$ by
the one-sided continuity of $F(\mathbf{t},\mathbf{y};\Lambda)$
and $\widetilde{F}(\mathbf{t},\mathbf{y};\Lambda)$ which we have already shown.
Thus
automatically the assertions (ii) and (iii) of Theorem \ref{thm:main1} will follow.

Also, comparing \eqref{F_Taylor_exp} with \eqref{tilde_F_def}, we find that
\begin{equation}\label{Ctilde=C}
\widetilde{C}(\mathbf{k},\mathbf{y};\Lambda)=C(\mathbf{k},\mathbf{y};\Lambda).
\end{equation}
By the definition of $\widetilde{C}(\mathbf{k},\mathbf{y};\Lambda)$ and Proposition \ref{prop:main},
we have
$$
\widetilde{C}(\mathbf{k},\mathbf{y};\Lambda)=
\Bigl(\prod_{f\in\Lambda}-\frac{k_f!}{(2\pi\sqrt{-1})^{k_f}}\Bigr)
S(\mathbf{k},\mathbf{y};\Lambda)
$$
for 
$\mathbf{y}\in V\setminus\bigcup_{f\in \widetilde{\Lambda}\cap\Lambda_1}(\mathfrak{H}_{\Lambda\setminus\{f\}}+\mathbb{Z}^r)$.
Combining this with \eqref{Ctilde=C}, we obtain the assertion of Theorem \ref{thm:main1b}.
Therefore the only remaining task is to show \eqref{tilde_F=F} for 
$\mathbf{y}\in V\setminus\mathfrak{H}_{\mathscr{R}}$.

%%%%%%%%%%%%%%%%%%%%%%%%%%%%%%%%%%%%%%%%%%%%%%%%%%%%%%%%%%%%%%%%%%%%%%%%%%%%%%%%%%%%%%%%%%%%%%%%%%%%
\section{The generating function and convex polytopes}\label{sec-7}
%%%%%%%%%%%%%%%%%%%%%%%%%%%%%%%%%%%%%%%%%%%%%%%%%%%%%%%%%%%%%%%%%%%%%%%%%%%%%%%%%%%%%%%%%%%%%%%%%%%%

The aim of the following three sections is to prove \eqref{tilde_F=F},
which will be shown in Section \ref{sec-9}.    
The present and the next sections are devoted to the preparations for
the proof of \eqref{tilde_F=F}, which are connected with the theory of
convex polytopes.

First, we summarize some 
definitions and
facts about convex polytopes (see \cite{KM5,KM3,Zie95}).
For a subset $X\subset\mathbb{R}^N$, we denote by $\Conv(X)$ the convex hull of $X$.
A subset $\mathcal{P}\subset\mathbb{R}^N$ is called a convex polytope 
if $\mathcal{P}=\Conv(X)$ for some finite subset $X\subset\mathbb{R}^N$.
Let $\mathcal{P}$ be a $d$-dimensional polytope.
Let $\mathcal{H}$ be a hyperplane in $\mathbb{R}^N$.
Then $\mathcal{H}$ divides $\mathbb{R}^N$ into two half-spaces.
If $\mathcal{P}$ is entirely contained in one of the two closed half-spaces
and $\mathcal{P}\cap\mathcal{H}\neq\emptyset$, then $\mathcal{H}$ is
called a supporting hyperplane of $\mathcal{P}$.
For a supporting hyperplane $\mathcal{H}$ and
a subset $\mathcal{F}=\mathcal{P}\cap\mathcal{H}\neq\emptyset$,
the subset $\mathcal{F}$
is called a face of the polytope $\mathcal{P}$
and $\mathcal{H}$ a supporting hyperplane associated with $\mathcal{F}$.
If the dimension of a face $\mathcal{F}$ is $j$, then we call it a $j$-face $\mathcal{F}$. 
A $0$-face is called a vertex,
a $1$-face an edge
and a $(d-1)$-face a facet.
For convenience, 
we regard $\mathcal{P}$ itself as its unique $d$-face.
Let $\Vrt(\mathcal{P})$ be the set of all vertices of $\mathcal{P}$. 
Then 
\begin{equation}
\label{eq:prop_F}
  \mathcal{F}=\Conv(\Vrt(\mathcal{P})\cap\mathcal{F}),
\end{equation}
for a face $\mathcal{F}$.
A $d$-dimensional simple polytope is a polytope whose vertices are adjacent to exactly $d$ edges. 
% It is known that 
% triangulation of a convex polytope can be executed without adding any vertices \cite{Stan80}.

For $\mathbf{a}={}^t(a_1,\ldots,a_N),\mathbf{b}={}^t(b_1,\ldots,b_N)\in\mathbb{R}^N$
we define $\mathbf{a}\cdot \mathbf{b}=a_1b_1+\cdots+a_Nb_N$.
The definition of polytopes above is that of ``V-polytopes''.
We mainly deal with another representation of polytopes,
``H-polytopes'' instead, that is, a bounded subset of the form
\begin{equation}
  \mathcal{P}=\bigcap_{i\in I}\mathcal{H}_i^+\subset\mathbb{R}^N,
\end{equation}
where $\sharp{I}<\infty$ and $\mathcal{H}_i^+=\{\mathbf{x}\in\mathbb{R}^N~|~\mathbf{a}_i\cdot \mathbf{x}\geq h_i\}$
with $\mathbf{a}_i\in\mathbb{R}^N$ and $h_i\in\mathbb{R}$.
% Let
% \begin{equation}
%   \mathcal{P}=\bigcap_{i\in I}\mathcal{H}_i^+\subset\mathbb{R}^n,
% \end{equation}
% where $\sharp I<\infty$ and $\mathcal{H}_i^+=\{\mathbf{x}\in\mathbb{R}^n~|~\mathbf{a}_i\cdot \mathbf{x}\geq h_i\}$
% with $\mathbf{a}_i\in\mathbb{R}^n$ and $h_i\in\mathbb{R}$.
% The following theorem is intuitively clear but nontrivial. % (see, for example, \cite[Theorem 1.1]{Zie95}).
It is known (Weyl--Minkowski)
that H-polytopes are V-polytopes and vice versa.

We have an expression of $k$-faces in terms of hyperplanes
 $\mathcal{H}_i=\{\mathbf{x}\in\mathbb{R}^n~|~\mathbf{a}_i\cdot \mathbf{x}=h_i\}$.
\begin{prop}[{\cite[Proposition 2.7]{KM3}}]
\label{prop:k_face_1}
  Let $J\subset I$.
  Assume that
  $\mathcal{F}=\mathcal{P}\cap\bigcap_{j\in J}\mathcal{H}_j\neq\emptyset$.
  Then $\mathcal{F}$ is a face.
\end{prop}
\begin{prop}[{\cite[Proposition 2.8]{KM3}}]
\label{prop:k_face_2}
Let $\mathcal{H}$ be a supporting hyperplane and
$\mathcal{F}=\mathcal{P}\cap\mathcal{H}$ is a $k$-face.
Then there exists a set of indices $J\subset I$ such that
$\sharp J=(\dim\mathcal{P})-k$ and
$\mathcal{F}=\mathcal{P}\cap\bigcap_{j\in J}\mathcal{H}_j$.
\end{prop}
\begin{lemma}[{\cite[Lemma 6.5]{KM5}}]
\label{lm:simple_exp}
Let $\mathcal{P}$ be a simple polytope and   
 $\{\mathbf{p}_0,\ldots,\mathbf{p}_K\}$ be the vertices of $\mathcal{P}$.
Let 
\begin{equation}
E_k=\{j~|~\Conv(\{\mathbf{p}_k,\mathbf{p}_j\})\text{ is an edge}\}.
\end{equation}
Then we have
\begin{equation}
    \int_{\mathcal{P}} e^{\mathbf{a}\cdot \mathbf{x}}d\mathbf{x}
=
\sum_{k=0}^K|\det(\mathbf{p}_k-\mathbf{p}_j)_{j\in E_k}|\frac{e^{\mathbf{a}\cdot\mathbf{p}_k}}{\prod_{j\in E_k}\mathbf{a}\cdot (\mathbf{p}_k-\mathbf{p}_j)}.
\end{equation}
\end{lemma}

%Let $n\in\mathbb{N}$ and let $\{\mathbf{e}_1,\ldots,\mathbf{e}_n\}$ be a standard basis of $\mathbb{R}^n$.
%Let $M=(m_{ij})_{1\leq j\leq n}^{1\leq i\leq n}$
%be an $n\times n$ matrix.
%For $1\leq j\leq n$, 
%and $\mathbf{v}\in\mathbb{R}^n$,
%let $M(j,\mathbf{v})$ be the matrix $M$ with only the $j$-th column replaced by $\mathbf{v}$.
%Let
%\begin{equation}
%\label{eq:uj}
%\mathbf{u}(j)=
%\det M(j,(\mathbf{e}_i)_{1\leq i\leq n})=
%\det
%\begin{pmatrix}
%  m_{11}&\ldots&m_{1\,k-1}&\mathbf{e}_1&m_{1\,k+1}&\ldots&m_{1n}\\
%  \vdots&\ddots&\vdots&\vdots&\vdots&\ddots&\vdots\\
%  m_{n1}&\ldots&m_{n\,k-1}&\mathbf{e}_n&m_{n\,k+1}&\ldots&m_{nn}
%\end{pmatrix}
%=\sum_{i=1}^n b_i\mathbf{e}_i\in\mathbb{R}^n,
%\end{equation}
%where $b_i$ is the cofactor of $(i,j)$-entry of $M$.
%The following lemma is shown by elementary linear algebras.
%\begin{lemma}
%\label{lm:elementary}
%  \begin{enumerate}
%  \item 
%    For $\mathbf{v}\in\mathbb{R}^n$, we have
%\begin{equation}
%\mathbf{u}(j)\cdot\mathbf{v}=\det M(j,\mathbf{v}).
%\end{equation}
%\item Let $U=(\mathbf{u}(j))_{1\leq j\leq n}$ be the $n\times n$ matrix
%whose $j$-th column consists of $\mathbf{u}(j)$.
%Then
%\begin{equation}
%  \det U=(\det M)^{n-1}.
%\end{equation}
%\end{enumerate}
%\end{lemma}

Now we present a fundamental proposition, which gives an expression of
$\widetilde{F}(\mathbf{t},\mathbf{y};\Lambda)$ involving integrals over
certain convex polytopes.    This proposition is a generalization of \cite[Theorem 7]{KM3}.

\begin{prop}
\label{prop:gen_func}
  \begin{equation}
  \begin{split}
    \widetilde{F}(\mathbf{t},\mathbf{y};\Lambda)
    &=
    \Bigl(
    \prod_{f\in \Lambda}\frac{t_f}{\exp(t_f-2\pi\sqrt{-1}\const{f})-1}
    \Bigr)
    \frac{1}{\sharp(\mathbb{Z}^r/\langle \vec{B}_0\rangle)}
%    \sum_{\mathbf{w}\in \mathbb{Z}^r/\langle \vec{B}_0\rangle}
    \\
    &\qquad\times
    \sum_{\mathbf{m}\in\mathbb{Z}^r}
    \exp\Bigl(
    \sum_{f\in B_0}
    (t_f-2\pi\sqrt{-1}\const{f}) 
    \langle\mathbf{y}+\mathbf{m},\vec{f}^{B_0}\rangle
    \Bigr)
%    \{\langle\mathbf{y}+\mathbf{w}-\sum_{g\in L_0}x_g\vec{g},\vec{f}^B\rangle\}
    \\
    &\qquad\times
    \int_{\mathcal{P}(\mathbf{m};\mathbf{y})}
    \exp\Bigl(\sum_{g\in L_0}
    t_g^*x_g
    \Bigr)
    \prod_{g\in L_0}    dx_g,
  \end{split}
\end{equation}
where
\begin{equation}
\label{def_t_g^*}
  t_g^*=
  (t_g-2\pi\sqrt{-1}\const{g}) 
  -
  \sum_{f\in B_0}
  (t_f-2\pi\sqrt{-1}\const{f}) 
  \langle\vec{g},\vec{f}^{B_0}\rangle
\end{equation}
and
\begin{equation}
\label{eq:D_m_y}
  \mathcal{P}(\mathbf{m};\mathbf{y})=
  \left\{\mathbf{x}=(x_g)_{g\in L_0}~\middle|~
  \begin{aligned}
  &0\leq x_g\leq1
  \quad (g\in L_0)
\\
  &\langle \mathbf{y}+\mathbf{m}-\vec{f},\vec{f}^{B_0}\rangle\leq
  \sum_{g\in L_0}x_g\langle\vec{g},\vec{f}^{B_0}\rangle
  \leq \langle \mathbf{y}+\mathbf{m},\vec{f}^{B_0}\rangle
  \quad (f\in B_0)
\end{aligned}
%\right.
\right\}
\end{equation}
is a convex polytope or an empty set.
\end{prop}

\begin{proof}
We fix a representative of each $\mathbf{w}\in\mathbb{Z}^r/\langle \vec{B}_0\rangle$ in $\mathbb{Z}^r$. 
Let $\mathbf{m}=(m_f)_{f\in B_0}\in\mathbb{Z}^r$, and denote by
$\mathcal{Q}(\mathbf{w},\mathbf{m})$ the set of all
$\mathbf{x}=(x_g)_{g\in L_0}$ satisfying the conditions
$0\leq x_g\leq 1$ ($g\in L_0$) and
\begin{equation}\label{7-9}
  -m_f\leq\Bigl\langle\mathbf{y}+\mathbf{w}-\sum_{g\in L_0}x_g\vec{g},\vec{f}^{B_0}\Bigr\rangle<-m_f+1.
\end{equation}
This condition is equivalent to
\begin{equation}\label{7-10}
    \Bigl\langle\mathbf{y}+\mathbf{w}+\sum_{h\in B_0}m_h\vec{h}-\vec{f},\vec{f}^{B_0}\Bigr\rangle
<
\sum_{g\in L_0}x_g\langle\vec{g},\vec{f}^{B_0}\rangle
\leq
    \Bigl\langle\mathbf{y}+\mathbf{w}+\sum_{h\in B_0}m_h\vec{h},\vec{f}^{B_0}\Bigr\rangle,
  \end{equation}
because $\Bigl\langle\sum_{h\in B_0}m_h\vec{h},\vec{f}^{B_0}
\Bigr\rangle=m_f$.
Also we have
\begin{equation}\label{7-11}
  \begin{split}
    \Bigl\{\Bigl\langle\mathbf{y}+\mathbf{w}-\sum_{g\in L_0}x_g\vec{g},\vec{f}^{B_0}\Bigr\rangle\Bigr\}
    &=
    \Bigl\langle\mathbf{y}+\mathbf{w}-\sum_{g\in L_0}x_g\vec{g},\vec{f}^{B_0}\Bigr\rangle+m_f
    \\
    &=
    \Bigl\langle\mathbf{y}+\mathbf{w}+\sum_{h\in B_0}m_h\vec{h},\vec{f}^{B_0}\Bigr\rangle
-\sum_{g\in L_0}x_g\langle\vec{g},\vec{f}^{B_0}\rangle.
  \end{split}
\end{equation}
Denote the multiple integral on the right-hand side of \eqref{eq:TF}
by $I(\mathbf{w})$, and divide it as
$$
I(\mathbf{w})=\sum_{\mathbf{m}\in\mathbb{Z}^r}
\int_{\mathcal{Q}(\mathbf{w},\mathbf{m})}.
$$
Applying \eqref{7-11}, we obtain
\begin{equation}
  \begin{split}
  I(\mathbf{w})&=\sum_{\mathbf{m}\in\mathbb{Z}^r}
    \exp\Bigl(
    \sum_{f\in B_0}
    (t_f-2\pi\sqrt{-1}\const{f}) 
    \langle\mathbf{y}+\mathbf{w}+\sum_{h\in B_0}m_h\vec{h}, 
    \vec{f}^{B_0}\rangle
    \Bigr)\\
    &\times\int_{\mathcal{Q}(\mathbf{w},\mathbf{m})}
    \exp\Bigl(\sum_{g\in L_0}t_g^* x_g\Bigr)\prod_{g\in L_0}dx_g.
  \end{split}
\end{equation}  
Note 
that $\mathbf{w}+\sum_{h\in B_0}m_h\vec{h}$
runs over $\mathbb{Z}^r$, 
when
$\mathbf{w}\in\mathbb{Z}^r/\langle \vec{B}_0\rangle$
 and
$\mathbf{m}\in\mathbb{Z}^r$
run. 
Therefore, rewriting $\mathbf{w}+\sum_{h\in B_0}m_h\vec{h}$
as $\mathbf{m}$, we obtain the assertion of the proposition.
\end{proof}

\begin{remark}
\label{rem:fig1}
For readers' convenience, we give typical pictures of $\mathcal{P}(\mathbf{m};\mathbf{y})$ in the cases $\mathbf{y}\in \mathfrak{H}_{\mathscr{R}}$ and
$\mathbf{y}\notin \mathfrak{H}_{\mathscr{R}}$, which will be treated below in Lemmas
\ref{lm:vert_form},
\ref{lm:cond_inter_uniq},
\ref{lm:char_vert2} and
\ref{lm:simplicity}.

Let $V=\mathbb{R}^2$ ($r=2$).
Let $\mathbf{e}_1=
\begin{pmatrix}
  1\\0
\end{pmatrix},
\mathbf{e}_2=
\begin{pmatrix}
  0\\1
\end{pmatrix}$ and
\begin{align}
  \Lambda&=\{f_1=(\mathbf{e}_1,\alpha_1),f_2=(\mathbf{e}_2,\alpha_2),g=(a\mathbf{e}_1+b\mathbf{e}_2,\alpha_g),h=(c\mathbf{e}_1+d\mathbf{e}_2,\alpha_h)\}=B_0\cup L_0,\\
  B_0&=\{(\mathbf{e}_1,\alpha_1),(\mathbf{e}_2,\alpha_2)\},\\
  L_0&=\{(a\mathbf{e}_1+b\mathbf{e}_2,\alpha_g),(c\mathbf{e}_1+d\mathbf{e}_2,\alpha_h)\},\\
  \mathfrak{H}_{\mathscr{R}}&=(\mathbb{R}\mathbf{e}_1+\mathbb{Z}^2)\cup
(\mathbb{R}\mathbf{e}_2+\mathbb{Z}^2)\cup
(\mathbb{R}(a\mathbf{e}_1+b\mathbf{e}_2)+\mathbb{Z}^2)\cup
(\mathbb{R}(c\mathbf{e}_1+d\mathbf{e}_2)+\mathbb{Z}^2),
\end{align}
where $a,b,c,d$ are positive integers,
which corresponds to the series
\begin{equation}
  S(\mathbf{k},\mathbf{y};\Lambda)=\sum_{n_1,n_2}\frac{e^{2\pi\sqrt{-1}(n_1y_1+n_2y_2)}}{(n_1+\alpha_1)^{k_1}(n_2+\alpha_2)^{k_2}(an_1+bn_2+\alpha_g)^{k_g}(cn_1+dn_2+\alpha_h)^{k_h}},
\end{equation}
where 
$\alpha_1,\alpha_2,\alpha_g,\alpha_h\in\mathbb{C}$,
$k_1,k_2,k_g,k_h\geq2$ and
 $n_1,n_2$ run over all integers such that the denominator does not vanishes.
Then we have
\begin{equation}
  \mathcal{P}(\mathbf{m};\mathbf{y})=
  \left\{\mathbf{x}=(x_g,x_h)~\middle|~
  \begin{aligned}
  &0\leq x_g\leq1,\\
  &0\leq x_h\leq1,\\
  & y_1+m_1-1\leq ax_g+cx_h\leq y_1+m_1,\\
  & y_2+m_2-1\leq bx_g+dx_h\leq y_2+m_2
\end{aligned}
\right\}.
\end{equation}
In the case $a=b=c=1$, $d=2$,
the polytope $\mathcal{P}(\mathbf{m};\mathbf{y})$ is drawn as in Figure \ref{fig:pmy2} if
$\mathbf{y}\in \mathfrak{H}_{\mathscr{R}}$ and
in Figure \ref{fig:pmy1} if
$\mathbf{y}\notin \mathfrak{H}_{\mathscr{R}}$.
In the former case, 
there are more than $2$ ($=\sharp\Lambda-r$) hyperplanes at some vertices while
in the latter case,
there are only $2$ hyperplanes at each vertex, which ensures that $\mathcal{P}(\mathbf{m};\mathbf{y})$ is a simple polytope in higher dimensions.
  \begin{figure}[h]
    \begin{minipage}{0.45\linewidth}
      \includegraphics[scale=0.8]{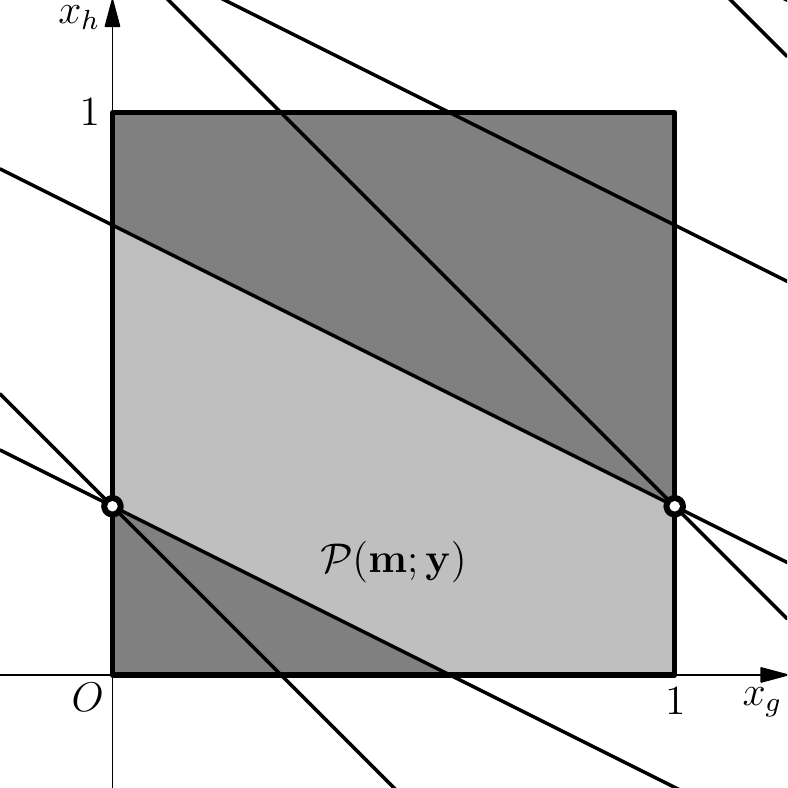}
      \caption{$\mathbf{y}\in \mathfrak{H}_{\mathscr{R}}$}
      \label{fig:pmy2}
    \end{minipage}
    \begin{minipage}{0.45\linewidth}
      \includegraphics[scale=0.8]{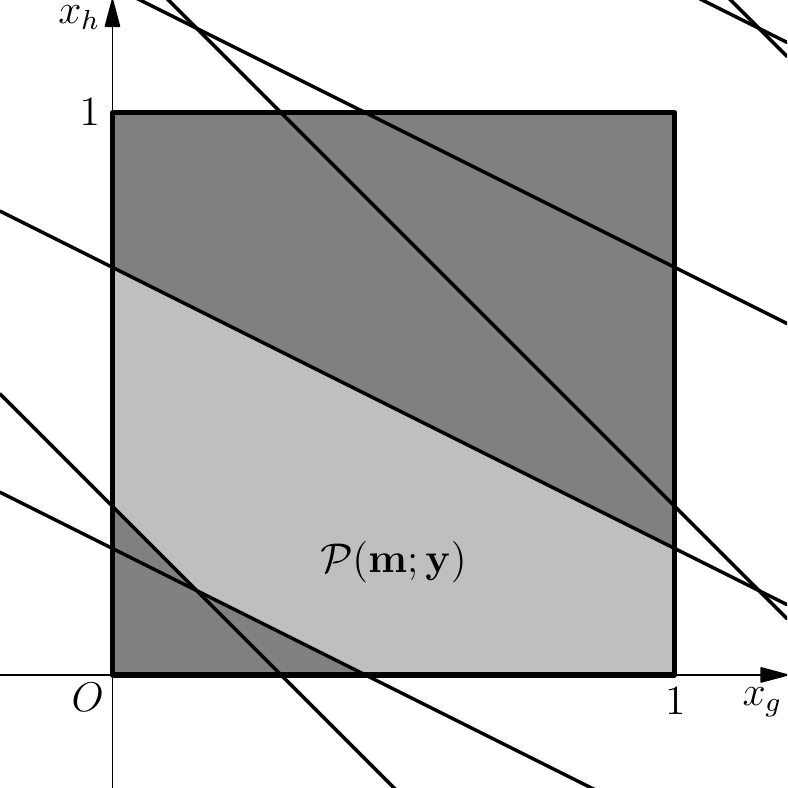}
      \caption{$\mathbf{y}\notin \mathfrak{H}_{\mathscr{R}}$}
      \label{fig:pmy1}
    \end{minipage}
  \end{figure}
\end{remark}

%%%%%%%%%%%%%%%%%%%%%%%%%%%%%%%%%%%%%%%%%%%%%%%%%%%%%%%%%%%%%%%%%%%%%%%%%%%%%%%%%%%%%%%%%%%%%%%%%
\section{Properties of the polytopes $\mathcal{P}(\mathbf{m};\mathbf{y})$}\label{sec-8}
%%%%%%%%%%%%%%%%%%%%%%%%%%%%%%%%%%%%%%%%%%%%%%%%%%%%%%%%%%%%%%%%%%%%%%%%%%%%%%%%%%%%%%%%%%%%%%%%%%

The argument developed in this and the next sections is a generalization of that in
\cite[Section 6]{KM5}.

Let $\mathscr{A}=\{0,1\}^{\sharp\Lambda-r}$.
Let $\mathscr{B}'$ be the set of all subsets of $\Lambda$ which have
$r$ elements, and define $\mathscr{W}'=\mathscr{B}'\times\mathscr{A}$
and $\mathscr{W}=\mathscr{B}\times\mathscr{A}$. 
Obviously $\mathscr{W}\subset\mathscr{W}'$.
For an element $W=(B,A)\in\mathscr{W}'$,
we number $A=(a_g)_{g\in\Lambda\setminus B}$.
We fix a decomposition $\Lambda=B_0\cup L_0$ with $B_0\in\mathscr{B}$,
and for $f\in \Lambda$, $a\in\{0,1\}$, $\mathbf{m}\in\mathbb{Z}^{r}$ and 
$\mathbf{y}\in V$,
we define $\mathbf{u}(f,a)\in\mathbb{R}^{\sharp L_0}$ by
\begin{equation}
  \mathbf{u}(f,a)_g
  =
  \begin{cases}
    (-1)^{1-a}\langle\vec{g},\vec{f}^{B_0}\rangle\qquad &\text{if }f\in B_0,\\
    (-1)^a\delta_{g f}\qquad &\text{if }f\not\in B_0,
  \end{cases}
\end{equation}
where $g$ runs over $L_0$,
and define $v(f,a;\mathbf{m};\mathbf{y})\in\mathbb{R}$ by
\begin{equation}
  v(f,a;\mathbf{m};\mathbf{y})
  =
  \begin{cases}
    (-1)^{1-a}
\langle \mathbf{y}+\mathbf{m}-a\vec{f},\vec{f}^{B_0}\rangle
%(\{\langle \mathbf{y},\vec{f}^{B_0}\rangle\}+m_f-a)
    \qquad &\text{if }f\in B_0,\\
    (-1)^aa=-a\qquad &\text{if }f\not\in B_0.
  \end{cases}
\end{equation}
Further we define the hyperplanes 
\begin{equation}
\label{8-3}
  \mathcal{H}(f,a;\mathbf{m};\mathbf{y})=
    \{\mathbf{x}=(x_g)_{g\in L_0}\in\mathbb{R}^{\sharp L_0}~|~\mathbf{u}(f,a)\cdot\mathbf{x}=v(f,a;\mathbf{m};\mathbf{y})\},
\end{equation}
and the half-spaces
\begin{equation}
  \mathcal{H}^+(f,a;\mathbf{m};\mathbf{y})=
    \{\mathbf{x}=(x_g)_{g\in L_0}\in\mathbb{R}^{\sharp L_0}~|~\mathbf{u}(f,a)\cdot\mathbf{x}\geq v(f,a;\mathbf{m};\mathbf{y})\},
\end{equation}
where for $\mathbf{w}=(w_g),\mathbf{x}=(x_g)\in\mathbb{C}^{\sharp L_0}$, we have set
\begin{equation}
  \mathbf{w}\cdot\mathbf{x}=\sum_{g\in L_0}w_g x_g.
\end{equation}
Then we have
\begin{equation}
\label{eq:P_ineq}
  \mathcal{P}(\mathbf{m};\mathbf{y})=
  \bigcap_{\substack{f\in \Lambda\\a\in\{0,1\}}}\mathcal{H}^+(f,a;\mathbf{m};\mathbf{y}).
\end{equation}
\begin{lemma}
\label{lm:vert_form}
All vertices of $\mathcal{P}(\mathbf{m};\mathbf{y})$ are of the form
\begin{equation}
\label{eq:vert_form}
  \bigcap_{f\in L'}\mathcal{H}(f,a_f;\mathbf{m};\mathbf{y}),
\end{equation}
where $L'\subset \Lambda$ with $\sharp L'=\sharp L_0=\sharp\Lambda-r$ 
and $(a_f)_{f \in L'}\in\mathscr{A}$.
\end{lemma}
\begin{proof}
By Proposition \ref{prop:k_face_2},
we see that any vertex of $\mathcal{P}(\mathbf{m};\mathbf{y})$ is obtained as the intersection of 
$(\sharp L_0)$ hyperplanes.
Since for $f\in \Lambda$,
two hyperplanes $\mathcal{H}(f,a;\mathbf{m};\mathbf{y})$ ($a=0,1$)
are parallel and hence their intersection is empty,
we see that a vertex must be of the form \eqref{eq:vert_form}.
\end{proof}

Since $\sharp L'=\sharp\Lambda-r$, we have $B=\Lambda\setminus L'\in\mathscr{B}'$.
Therefore Lemma \ref{lm:vert_form} implies that any vertex of $\mathcal{P}(\mathbf{m};\mathbf{y})$
determines an element $(B,A)\in\mathscr{W}'$.    The next lemma is a kind of converse assertion.

\begin{lemma}
\label{lm:cond_inter_uniq}
Let $(B,A)\in\mathscr{W}'$ and $L'=\Lambda\setminus B$.
The set
\begin{equation}
\label{eq:def_p}
  \bigcap_{g\in L'}\mathcal{H}(g,a_g;\mathbf{m};\mathbf{y})
\end{equation}
consists of only one point, which we denote by $\mathbf{p}(\mathbf{m};\mathbf{y};W)$,
if and only if $W=(B,A)\in\mathscr{W}$.
\end{lemma}
\begin{proof}
Let $B=\{f_1,\ldots,f_r\}\in\mathscr{B}'$ and $a_f\in\{0,1\}$ for $f\in L'=\Lambda\setminus B$. 
Consider the intersection of $(\sharp\Lambda-r)$ hyperplanes \eqref{eq:def_p}.
Then this set consists of the solutions of the system of the $(\sharp L_0)$ linear equations
\begin{equation}
\label{eq:eq_vert}
  \begin{cases}
    \sum_{g\in L_0}x_g\langle\vec{g},\vec{f}^{B_0}\rangle
    =
\langle \mathbf{y}+\mathbf{m}-a_f\vec{f},\vec{f}^{B_0}\rangle
%\{\langle \mathbf{y},\vec{f}^{B_0}\rangle\}+m_f-a_{f}
\qquad&\text{for }f\in B_0\setminus B,\\
    x_f=a_f\qquad&\text{for }f\in L_0\setminus B.
  \end{cases}
\end{equation}
% Let $J=B\setminus B_0$ and $J^c=B\cap B_0$.
% Let $K=B_0\setminus B$ and $K^c=B_0\cap B$.
% % Let $J=\{f~|~f\in B\setminus B_0\}$ and $J^c=\{1,\ldots,r\}\setminus J$.
% % Let $K=\{k~|~\beta_k\in B_0\setminus B\}$ and $K^c=\{1,\ldots,r\}\setminus K$.
% Note that $\sharp{J}=\sharp{K}\leq r$ and $J^c=K^c=B\cap B_0$.
The system of the linear equations \eqref{eq:eq_vert} has a unique solution
if and only if
\begin{equation}
\label{eq:det_nz}
\det
(\langle\vec{g},\vec{f}^{B_0}\rangle)_{f\in B_0\setminus B}^{g\in B\setminus B_0}
\neq0,
\end{equation}
% \begin{equation}
% \label{eq:det_nz}
% \det
% (\langle\vec{g},\vec{f}^{B_0}\rangle)_{f\in K}^{g\in J}
% \neq0,
% \end{equation}
and hence if and only if
$B\in\mathscr{B}$,
since
\begin{equation}
\label{eq:detgf}
  \begin{split}
\Biggl\lvert
\det
\begin{pmatrix}
(\langle \vec{g},\vec{f}^{B_0}\rangle)_{f\in B_0\setminus B}^{g\in B\setminus B_0}&
(\langle \vec{g},\vec{f}^{B_0}\rangle)_{f\in B_0\cap B}^{g\in B\setminus B_0}\\
(\langle \vec{g},\vec{f}^{B_0}\rangle)_{f\in B_0\setminus B}^{g\in B\cap B_0}&
(\langle \vec{g},\vec{f}^{B_0}\rangle)_{f\in B_0\cap B}^{g\in B\cap B_0}
\end{pmatrix}
\Biggr\rvert
&=
\Biggl\lvert
\det
\begin{pmatrix}
(\langle\vec{g},\vec{f}^{B_0}\rangle)_{f\in B_0\setminus B}^{g\in B\setminus B_0}
 & * \\
0 & E_{\sharp(B\cap B_0)}
\end{pmatrix}
\Biggr\rvert
\\
&=
\bigl\lvert
\det
(\langle\vec{g},\vec{f}^{B_0}\rangle)_{f\in B_0\setminus B}^{g\in B\setminus B_0}
\bigr\rvert,
\end{split}
\end{equation}
where $E_p$ is the $p\times p$ identity matrix.  
\end{proof}

\begin{lemma}
\label{lm:char_vert2}
Let $W=(B,A)\in\mathscr{W}$.
The point
$\mathbf{p}(\mathbf{m};\mathbf{y};W)$ is a vertex of $\mathcal{P}(\mathbf{m};\mathbf{y})$
if and only if
\begin{equation}
\label{eq:def_v}
\mathbf{p}(\mathbf{m};\mathbf{y};W)\in
\bigcap_{\substack{f\in B\\a\in\{0,1\}}}\mathcal{H}^+(f,a;\mathbf{m};\mathbf{y}).
\end{equation}
\end{lemma}
\begin{proof}
By \eqref{eq:P_ineq}, we see that $\mathcal{P}(\mathbf{m};\mathbf{y})$ is defined 
by $(\sharp\Lambda)$ pairs of inequalities.
By Proposition \ref{prop:k_face_1}, 
the point $\mathbf{p}(\mathbf{m};\mathbf{y};W)$ is a vertex of $\mathcal{P}(\mathbf{m};\mathbf{y})$
if and only if all of these inequalities hold.    
We see that $(\sharp L_0)$ pairs among them are automatically satisfied,
because
\begin{equation}
\label{8-14}
\{\mathbf{p}(\mathbf{m};\mathbf{y};W)\}=
  \bigcap_{g\in \Lambda\setminus B}\mathcal{H}(g,a_g;\mathbf{m};\mathbf{y})
\subset
  \bigcap_{\substack{g\in \Lambda\setminus B\\a\in\{0,1\}}}\mathcal{H}^+(g,a;\mathbf{m};\mathbf{y}).
\end{equation}
Therefore 
$\mathbf{p}(\mathbf{m};\mathbf{y};W)$
is a vertex of $\mathcal{P}(\mathbf{m};\mathbf{y})$ 
% $\mathbf{p}(\mathbf{m};\mathbf{y};\mathbf{W})\in\Vrt(\mathcal{P}(\mathbf{m};\mathbf{y}))$ 
if and only if the remaining $r$ pairs of
inequalities are satisfied, which implies
\eqref{eq:def_v}.
\end{proof}

\begin{lemma}
\label{lm:simplicity}
If $\mathbf{y}\in V\setminus\mathfrak{H}_{\mathscr{R}}$ and $\mathcal{P}(\mathbf{m};\mathbf{y})$ is not empty, then 
$\mathcal{P}(\mathbf{m};\mathbf{y})$ is a simple polytope.
\end{lemma}
\begin{proof}
By Lemmas \ref{lm:vert_form} and \ref{lm:cond_inter_uniq},
it is sufficient to check the following claim:
{\it If for $W=(B,A)\in\mathscr{W}$, the point $\mathbf{p}(\mathbf{m};\mathbf{y};W)$
lies on some other hyperplanes of the form \eqref{8-3} than the defining hyperplanes
$\{\mathcal{H}(g,a_g;\mathbf{m};\mathbf{y})\}_{g\in \Lambda\setminus B}$,
then $\mathbf{y}\in\mathfrak{H}_{\mathscr{R}}$.}
Because this claim implies that if $\mathbf{y}\in V\setminus\mathfrak{H}_{\mathscr{R}}$,
we can uniquely determine the $(\sharp L_0)$ hyperplanes
on which the point
$\mathbf{p}(\mathbf{m};\mathbf{y};W)$ lies, and hence
it implies the simplicity of the polytope $\mathcal{P}(\mathbf{m};\mathbf{y})$.

Since
\begin{equation}
  \mathbf{p}(\mathbf{m};\mathbf{y};W)\not\in
  \bigcup_{g\in \Lambda\setminus B}\mathcal{H}(g,1-a_g;\mathbf{m};\mathbf{y})
\end{equation}
always holds,
it is sufficient to check that
\begin{equation}
\label{eq:H_y_W}
 \mathbf{p}(\mathbf{m};\mathbf{y};W)\in
  \bigcup_{\substack{f\in B\\a\in\{0,1\}}}\mathcal{H}(f,a;\mathbf{m};\mathbf{y})
\end{equation}
implies  $\mathbf{y}\in\mathfrak{H}_{\mathscr{R}}$.

First we show the claim when
\begin{equation}
\label{eq:v_on_H} 
\mathbf{p}(\mathbf{m};\mathbf{y};W)\in\mathcal{H}(h,a_h;\mathbf{m};\mathbf{y})
\end{equation}
holds for some $h\in B\cap B_0$ and $a_h\in\{0,1\}$.
For $\mathbf{x}=\mathbf{p}(\mathbf{m};\mathbf{y};W)$,
condition \eqref{eq:v_on_H} is equivalent to 
\begin{equation}
  \label{eq:v_on_H_equiv}
  \sum_{g\in L_0}x_g\langle\vec{g},\vec{h}^{B_0}\rangle
  =
\langle \mathbf{y}+\mathbf{m}-a_h\vec{h},\vec{h}^{B_0}\rangle.
%\{\langle \mathbf{y},\vec{h}^{B_0}\rangle\}+m_h-a_h.
\end{equation}
Let $p=\sharp(B_0\setminus B)=\sharp(B\setminus B_0)$.
% Let $J=\{j~|~\gamma_j\in\mathbf{V}\setminus B\}=\{j_1,\ldots,j_p\}$ and 
% $K=\{k~|~\beta_k\in B\setminus\mathbf{V}\}=\{k_1,\ldots,k_p\}$ as in the proof of Lemma \ref{lm:cond_inter_uniq}.
Divide the left-hand side of (the first formula of) \eqref{eq:eq_vert} and \eqref{eq:v_on_H_equiv}
into two parts according to the conditions $g\in B\setminus B_0$ and $g\in L_0\setminus B$
(with noting $L_0=(B\setminus B_0)\cup(L_0\setminus B)$).
Then we obtain an overdetermined system with the $p$ variables $x_g$ for $g\in B\setminus B_0$
and the $(p+1)$ equations
\begin{equation}
  \sum_{g\in B\setminus B_0}x_{g}\langle \vec{g},\vec{f}^{B_0}\rangle
  =
\langle \mathbf{y}+\mathbf{m}-a_f\vec{f},\vec{f}^{B_0}\rangle+c_f
%\{\langle \mathbf{y},\vec{f}^{B_0}\rangle\}+c_f
\end{equation}
for $f\in (B_0\setminus B)\cup \{h\}$,
where
\begin{equation}
  c_f=
  -\sum_{g\in L_0\setminus B}a_g\langle\vec{g},\vec{f}^{B_0}\rangle.
\end{equation}
Hence we have
\begin{equation}
  \begin{pmatrix}
  (x_{g})_{g\in B\setminus B_0} & -1
  \end{pmatrix}
  M(\mathbf{y})  =
  \begin{pmatrix}
    0 &\cdots& 0
  \end{pmatrix},
\end{equation}
where $(x_{g})_{g\in B\setminus B_0}$ is a row vector and $M(\mathbf{y})$ is a
$(p+1)\times(p+1)$ matrix defined by
\begin{equation}
  M(\mathbf{y})  =
  \begin{pmatrix}
  (\langle\vec{g},\vec{f}^{B_0}\rangle)^{g\in B\setminus B_0}_{f\in (B_0\setminus B)\cup \{h\}}
  \\
  (
\langle \mathbf{y}+\mathbf{m}-a_f\vec{f},\vec{f}^{B_0}\rangle+c_f
%\{\langle \mathbf{y},\vec{f}^{B_0}\rangle\}+c_{f}
)_{f\in (B_0\setminus B)\cup \{h\}}
  \end{pmatrix}.
\end{equation}
As the consistency for these equations, we get $\det M(\mathbf{y})=0$.
We may rewrite %ing %with some $p_e\in\mathbb{Z}$ for $e\in(B_0\setminus B)\cup \{h\}$
\begin{equation}
\langle \mathbf{y}+\mathbf{m}-a_f\vec{f},\vec{f}^{B_0}\rangle+c_f
=
%  \{\langle \mathbf{y},\vec{f}^{B_0}\rangle\}+c_{f}=
  \Bigl\langle \mathbf{y}+
\mathbf{m}
%\sum_{e\in(B_0\setminus B)\cup \{h\}}(m_e-a_{e}+p_e)\vec{e}
  -\sum_{g\in (\Lambda\setminus B)\cup\{h\}}a_g\vec{g},\vec{f}^{B_0}\Bigr\rangle,
%  -\sum_{g\in (L_0\setminus B)\cup\{h\}}a_g\vec{g},\vec{f}^{B_0}\Bigr\rangle,
%  -\sum_{g\in L_0\setminus B}a_g\vec{g},\vec{f}^{B_0}\Bigr\rangle,
\end{equation}
because
\begin{equation*}
\sum_{g\in (\Lambda\setminus B)\cup\{h\}}a_g\vec{g}-
\left(\sum_{g\in L_0\setminus B}a_g\vec{g}+a_f\vec{f}\right)
=\sum_{\substack{g\in (B_0\setminus B)\cup\{h\}\\ g\neq f}}a_g\vec{g}
\end{equation*}
and $\langle\vec{g},\vec{f}^{B_0}\rangle=0$ for all $\vec{g}$ on the right-hand side.

Since the row vectors 
$(\langle\vec{g},\vec{f}^{B_0}\rangle)_{f\in (B_0\setminus B)\cup \{h\}}$
for $g\in B\setminus B_0$ are linearly independent, 
$\det M(\mathbf{y})=0$ implies that the last row vector is written as a linear combination of the other row vectors.   That is,
for $f\in (B_0\setminus B)\cup \{h\}$ we have
\begin{equation}
    \Bigl\langle \mathbf{y}+
\mathbf{m}
  -\sum_{g\in (\Lambda\setminus B)\cup\{h\}}a_g\vec{g},\vec{f}^{B_0}\Bigr\rangle=\sum_{g\in B\setminus B_0} q_g \langle\vec{g},\vec{f}^{B_0}\rangle,
\end{equation}
with some $q_g\in\mathbb{R}$, and so
\begin{equation}
    \Bigl\langle \mathbf{y}+
\mathbf{m}
  -\sum_{g\in (\Lambda\setminus B)\cup\{h\}}a_g\vec{g}-\sum_{g\in B\setminus B_0} q_g \vec{g},\vec{f}^{B_0}\Bigr\rangle=0.
\end{equation}
Vectors which are orthogonal to all $\vec{f}^{B_0}$ 
($f\in (B_0\setminus B)\cup\{h\}$) are spanned by $\vec{g}$
($g\in (B\cap B_0)\setminus\{h\}$).    Therefore, 
since $\mathbf{m}
  -\sum_{g\in (\Lambda\setminus B)\cup\{h\}}a_g\vec{g}\in\mathbb{Z}^r$,
%we find that $\det M(\mathbf{y})=0$ 
we have
\begin{equation}
  \begin{split}
    \mathbf{y}
    \in
    \sum_{g\in B\setminus B_0}\mathbb{R}\vec{g}
    +
    \mathbb{Z}^r
    +
    \sum_{g\in (B\cap B_0)\setminus \{h\}}\mathbb{R}\vec{g}
    % \sum_{g\in B\setminus B_0}\mathbb{R}\vec{g}
    % +\sum_{g\in(B_0\setminus B)\cup \{h\}}
    % \mathbb{Z}\vec{g}-\sum_{g\in L_0\setminus B}a_g\vec{g}
    % +
    % \sum_{g\in (B\cap B_0)\setminus \{h\}}\mathbb{R}\vec{g}
    =
    \sum_{g\in B\setminus\{h\}}\mathbb{R}\vec{g}
    +
    \mathbb{Z}^r
    % \sum_{g\in B\setminus\{h\}}\mathbb{R}\vec{g}
    % +\sum_{g\in (\Lambda\setminus B)\cup \{h\}}
    % \mathbb{Z}\vec{g}
        \subset\mathfrak{H}_{\mathscr{R}},
  \end{split}
\end{equation}
which implies the desired claim.
%and $\mathbf{y}$.

Next we consider 
the condition 
$\mathbf{p}(\mathbf{m};\mathbf{y};W)\in\mathcal{H}(h,a_h;\mathbf{m};\mathbf{y})$
for some $h\in B\setminus B_0$ and $a_h\in\{0,1\}$.
Then similarly as above, we see that $\mathbf{y}$ lies in $\mathfrak{H}_{\mathscr{R}}$ because
\begin{equation}
\det
  \begin{pmatrix}
  (\langle\vec{g},\vec{f}^{B_0}\rangle)^{g\in B\setminus (B_0\cup\{h\})}_{f\in B_0\setminus B}
  \\
  (
\langle \mathbf{y}+\mathbf{m}-a_f\vec{f},\vec{f}^{B_0}\rangle+d_f
%\{\langle \mathbf{y},\vec{f}^{B_0}\rangle\}+d_{f}
)_{f\in B_0\setminus B}
  \end{pmatrix}
  =0,
\end{equation}
where
\begin{equation}
%  d_f=m_f-a_{f}-\sum_{g\in (L_0\setminus B)\cup\{h\}}a_g\langle\vec{g},\vec{f}^{B_0}\rangle.
  d_f=-\sum_{g\in (L_0\setminus B)\cup\{h\}}a_g\langle\vec{g},\vec{f}^{B_0}\rangle.
\end{equation}
This completes the proof of the lemma.
\end{proof}

\begin{remark}
\label{rem:fig2}
We draw the picture of a vertex of
$\mathcal{P}(\mathbf{m};\mathbf{y})$ in the 
same setting as in Figure \ref{fig:pmy1}.
%the latter case in Remark \ref{rem:fig1}.
%$\mathbf{y}\notin \mathfrak{H}_{\mathscr{R}}$.
For example,
for $W=(B,A)\in\mathscr{B}$ with
$B=\{f_2,h\}\in\mathscr{B}$ and
$A=(a_{f_1},a_g)=(1,0)$ ($\Lambda\setminus B=\{f_1,g\}$), the associated vertex is 
$\mathbf{p}(\mathbf{m};\mathbf{y};W)=\mathbf{p}(\mathbf{m};\mathbf{y};(\{f_2,h\},(1,0)))$, which is uniquely determined by the hyperplanes
$\mathcal{H}(f_1,1;\mathbf{m};\mathbf{y})$
and
$\mathcal{H}(g,0;\mathbf{m};\mathbf{y})$. See Figure \ref{fig:pmy3}.
  \begin{figure}[h]
    \centering
      \includegraphics[scale=0.8]{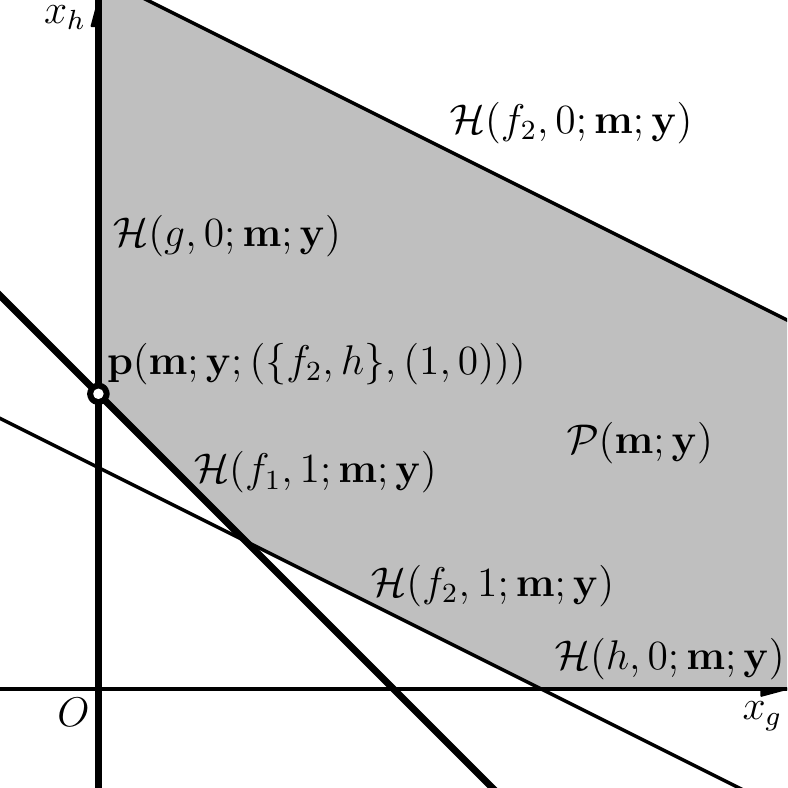}
      \caption{vertices and hyperplanes}
      \label{fig:pmy3}
    \end{figure}
\end{remark}

\begin{lemma}
\label{lm:exponential}
Let $\mathbf{y}\in V\setminus\mathfrak{H}_{\mathscr{R}}$ and $W\in\mathscr{W}$.
Then we have
\begin{multline}
  \label{eq:exp_vert}
  \sum_{f\in B_0}
  (t_f-2\pi\sqrt{-1} \const{f})
  \langle \mathbf{y}+\mathbf{m},\vec{f}^{B_0}\rangle+
  \mathbf{t}^*\cdot\mathbf{p}(\mathbf{m};\mathbf{y};W)
  \\
  =
  \sum_{g\in \Lambda\setminus B}
  (t_g-2\pi\sqrt{-1}\const{g}) a_g
  +\sum_{f\in B}
  (t_f-2\pi\sqrt{-1}\const{f})
  \Bigl\langle 
  \mathbf{y}+\mathbf{m}
%\sum_{g\in B_0}m_g\vec{g}^{B_0}
  -\sum_{g\in \Lambda\setminus B}a_g\vec{g},\vec{f}^{B}
  \Bigr\rangle,
%  \pmod{2\pi\sqrt{-1}\mathbb{Z}}.
\end{multline}
where $\mathbf{t}^*=(t_g^*)_{g\in L_0}$ with $t_g^*$ defined by \eqref{def_t_g^*}.
\end{lemma}
\begin{proof}
% Note that if $\mathbf{y}\in\mathfrak{D}\setminus\mathfrak{H}_{\mathscr{R}}$, then by Lemma \ref{lm:DH} we have
% \begin{equation}
%   \{\langle \mathbf{y},\vec{f}^{B_0}\rangle\}=\langle \mathbf{y},\vec{f}^{B_0}\rangle
% \end{equation}
% for all $f\in B_0$.
By \eqref{eq:def_p},
the point $\mathbf{p}(\mathbf{m};\mathbf{y};W)=(x_g)_{g\in L_0}$ satisfies
\begin{equation}
\label{8-30}
  \begin{cases}
    \sum_{g\in L_0}x_g\langle\vec{g},\vec{f}^{B_0}\rangle
    =
\langle \mathbf{y}+\mathbf{m}-a_f\vec{f},\vec{f}^{B_0}\rangle
%\langle \mathbf{y},\vec{f}^{B_0}\rangle+m_f-a_{f}
\qquad&\text{for }f\in B_0\setminus B,\\
    x_f=a_f\qquad&\text{for }f\in L_0\setminus B.
  \end{cases}
\end{equation}
By Lemma \ref{lm:cond_inter_uniq},
the system of these equations has a unique solution.

In the case $f\in B_0\setminus B$, we have
\begin{equation}
\label{eq:x_b}
  \begin{split}
  \sum_{h\in B\setminus B_0}x_{h}\langle\vec{h},\vec{f}^{B_0}\rangle
  &
  =
\langle \mathbf{y}+\mathbf{m}-a_f\vec{f},\vec{f}^{B_0}\rangle
%  \langle \mathbf{y},\vec{f}^{B_0}\rangle+m_f-a_f
  -\sum_{g\in L_0\setminus B}a_g\langle\vec{g},\vec{f}^{B_0}\rangle
  \\
  &=
  \Bigl\langle 
  \mathbf{y}+\mathbf{m}
%\sum_{g\in B_0}m_g\vec{g}
  -\sum_{g\in \Lambda\setminus B}a_g\vec{g},
  \vec{f}^{B_0}\Bigr\rangle.
\end{split}
\end{equation}
On the other hand, in the case $f\in B_0\cap B$ we have
\begin{equation}
\label{eq:mu_i}
  \vec{f}^B=\vec{f}^{B_0}-\sum_{h\in B\setminus B_0}\vec{h}^B\langle\vec{h},\vec{f}^{B_0}\rangle.
\end{equation}
In fact, since
\begin{equation}
\label{eq:id}
\mathbf{z}=\sum_{h\in B}\vec{h}^B\langle\vec{h},\mathbf{z}\rangle
% \id=\sum_{h\in B}\vec{h}^B\otimes\vec{h}
% \in V\otimes V\simeq\End(V),
\end{equation}
holds for any $\mathbf{z}\in V$, we have
\begin{equation}
  \vec{f}^{B_0}=%\id(\vec{f}^{B_0})=
\sum_{h\in B\cap B_0}\vec{h}^B\langle\vec{h},\vec{f}^{B_0}\rangle+\sum_{h\in B\setminus B_0}\vec{h}^B\langle\vec{h},\vec{f}^{B_0}\rangle
=
\vec{f}^B+\sum_{h\in B\setminus B_0}\vec{h}^B\langle\vec{h},\vec{f}^{B_0}\rangle.
\end{equation}

Here we note that for $h\in B\setminus B_0$,
\begin{equation}
\label{eq:exp_p_sol}
  x_h=
  \Bigl\langle 
  \mathbf{y}+\mathbf{m}
%\sum_{g\in B_0}m_g\vec{g}
  -\sum_{g\in \Lambda\setminus B}a_g\vec{g},
  \vec{h}^{B}\Bigr\rangle
\end{equation}
holds.
Because, using \eqref{eq:id}, 
for $f\in B_0\setminus B$ we obtain
\begin{equation}
\label{8-34}
  \begin{split}
    \sum_{h\in B\setminus B_0}
    \Bigl\langle 
    \mathbf{y}+\mathbf{m}
%\sum_{g\in B_0}m_g\vec{g}
    -\sum_{g\in \Lambda\setminus B}a_g\vec{g},
    \vec{h}^B\Bigr\rangle
    \langle\vec{h},\vec{f}^{B_0}\rangle
    &=
    \sum_{h\in B}
    \Bigl\langle 
    \mathbf{y}+\mathbf{m}
%\sum_{g\in B_0}m_g\vec{g}
    -\sum_{g\in \Lambda\setminus B}a_g\vec{g},
    \vec{h}^B\Bigr\rangle
    \langle\vec{h},\vec{f}^{B_0}\rangle
    \\
    &=
    \Bigl\langle 
    \mathbf{y}+\mathbf{m}
%\sum_{g\in B_0}m_g\vec{g}
    -\sum_{g\in \Lambda\setminus B}a_g\vec{g},
    \vec{f}^{B_0}\Bigr\rangle.
  \end{split}
\end{equation}
Comparing \eqref{eq:x_b} with \eqref{8-34}, we obtain \eqref{eq:exp_p_sol} due to the uniqueness of the solution.

Noting $L_0=(L_0\setminus B)\cup(B\setminus B_0)$, we have
\begin{multline}
  \label{eq:tp}
  \sum_{f\in B_0}
  (t_f-2\pi\sqrt{-1} \const{f})
  \langle \mathbf{y}+\mathbf{m},\vec{f}^{B_0}\rangle+
  \mathbf{t}^*\cdot\mathbf{p}(\mathbf{m};\mathbf{y};W)
  \\
  \begin{aligned}
    &=
    \sum_{f\in B_0}
    (t_f-2\pi\sqrt{-1}\const{f})
    \langle \mathbf{y}+\mathbf{m},\vec{f}^{B_0}\rangle
    +
    \sum_{g\in L_0} 
    \Bigl((t_g-2\pi\sqrt{-1}\const{g})
    -\sum_{f\in B_0}(t_f-2\pi\sqrt{-1}\const{f})\langle\vec{g},\vec{f}^{B_0}\rangle
    \Bigr)x_g
    \\
    &=
    \sum_{g\in L_0}(t_g-2\pi\sqrt{-1}\const{g}) x_g
    +\sum_{f\in B_0}
    (t_f-2\pi\sqrt{-1}\const{f})
    \Bigl(\langle\mathbf{y}+\mathbf{m},\vec{f}^{B_0}\rangle
    -\sum_{g\in L_0}x_g\langle\vec{g},\vec{f}^{B_0}\rangle\Bigr)
    \\
    &=
    \sum_{g\in L_0\setminus B}
    (t_g-2\pi\sqrt{-1}\const{g}) a_g
    +
    \sum_{h\in B\setminus B_0}
    (t_h-2\pi\sqrt{-1}\const{h})
    x_h\\
    &\qquad
    +\sum_{f\in B_0\cap B}
    (t_f-2\pi\sqrt{-1}\const{f})
    \Bigl(\langle\mathbf{y}+\mathbf{m},\vec{f}^{B_0}\rangle
    -\sum_{g\in L_0\setminus B}a_g\langle\vec{g},\vec{f}^{B_0}\rangle
    -\sum_{h\in B\setminus B_0}x_h\langle\vec{h},\vec{f}^{B_0}\rangle 
    \Bigr)
    \\
    &\qquad
    +\sum_{f\in B_0\setminus B} 
    (t_f-2\pi\sqrt{-1}\const{f})
    \Bigl(\langle\mathbf{y}+\mathbf{m},\vec{f}^{B_0}\rangle
    -\sum_{g\in L_0\setminus B}a_g\langle\vec{g},\vec{f}^{B_0}\rangle
    -\sum_{h\in B\setminus B_0}x_h\langle\vec{h},\vec{f}^{B_0}\rangle 
    \Bigr).
  \end{aligned}
\end{multline}
By the first equality of \eqref{eq:x_b}, the last term on the last member of
\eqref{eq:tp} is equal to
\begin{multline}
  \sum_{f\in B_0\setminus B} 
  (t_f-2\pi\sqrt{-1}\const{f})
  \Bigl(\langle\mathbf{y}+\mathbf{m},\vec{f}^{B_0}\rangle
  -\sum_{g\in L_0\setminus B}a_g\langle\vec{g},\vec{f}^{B_0}\rangle
  -\sum_{h\in B\setminus B_0}x_h\langle\vec{h},\vec{f}^{B_0}\rangle 
  \Bigr)
  \\
  =
  \sum_{f\in B_0\setminus B} 
  (t_f-2\pi\sqrt{-1}\const{f})
  \langle a_f\vec{f},\vec{f}^{B_0}\rangle
  =
  \sum_{f\in B_0\setminus B} 
  (t_f-2\pi\sqrt{-1}\const{f})
  a_f.
\end{multline}
On the other hand, for $f\in B_0\cap B$, we have
\begin{multline}
  \langle\mathbf{y}+\mathbf{m},\vec{f}^{B_0}\rangle
  -\sum_{g\in L_0\setminus B}a_g\langle\vec{g},\vec{f}^{B_0}\rangle
  -\sum_{h\in B\setminus B_0}x_h\langle\vec{h},\vec{f}^{B_0}\rangle 
  \\
  \begin{aligned}
    &=
    \Bigl\langle 
    \mathbf{y}+\mathbf{m}
%\sum_{g\in B_0}m_g\vec{g}
    -\sum_{g\in L_0\setminus B}a_g\vec{g},\vec{f}^{B_0}
    \Bigr\rangle
    -\sum_{h\in B\setminus B_0}
    \Bigl\langle 
    \mathbf{y}+\mathbf{m}
%\sum_{g\in B_0}m_g\vec{g}
    -\sum_{g\in \Lambda\setminus B}a_g\vec{g},
  \vec{h}^{B}\Bigr\rangle
  \langle\vec{h},\vec{f}^{B_0}\rangle 
  \\
    &=
    \Bigl\langle 
    \mathbf{y}+\mathbf{m}
%\sum_{g\in B_0}m_g\vec{g}
    -\sum_{g\in \Lambda\setminus B}a_g\vec{g},\vec{f}^{B_0}
    -\sum_{h\in B\setminus B_0}
    \vec{h}^{B}
    \langle\vec{h},\vec{f}^{B_0}\rangle 
    \Bigr\rangle
    \\
    &=
    \Bigl\langle 
    \mathbf{y}+\mathbf{m}
%\sum_{g\in B_0}m_g\vec{g}
    -\sum_{g\in \Lambda\setminus B}a_g\vec{g},\vec{f}^B
    \Bigr\rangle,
  \end{aligned}
\end{multline}
by \eqref{eq:exp_p_sol} and \eqref{eq:mu_i}.
Therefore
we finally obtain \eqref{eq:exp_vert}.
\end{proof}

\begin{lemma}
\label{lm:char_vert}
Let $\mathbf{y}\in V\setminus\mathfrak{H}_{\mathscr{R}}$ and $W\in\mathscr{W}$.
Then the point $\mathbf{p}(\mathbf{m};\mathbf{y};W)$ is a vertex of 
$\mathcal{P}(\mathbf{m};\mathbf{y})$
if and only if
\begin{equation}
\label{0leqleq1}
  0\leq
  \Bigl\langle 
  \mathbf{y}+\mathbf{m}
%\sum_{g\in B_0}m_g\vec{g}
  -\sum_{g\in \Lambda\setminus B}a_g\vec{g},\vec{f}^{B}
  \Bigr\rangle
  \leq1
\end{equation}
for all $f\in B$.
\end{lemma}
\begin{proof}
%  By \eqref{eq:def_v},
By Lemma \ref{lm:char_vert2},
the point $\mathbf{p}(\mathbf{m};\mathbf{y};W)=(x_g)_{g\in L_0}$ is indeed a vertex
if and only if
\begin{equation}
\label{eq:ineq_vert}
  \begin{cases}
    \langle \mathbf{y}+\mathbf{m}-\vec{f},\vec{f}^{B_0}\rangle\leq
%    \langle \mathbf{y},\vec{f}^{B_0}\rangle+m_f-1\leq
    \sum_{g\in L_0}x_g\langle\vec{g},\vec{f}^{B_0}\rangle
    \leq\langle \mathbf{y}+\mathbf{m},\vec{f}^{B_0}\rangle,\qquad&\text{for }f\in B\cap B_0,\\
%    \leq\langle \mathbf{y},\vec{f}^{B_0}\rangle+m_f,\qquad&\text{for }f\in B\cap B_0,\\
    0\leq x_f\leq1,\qquad&\text{for }f\in B\setminus B_0.
  \end{cases}
\end{equation}
For $f\in B\cap B_0$, applying \eqref{8-30} (the second equality) and \eqref{eq:exp_p_sol},
we have
\begin{equation}
  \begin{split}
    \sum_{g\in L_0}x_g\langle\vec{g},\vec{f}^{B_0}\rangle
    &=
    \sum_{g\in L_0\setminus B}x_g\langle\vec{g},\vec{f}^{B_0}\rangle
    +\sum_{h\in B\setminus B_0}x_h\langle\vec{h},\vec{f}^{B_0}\rangle
    \\
    &=
    \Bigl\langle
    \sum_{g\in \Lambda\setminus B}a_g\vec{g},
    \vec{f}^{B_0}\Bigr\rangle
    +
    \Bigl\langle 
    \mathbf{y}+\mathbf{m}
%    \sum_{g\in B_0}m_g\vec{g}
    -\sum_{g\in \Lambda\setminus B}a_g\vec{g},
    \sum_{h\in B\setminus B_0}
    \vec{h}^B\langle\vec{h},\vec{f}^{B_0}\rangle
    \Bigr\rangle.
  \end{split}
\end{equation}
Therefore, noting \eqref{eq:mu_i}, we see that the first pair of inequalities of 
\eqref{eq:ineq_vert} is 
\begin{equation*}
\Bigl\langle \mathbf{y}+\mathbf{m}-\vec{f}-\sum_{g\in\Lambda\setminus B}a_g\vec{g},\vec{f}^{B_0}\Bigr\rangle\leq
\Bigl\langle \mathbf{y}+\mathbf{m}-\sum_{g\in\Lambda\setminus B}a_g\vec{g},\vec{f}^{B_0}-\vec{f}^B\Bigr\rangle\leq
\Bigl\langle \mathbf{y}+\mathbf{m}-\sum_{g\in\Lambda\setminus B}a_g\vec{g},\vec{f}^{B_0}\Bigr\rangle,
\end{equation*}
which is equivalent to
\begin{equation}
\label{8-43}
0\leq\Bigl\langle 
  \mathbf{y}+\mathbf{m}
%    \sum_{g\in B_0}m_g\vec{g}
    -\sum_{g\in \Lambda\setminus B}a_g\vec{g},
  \vec{f}^B\Bigr\rangle\leq1.
\end{equation}
For $f\in B\setminus B_0$, noting \eqref{eq:exp_p_sol} we see that the 
second pair of inequalities of \eqref{eq:ineq_vert} is again of the same form as \eqref{8-43}.
Therefore the desired assertion follows.
%\eqref{eq:exp_p_sol}, we finally obtain
%the condition equivalent to \eqref{eq:ineq_vert}
%\begin{equation}
%0\leq\Bigl\langle 
%  \mathbf{y}+\mathbf{m}
%  \sum_{g\in B_0}m_g\vec{g}
%  -\sum_{g\in \Lambda\setminus B}a_g\vec{g},
%  \vec{f}^B\Bigr\rangle\leq1
%  \end{equation}
%for all $f\in B$.
\end{proof}

Fix $W=(B,A)\in\mathscr{W}$.
Let $U$ be the $(\sharp L_0)\times(\sharp L_0)$ matrix
whose $f$-th column consists of 
$\mathbf{u}(f,a_f)$ for $f\in \Lambda\setminus B$
and $U(h,\mathbf{v})$ be the matrix $U$
with only the $h$-th column replaced by $\mathbf{v}$.
Then we have the following two lemmas.
\begin{lemma}
\label{lm:det_U}
\begin{equation}
  |\det U|=
  \frac
  {\sharp(\mathbb{Z}^r/\langle \vec{B}\rangle)}{\sharp(\mathbb{Z}^r/\langle \vec{B}_0\rangle)}.
\end{equation}
\end{lemma}
\begin{proof}
By rearranging rows and columns 
we see that
\begin{equation}
  \begin{split}
    |\det U|
    &=
    \Bigl\lvert\det
    (\mathbf{u}(f,a_f)_g)^{g\in L_0}_{f\in \Lambda\setminus B}
    \Bigr\rvert
    \\
    &=
    \Biggl\lvert\det
    \begin{pmatrix}
      (\mathbf{u}(f,a_f)_g)^{g\in B\setminus B_0}_{f\in B_0\setminus B}&
      (\mathbf{u}(f,a_f)_g)^{g\in B\setminus B_0}_{f\in L_0\setminus B}\\
      (\mathbf{u}(f,a_f)_g)^{g\in L_0\setminus B}_{f\in B_0\setminus B}&
      (\mathbf{u}(f,a_f)_g)^{g\in L_0\setminus B}_{f\in L_0\setminus B}
    \end{pmatrix}
    \Biggr\rvert
    \\
    &=
    \Biggl\lvert\det
    \begin{pmatrix}
      (\langle \vec{g},\vec{f}^{B_0}\rangle)^{g\in B\setminus B_0}_{f\in B_0\setminus B}&
      0\\
      *&
      E_{\sharp(L_0\setminus B)}
    \end{pmatrix}
    \Biggr\rvert
    \\
    &=
    \Bigl\lvert\det
    (\langle \vec{g},\vec{f}^{B_0}\rangle)^{g\in B\setminus B_0}_{f\in B_0\setminus B}
    \Bigr\rvert
    \\
    &=
    \Bigl\lvert\det
    (\langle \vec{g},\vec{f}^{B_0}\rangle)^{g\in B}_{f\in B_0}
    \Bigr\rvert
  \end{split}
\end{equation}
by \eqref{eq:detgf}. Further
\begin{equation}
  \begin{split}
    \Bigl\lvert\det
    (\langle \vec{g},\vec{f}^{B_0}\rangle)^{g\in B}_{f\in B_0}
    \Bigr\rvert
    =
    \Bigl\lvert
    \det(\vec{g})_{g\in B}
    \det(\vec{f}^{B_0})_{f\in B_0}
    \Bigr\rvert
    =
    \frac{\Bigl\lvert
    \det(\vec{g})_{g\in B}
    \Bigr\rvert}
  {\Bigl\lvert
    \det(\vec{f})_{f\in B_0}
    \Bigr\rvert}
    =
    \frac
    {\sharp(\mathbb{Z}^r/\langle \vec{B}\rangle)}{\sharp(\mathbb{Z}^r/\langle \vec{B}_0\rangle)}.
  \end{split}
\end{equation}
\end{proof}

\begin{lemma}
  \label{lm:det_UU}  
  For $f\in\Lambda\setminus B$, we have
  \begin{equation}
    \label{eq:UUVV}
    \frac{\det U(f,\mathbf{t}^*)}{\det U}
    =(-1)^{a_f}
    \Bigl(
    (t_f-2\pi\sqrt{-1}\const{f})
    -\sum_{g\in B}
    (t_g-2\pi\sqrt{-1}\const{g})
    \langle \vec{f},\vec{g}^{B}\rangle
    \Bigr).
  \end{equation}
  % For $h\in\Lambda\setminus B$, we have
  % \begin{equation}
  %   \label{eq:UUVV}
  %   \frac{\det U(h,\mathbf{t}^*)}{\det U}
  %   =(-1)^{a_h}
  %   \Bigl(
  %   (t_h-2\pi\sqrt{-1}\const{h})
  %   -\sum_{g\in B}
  %   (t_g-2\pi\sqrt{-1}\const{g})
  %   \langle \vec{h},\vec{g}^{B}\rangle
  %   \Bigr).
  % \end{equation}
\end{lemma}
\begin{proof}
We show that $\mathbf{x}=(x_f)_{f\in\Lambda\setminus B}$ defined by
\begin{equation}
\label{def_x_f}
  x_f=(-1)^{a_f}
    \Bigl(
    (t_f-2\pi\sqrt{-1}\const{f})
    -\sum_{g\in B}
    (t_g-2\pi\sqrt{-1}\const{g})
    \langle \vec{f},\vec{g}^{B}\rangle
    \Bigr)
\end{equation}
is a unique solution of the linear equation
\begin{equation}
  \label{eq:Uxt}
  U\mathbf{x}=\mathbf{t}^*.
\end{equation}
Then the statement follows from Cramer's rule.

Now we show that \eqref{def_x_f} satisfies \eqref{eq:Uxt}.   First observe
\begin{equation}
  \begin{split}
    (U\mathbf{x})_h
    &=
    \sum_{f\in B_0\setminus B}
    (-1)^{1-a_f}\langle\vec{h},\vec{f}^{B_0}\rangle
    (-1)^{a_f}
    \Bigl(
    (t_f-2\pi\sqrt{-1}\const{f})
    -\sum_{g\in B}
    (t_g-2\pi\sqrt{-1}\const{g})
    \langle \vec{f},\vec{g}^{B}\rangle
    \Bigr)
    \\
    &\qquad
    +
    \sum_{f\in L_0\setminus B}(-1)^{a_f}\delta_{hf}
    (-1)^{a_f}
    \Bigl(
    (t_f-2\pi\sqrt{-1}\const{f})
    -\sum_{g\in B}
    (t_g-2\pi\sqrt{-1}\const{g})
    \langle \vec{f},\vec{g}^{B}\rangle
    \Bigr)
    \\
    &=-
    \sum_{f\in B_0\setminus B}
    \langle\vec{h},\vec{f}^{B_0}\rangle
    \Bigl(
    (t_f-2\pi\sqrt{-1}\const{f})
    -\sum_{g\in B}
    (t_g-2\pi\sqrt{-1}\const{g})
    \langle \vec{f},\vec{g}^{B}\rangle
    \Bigr)
    \\
    &\qquad
    +
    \sum_{f\in L_0\setminus B}\delta_{hf}
    \Bigl(
    (t_h-2\pi\sqrt{-1}\const{h})
    -\sum_{g\in B}
    (t_g-2\pi\sqrt{-1}\const{g})
    \langle \vec{h},\vec{g}^{B}\rangle
    \Bigr).
  \end{split}
\end{equation}
Assume $h\in B\setminus B_0$. Then $\delta_{hf}=0$ for all $f\in L_0\setminus B$, and hence
\begin{equation}
\label{8-52}
  \begin{split}
    (U\mathbf{x})_h
    &=-
    \sum_{f\in B_0\setminus B}
    \langle\vec{h},\vec{f}^{B_0}\rangle
    \Bigl(
    (t_f-2\pi\sqrt{-1}\const{f})
    -\sum_{g\in B}
    (t_g-2\pi\sqrt{-1}\const{g})
    \langle \vec{f},\vec{g}^{B}\rangle
    \Bigr)
    \\
    &=-
    \sum_{f\in B_0\setminus B}
    \langle\vec{h},\vec{f}^{B_0}\rangle
    (t_f-2\pi\sqrt{-1}\const{f})
    +
    \sum_{f\in B_0\setminus B}
    \sum_{g\in B}
    (t_g-2\pi\sqrt{-1}\const{g})
    \langle\vec{h},\vec{f}^{B_0}\rangle
    \langle \vec{f},\vec{g}^{B}\rangle.
    % \\
    % &=-
    % \sum_{f\in B_0\setminus B}
    % \langle\vec{h},\vec{f}^{B_0}\rangle
    % (t_f-2\pi\sqrt{-1}\const{f})
    % +
    % \sum_{g\in B}
    % (t_g-2\pi\sqrt{-1}\const{g})
    % \langle\vec{h},    
    % \sum_{f\in B_0\setminus B}
    % \vec{f}^{B_0}\langle \vec{f},\vec{g}^{B}\rangle
    % \rangle.
  \end{split}
\end{equation}
From \eqref{eq:id} we have
$$
\mathbf{z}=\sum_{f\in B_0\setminus B}\vec{f}^{B_0}\langle\vec{f},\mathbf{z}\rangle
+\sum_{f\in B_0\cap B}\vec{f}^{B_0}\langle\vec{f},\mathbf{z}\rangle,
$$
and hence
\begin{equation}
\label{take_inner_product}
\langle\vec{h},\mathbf{z}\rangle=
\sum_{f\in B_0\setminus B}\langle\vec{h},\vec{f}^{B_0}\rangle\langle\vec{f},\mathbf{z}\rangle
+\sum_{f\in B_0\cap B}\langle\vec{h},\vec{f}^{B_0}\rangle\langle\vec{f},\mathbf{z}\rangle.
\end{equation}
Putting $\mathbf{z}=\sum_{g\in B}(t_g-2\pi\sqrt{-1}\const{g})\vec{g}^B$
in \eqref{take_inner_product}, we obtain
\begin{equation}
\label{yuruyuru}
 \begin{split}
 &\sum_{f\in B_0\setminus B}\sum_{g\in B}(t_g-2\pi\sqrt{-1}\const{g})
  \langle\vec{h},\vec{f}^{B_0}\rangle\langle\vec{f},\vec{g}^B\rangle
  \\
 &=\sum_{g\in B}(t_g-2\pi\sqrt{-1}\const{g})\delta_{hg}
  -\sum_{f\in B_0\cap B}\sum_{g\in B}(t_g-2\pi\sqrt{-1}\const{g})
  \langle\vec{h},\vec{f}^{B_0}\rangle\langle\vec{f},\vec{g}^B\rangle.
  \end{split}
\end{equation}
Substituting this into the right-hand side of \eqref{8-52}, we obtain
\begin{equation}
  \begin{split}
    (U\mathbf{x})_h=&
    -\sum_{f\in B_0\setminus B}
    \langle\vec{h},\vec{f}^{B_0}\rangle
    (t_f-2\pi\sqrt{-1}\const{f})+    (t_h-2\pi\sqrt{-1}\const{h})
    % \\
    % &\qquad
    -
    \sum_{f\in B_0\cap B}
    \sum_{g\in B}
    (t_g-2\pi\sqrt{-1}\const{g})
    \langle\vec{h},\vec{f}^{B_0}\rangle
    \langle \vec{f},\vec{g}^{B}\rangle
    \\
    &=
  (t_h-2\pi\sqrt{-1}\const{h})
  -\sum_{f\in B_0\setminus B}
  \langle\vec{h},\vec{f}^{B_0}\rangle
  (t_f-2\pi\sqrt{-1}\const{f})
  -
  \sum_{f\in B_0\cap B}
  (t_f-2\pi\sqrt{-1}\const{f})
  \langle\vec{h},\vec{f}^{B_0}\rangle
  =t_h^*.
\end{split}
\end{equation}
Assume $h\in L_0\setminus B$. Then again using \eqref{yuruyuru}, we have
\begin{equation}
  \begin{split}
    (U\mathbf{x})_h
    &=-
    \sum_{f\in B_0\setminus B}
    \langle\vec{h},\vec{f}^{B_0}\rangle
    (t_f-2\pi\sqrt{-1}\const{f})
    +
    \sum_{f\in B_0\setminus B}\sum_{g\in B}
    (t_g-2\pi\sqrt{-1}\const{g})
    \langle\vec{h},\vec{f}^{B_0}\rangle
    \langle \vec{f},\vec{g}^{B}\rangle
    \\
    &\qquad
    +
    (t_h-2\pi\sqrt{-1}\const{h})
    -\sum_{g\in B}
    (t_g-2\pi\sqrt{-1}\const{g})
    \langle \vec{h},\vec{g}^{B}\rangle
    \\
    &=-
    \sum_{f\in B_0\setminus B}
    \langle\vec{h},\vec{f}^{B_0}\rangle
    (t_f-2\pi\sqrt{-1}\const{f})
    -
    \sum_{f\in B_0\cap B}\sum_{g\in B}
    (t_g-2\pi\sqrt{-1}\const{g})
    \langle\vec{h},\vec{f}^{B_0}\rangle
    \langle \vec{f},\vec{g}^{B}\rangle
    \\
    &\qquad
    +
    (t_h-2\pi\sqrt{-1}\const{h})
    \\
    &=t_h^*.
  \end{split}
\end{equation}
\end{proof}

%%%%%%%%%%%%%%%%%%%%%%%%%%%%%%%%%%%%%%%%%%%%%%%%%%%%%%%%%%%%%%%%%%%%%%%%%%%%%%%%%%%%%%%%%%%%%%%%%
\section{Completion of the proof of Theorem \ref{thm:main1} and Theorem \ref{thm:main1b}}\label{sec-9}
%%%%%%%%%%%%%%%%%%%%%%%%%%%%%%%%%%%%%%%%%%%%%%%%%%%%%%%%%%%%%%%%%%%%%%%%%%%%%%%%%%%%%%%%%%%%%%%%%

In this section we prove \eqref{tilde_F=F} to complete the proof of
Theorems \ref{thm:main1} and \ref{thm:main1b}.    First we show an
elementary lemma.

Let $n\in\mathbb{N}$ and let $\{\mathbf{e}_1,\ldots,\mathbf{e}_n\}$ be a standard basis of $\mathbb{R}^n$.
Let $M=(m_{ij})_{1\leq j\leq n}^{1\leq i\leq n}$
be an $n\times n$ matrix.
For $1\leq j\leq n$, 
and $\mathbf{v}\in\mathbb{R}^n$,
let $M(j,\mathbf{v})$ be the matrix $M$ with only the $j$-th column replaced by $\mathbf{v}$.
Let
\begin{equation}
\label{eq:uj}
\mathbf{u}(j)=
\det M(j,(\mathbf{e}_i)_{1\leq i\leq n})=
\det 
\begin{pmatrix}
  m_{11}&\ldots&m_{1\,j-1}&\mathbf{e}_1&m_{1\,j+1}&\ldots&m_{1n}\\
  \vdots&\ddots&\vdots&\vdots&\vdots&\ddots&\vdots\\
  m_{n1}&\ldots&m_{n\,j-1}&\mathbf{e}_n&m_{n\,j+1}&\ldots&m_{nn}
\end{pmatrix}
=\sum_{i=1}^n b_{ij}\mathbf{e}_i\in\mathbb{R}^n,
\end{equation}
where $b_{ij}$ is the cofactor of $(i,j)$-entry of $M$.
The following lemma (properties of cofactor matrices) is shown by elementary linear algebra, 
and we omit the proof.
\begin{lemma}
\label{lm:elementary}
  \begin{enumerate}
  \item 
    For $\mathbf{v}\in\mathbb{R}^n$, we have
\begin{equation}
\mathbf{u}(j)\cdot\mathbf{v}=\det M(j,\mathbf{v}).
\end{equation}
\item Let $U=(\mathbf{u}(j))_{1\leq j\leq n}$ be the $n\times n$ matrix
whose $j$-th column consists of $\mathbf{u}(j)$.
Then
\begin{equation}
  \det U=(\det M)^{n-1}.
\end{equation}
\end{enumerate}
\end{lemma}

%\begin{proof}[Proof of Theorem \ref{thm:main1}]
Now we start the proof of \eqref{tilde_F=F}, that is
$F(\mathbf{t},\mathbf{y};\Lambda)=\widetilde{F}(\mathbf{t},\mathbf{y};\Lambda)$
for $\mathbf{y}\in V\setminus\mathfrak{H}_{\mathscr{R}}$.

Since all the polytopes $\mathcal{P}(\mathbf{m};\mathbf{y})$ for 
$\mathbf{y}\in V\setminus\mathfrak{H}_{\mathscr{R}}$ are empty or simple
by Lemma \ref{lm:simplicity}, we may apply
Lemma \ref{lm:simple_exp} to the right-hand side of Proposition \ref{prop:gen_func}.
In the present case the vertices are of the form $\mathbf{p}(\mathbf{m};\mathbf{y};W)$,
satisfying \eqref{0leqleq1} by Lemma \ref{lm:char_vert}.
Therefore
we have
\begin{equation}
\label{9-4}
  \begin{split}
    \widetilde{F}(\mathbf{t},\mathbf{y};\Lambda)
    &=
    \Bigl(
    \prod_{f\in \Lambda}\frac{t_f}{\exp(t_f-2\pi\sqrt{-1}\const{f})-1}
    \Bigr)
    \frac{1}{\sharp(\mathbb{Z}^r/\langle \vec{B}_0\rangle)}
%    \sum_{\mathbf{w}\in \mathbb{Z}^r/\langle \vec{B}_0\rangle}
    \sum_{\mathbf{m}\in\mathbb{Z}^r}
    \sum_{W}
    \\
    &\qquad\times
    \exp\Bigl(
    \sum_{f\in B_0}
    (t_f-2\pi\sqrt{-1}\const{f}) 
    \langle\mathbf{y}+\mathbf{m},\vec{f}^{B_0}\rangle
%    (\{\langle\mathbf{y},\vec{f}^{B_0}\rangle\}+m_f)
    +
    \mathbf{t}^*\cdot\mathbf{p}(\mathbf{m};\mathbf{y};W)
    \Bigr)
    \\&\qquad\qquad
    \times
    \frac{|\det
        \bigl(\mathbf{p}(\mathbf{m};\mathbf{y};W)-\mathbf{p}(\mathbf{m};\mathbf{y};W')
        \bigr)_{W'\in E(\mathbf{m};W)}|}
    {\prod_{W'\in E(\mathbf{m};W)}\mathbf{t}^*\cdot 
      \bigl(\mathbf{p}(\mathbf{m};\mathbf{y};W)-\mathbf{p}(\mathbf{m};\mathbf{y};W')
      \bigr)},
  \end{split}
\end{equation}
where 
% $\mathbf{y}=\mathbf{y}+\mathbf{w}$ with $\mathbf{w}\in \mathbb{Z}^r/\langle \vec{B}_0\rangle$ understood as a unique representative in $\mathbb{Z}^r$ such that $\mathbf{y}+\mathbf{w}\in\mathfrak{D}\setminus\mathfrak{H}_{\mathscr{R}}$, and 
$E(\mathbf{m};W)$ is the set of all indices $W'$ such that
$\Conv(\{\mathbf{p}(\mathbf{m};\mathbf{y};W),\mathbf{p}(\mathbf{m};\mathbf{y};W')\})$ is an edge of $\mathcal{P}(\mathbf{m};\mathbf{y})$,
and
for each 
$\mathbf{m}=(m_1,\ldots,m_r)\in\mathbb{Z}^r$,
$W=(B,A)\in\mathscr{W}$
 runs over those satisfying
\begin{equation}
\label{eq:m_ineq}
  0\leq
  \Bigl\langle 
  \mathbf{y}+\mathbf{m}
%\sum_{g\in B_0}m_g\vec{g}
  -\sum_{g\in \Lambda\setminus B}a_g\vec{g},\vec{f}^{B}
  \Bigr\rangle
  \leq1.
\end{equation}
%by Lemma \ref{lm:char_vert}.
%for all $f\in B$ by Lemma \ref{lm:exponential}.

Recall that a vertex $\mathbf{p}(\mathbf{m};\mathbf{y};W)$ 
satisfies $(\sharp\Lambda-r)$ equations of the form
\begin{equation}
\mathbf{u}(f,a_f)\cdot
\mathbf{p}(\mathbf{m};\mathbf{y};W)=v(f,a_f;\mathbf{m};\mathbf{y})
\end{equation}
for $f\in \Lambda\setminus B$ with $W=(B,A)$ (see \eqref{8-14}).
For $W'=(B',A')\in E(\mathbf{m};W)$, we see that the two distinct vertices
$\mathbf{p}(\mathbf{m};\mathbf{y};W)$ 
and 
$\mathbf{p}(\mathbf{m};\mathbf{y};W')$ 
 share common $(\sharp\Lambda-r-1)$ hyperplanes, that is,
there exists $h\in \Lambda\setminus B$ 
such that $\Lambda\setminus(B\cup\{h\})\subset \Lambda\setminus B'$ and
$a_f=a'_f$ for $f\in \Lambda\setminus(B\cup\{h\})$, which implies
\begin{equation}
\label{eq:common_eqs}
  \begin{split}
    \mathbf{u}(f,a_f)\cdot
    \mathbf{p}(\mathbf{m};\mathbf{y};W)
    &=v(f,a_f;\mathbf{m};\mathbf{y}),
    \\
    \mathbf{u}(f,a_f)\cdot
    \mathbf{p}(\mathbf{m};\mathbf{y};W')
    &=v(f,a_f;\mathbf{m};\mathbf{y})
  \end{split}
\end{equation}
for $f\in \Lambda\setminus(B\cup\{h\})$
and
\begin{equation}
  \begin{split}
    \mathbf{u}(h,a_h)\cdot
    \mathbf{p}(\mathbf{m};\mathbf{y};W)
    &=v(h,a_h;\mathbf{m};\mathbf{y}),
    \\
    \mathbf{u}(h,a_h)\cdot
    \mathbf{p}(\mathbf{m};\mathbf{y};W')
    &> v(h,a_h;\mathbf{m};\mathbf{y}).
  \end{split}
\end{equation}
This $h$ is unique because otherwise we have
$\mathbf{p}(\mathbf{m};\mathbf{y};W)=\mathbf{p}(\mathbf{m};\mathbf{y};W')$.
Since $\sharp E(\mathbf{m};W)=\sharp(\Lambda\setminus B)=\sharp\Lambda-r$
(because $\mathcal{P}(\mathbf{m};\mathbf{y})$ is simple), we find that there is 
a one-to-one corresponding between $E(\mathbf{m};W)$ and $\Lambda\setminus B$.
By \eqref{eq:common_eqs}, we see that
the set of the equations and the inequality with respect to $\mathbf{v}$
\begin{equation}
\label{eq:const_v}
  \begin{cases}
  \mathbf{u}(f,a_f)\cdot\mathbf{v}=0\qquad
  &(f\in \Lambda\setminus(B\cup\{h\})),\\
  \mathbf{u}(h,a_h)\cdot\mathbf{v}<0&
\end{cases}
\end{equation}
has a solution
$\mathbf{v}=\mathbf{p}(\mathbf{m};\mathbf{y};W)
-\mathbf{p}(\mathbf{m};\mathbf{y};W')$.

We construct another vector $\mathbf{e}(\mathbf{m};W,W')$
satisfying \eqref{eq:const_v}
 so that
\begin{equation}
\mathbf{p}(\mathbf{m};\mathbf{y};W)
-\mathbf{p}(\mathbf{m};\mathbf{y};W')
=c(\mathbf{m};\mathbf{y};W,W')\mathbf{e}(\mathbf{m};W,W'),
\end{equation}
where $c(\mathbf{m};\mathbf{y};W,W')>0$.
Then it follows that
\begin{multline}
  \frac{|\det(\mathbf{p}(\mathbf{m};\mathbf{y};W)-\mathbf{p}(\mathbf{m};\mathbf{y};W'))_{W'\in E(\mathbf{m};W)}|}
  {\prod_{W'\in E(\mathbf{m};W)}\mathbf{t}^*\cdot\bigl(\mathbf{p}(\mathbf{m};\mathbf{y};W)-\mathbf{p}(\mathbf{m};\mathbf{y};W')\bigr)}
  \\
  \begin{aligned}
    &=
    \frac{\bigl\lvert\prod_{W'\in E(\mathbf{m};W)}c(W,W';\mathbf{m};\mathbf{y})\bigr\rvert
      |\det(\mathbf{e}(\mathbf{m};W,W'))_{W'\in E(\mathbf{m};W)}|}
    {\prod_{W'\in E(\mathbf{m};W)}c(\mathbf{m};\mathbf{y};W,W')\mathbf{t}^*\cdot \mathbf{e}(\mathbf{m};W,W')}
    \\
    &=
    \frac{|\det(\mathbf{e}(\mathbf{m};W,W'))_{W'\in E(\mathbf{m};W)}|}
    {\prod_{W'\in E(\mathbf{m};W)}\mathbf{t}^*\cdot \mathbf{e}(\mathbf{m};W,W')}.
  \end{aligned}
\end{multline}

The construction of $\mathbf{e}(\mathbf{m};W,W')$ is as follows.
Let $\mathbf{e}_g$ for $g\in L_0$ be the standard 
orthonormal basis of $\mathbb{R}^{\sharp L_0}$.
Let $U$ be the $(\sharp\Lambda-r)\times(\sharp\Lambda-r)$ matrix
whose $f$-th column consists of 
$\mathbf{u}(f,a_f)$ for $f\in \Lambda\setminus B$ with $W=(B,A)$.
For $h\in \Lambda\setminus B$, let
 $U(h,\mathbf{v})$ be the matrix $U$
with only the $h$-th column replaced by $\mathbf{v}$.
Note that $\det U\neq 0$ by Lemma \ref{lm:det_U}.
%\ref{lm:cond_inter_uniq} (cf. \eqref{eq:detgf}).
%By \eqref{eq:uj},
Define
\begin{equation}
\label{8-12}
\mathbf{e}(\mathbf{m};W,W')=
-(\sgn\det U)\det U(h,(\mathbf{e}_g)_{g\in L_0})
=
-(\sgn\det U)
\sum_{g\in L_0}b_{gh} \mathbf{e}_g
\end{equation}
(the second equality is due to \eqref{eq:uj}),
where $W'=(B',A')$
such that $\Lambda\setminus(B\cup\{h\})\subset \Lambda\setminus B'$
and
$b_{gh}$ is the cofactor of $(g,h)$-entry of $U$.
Then by Lemma \ref{lm:elementary}(1),
we have
\begin{equation}
%\label{9-13}
  \begin{split}
    \mathbf{e}(\mathbf{m};W,W')\cdot\mathbf{u}(f,a_{f})&=
    -(\sgn\det U) 
    \det U\bigl(h,\mathbf{u}(f,a_f)\bigr)=0,
    \\
    \mathbf{e}(\mathbf{m};W,W')\cdot\mathbf{u}(h,a_{h})&=
    -(\sgn\det U) 
    \det U\bigl(h,\mathbf{u}(h,a_{h})\bigr)
    =-(\sgn\det U)\det U
    =-|\det U|<0
  \end{split}
\end{equation}
for $f\in \Lambda\setminus(B\cup\{h\})$ as required.

We observed that $h$ runs over $\Lambda\setminus B$ 
when $W'$ runs over $E(\mathbf{m};W)$.
Therefore 
by Lemma \ref{lm:elementary}(2) we see that \eqref{8-12} implies
% $$
% (\mathbf{e}(\mathbf{m};W,W'))_{W'\in E(\mathbf{m};W)}
% $$
% and so we see
$|\det(\mathbf{e}(\mathbf{m};W,W'))_{W'\in E(\mathbf{m};W)}|
=|\det U|^{\sharp\Lambda-r-1}$.

% Therefore \eqref{9-13} implies
% $$
% (\mathbf{e}(\mathbf{m};W,W'))_{W'\in E(\mathbf{m};W)}\cdot U=-|\det U|\cdot E_{\sharp\Lambda-r},
% $$
% and so we see
% $|\det(\mathbf{e}(\mathbf{m};W,W'))_{W'\in E(\mathbf{m};W)}|$
% coincides with $|\det\widetilde{U}|=|\det U|^{\sharp\Lambda-r-1}$ where $\widetilde{U}$ is the cofactor matrix of $U$ (by Lemma \ref{lm:elementary}(2)).
Also, from \eqref{8-12} we have
$$
\mathbf{t}^*\cdot \mathbf{e}(\mathbf{m};W,W')=-(\sgn\det U)\sum_{g\in L_0}b_{gh}t_g^*
=-(\sgn\det U)\det U(h,\mathbf{t}^*).
$$
Therefore
\begin{equation}
  \begin{split}
  \frac{|\det(\mathbf{e}(\mathbf{m};W,W'))_{W'\in E(\mathbf{m};W)}|}
  {\prod_{W'\in E(\mathbf{m};W)}\mathbf{t}^*\cdot \mathbf{e}(\mathbf{m};W,W')}
  &=
  (-1)^{\sharp\Lambda-r}(\sgn\det U)^{\sharp\Lambda-r} \frac{|\det U|^{\sharp\Lambda-r-1}}
  {\prod_{h\in \Lambda\setminus B}
    \det U(h,\mathbf{t}^*)}
  \\
  &=
  (-1)^{\sharp\Lambda-r}\frac{1}{|\det U|} \frac{(\det U)^{\sharp\Lambda-r}}
  {\prod_{h\in \Lambda\setminus B}
    \det U(h,\mathbf{t}^*)}.
\end{split}
\end{equation}
Substituting this into the right-hand side of \eqref{9-4} and using
Lemmas \ref{lm:exponential}, \ref{lm:det_U} and \ref{lm:det_UU},
we have
\begin{equation}
  \begin{split}
    \widetilde{F}(\mathbf{t},\mathbf{y};\Lambda)
    &=
    (-1)^{\sharp \Lambda-r}
    \Bigl(
    \prod_{f\in \Lambda}\frac{t_f}{\exp(t_f-2\pi\sqrt{-1}\const{f})-1}
    \Bigr)
%    \sum_{\mathbf{w}\in \mathbb{Z}^r/\langle \vec{B}_0\rangle}
    \sum_{\mathbf{m}\in\mathbb{Z}^r}
    \sum_{W}
    \frac{1}
    {\sharp(\mathbb{Z}^r/\langle \vec{B}\rangle)}
%    \prod_{h\in \Lambda\setminus B}
%    (-1)^{a_h}
    \\
    &
    \qquad
    \times
    \exp
    \Bigl(
    \sum_{g\in \Lambda\setminus B}
    (t_g-2\pi\sqrt{-1}\const{g}) a_g
    +\sum_{f\in B}
    (t_f-2\pi\sqrt{-1}\const{f})
    \Bigl\langle 
    \mathbf{y}+\mathbf{m}
%\mathbf{w}+\sum_{g\in B_0}m_g\vec{g}
    -\sum_{g\in \Lambda\setminus B}a_g\vec{g},\vec{f}^{B}
    \Bigr\rangle
    \Bigr)
    \\
    &
    \qquad
    \times
    \prod_{h\in \Lambda\setminus B}
    \frac{(-1)^{a_h}}
    {
      (t_h-2\pi\sqrt{-1}\const{h})
      -\sum_{g\in B}
      (t_g-2\pi\sqrt{-1}\const{g})
      \langle \vec{h},\vec{g}^{B}\rangle
    }.
  \end{split}
\end{equation}
%where $\mathbf{w}\in \mathbb{Z}^r/\langle \vec{B}_0\rangle$ is understood as a unique representative in $\mathbb{Z}^r$ such that $\mathbf{y}+\mathbf{w}\in\mathfrak{D}\setminus\mathfrak{H}_{\mathscr{R}}$.

We rewrite the double sum on the first line of the above so as to
exchange the order of the sums with respect to $W\in\mathscr{W}$ and $\mathbf{m}\in\mathbb{Z}^r$. 
For each $W\in\mathscr{W}$, we see that
\begin{equation}
  \mathbf{m}
  -
  \sum_{g\in \Lambda\setminus B}a_g\vec{g}
\end{equation}
runs over $\mathbb{Z}^r$
when
$\mathbf{m}$ runs over $\mathbb{Z}^r$.
% $\mathbf{w}$ runs over $\mathbb{Z}^r/\langle \vec{B}_0\rangle$ and the index
% $\mathbf{m}=(m_1,\ldots,m_r)$ over $\mathbb{Z}^r$ because
% $\sum_{g\in \Lambda\setminus B}a_g\vec{g}\in\mathbb{Z}^r$.
Thus 
\eqref{eq:m_ineq} can be rewritten in terms of $\mathbf{v}\in\mathbb{Z}^r$, that is, $\mathbf{v}$ runs over those
satisfying
\begin{equation}
\label{eq:q_ineq}
  0\leq\langle 
  \mathbf{y}+\mathbf{v},
  \vec{f}^B\rangle\leq1
\end{equation}
for all $f\in B$.
If there exist $f\in B$ and $\mathbf{v}\in\mathbb{Z}^r$ such that
$c=\langle\mathbf{y}+\mathbf{v},\vec{f}^{B}\rangle\in\mathbb{Z}$,
then we can write 
$\mathbf{y}+\mathbf{v}=\sum_{g\in R}c_g\vec{g}+c\vec{f}$, where 
$R=B\setminus\{f\}\in\mathscr{R}$ and $c_g\in \mathbb{R}$.
Therefore we have
$\mathbf{y}+\mathbf{v}-c\vec{f}\in\mathfrak{H}_{R}$, hence
$\mathbf{y}\in\mathfrak{H}_{\mathscr{R}}$,
which contradicts with the assumption.
Thus condition \eqref{eq:q_ineq} can be replaced by
\begin{equation}
  0<\langle\mathbf{y}+\mathbf{v},\vec{f}^B\rangle<1.
\end{equation}
Let
$G=\{\mathbf{v}\in\mathbb{Z}^r~|~0<\langle\mathbf{y}+\mathbf{v},\vec{f}^B\rangle<1\text{ for all }f\in B\}$.
We show that the natural projection $g:G\to \mathbb{Z}^r/\langle\vec{B}\rangle$ is bijective.
If $g(\mathbf{v})=g(\mathbf{v}')$ for $\mathbf{v},\mathbf{v}'\in G$, then $\mathbf{v}=\mathbf{v}'+\mathbf{x}$ for some $\mathbf{x}\in\langle\vec{B}\rangle$.
Since $\langle \mathbf{x},\vec{f}^B\rangle=\langle\mathbf{y}+\mathbf{v},\vec{f}^B\rangle-
\langle\mathbf{y}+\mathbf{v}',\vec{f}^B\rangle$,
we have
$-1<\langle\mathbf{x},\vec{f}^B\rangle<1$ for all $f\in B$, which implies $\mathbf{x}=0$.
Conversely,
for any $\mathbf{v}\in\mathbb{Z}^r$, 
putting $\mathbf{x}=\sum_{g\in B}c_g \vec{g}
\in\langle\vec{B}\rangle$ 
with $c_g=-[\langle\mathbf{y}+\mathbf{v},\vec{g}^B\rangle]\in\mathbb{Z}$, 
we have
$$
\langle\mathbf{y}+\mathbf{v}+\mathbf{x},\vec{f}^B\rangle=\langle\mathbf{y}+\mathbf{v},\vec{f}^B\rangle
-c_f=\{\langle\mathbf{y}+\mathbf{v},\vec{f}^B\rangle\}
$$
and so 
$0<\langle\mathbf{y}+\mathbf{v}+\mathbf{x},\vec{f}^B\rangle<1$
because $\mathbf{y}\notin\mathfrak{H}_{\mathscr{R}}$.
This implies the assertion.
Hence replacing
$\langle\mathbf{y}+\mathbf{v},\vec{f}^B\rangle$ by $\{\langle\mathbf{y}+\mathbf{v},\vec{f}^B\rangle\}$,
we see that $\mathbf{v}$ runs over all representatives of $\mathbb{Z}^r/\langle\vec{B}\rangle$.

Therefore 
by exchanging the order of the sums with respect to $W=(B,A)\in\mathscr{W}=\mathscr{B}\times\mathscr{A}$ and $\mathbf{m}\in\mathbb{Z}^r$, and summing with respect to $\mathbf{v}$, we have
\begin{equation}
  \label{eq:exp_F_H}
  \begin{split}
    \widetilde{F}(\mathbf{t},\mathbf{y};\Lambda)
    &=
    (-1)^{\sharp \Lambda-r}
    \Bigl(
    \prod_{f\in \Lambda}\frac{t_f}{\exp(t_f-2\pi\sqrt{-1}\const{f})-1}
    \Bigr)
    \sum_{W\in\mathscr{W}}
    \frac{1}
    {\sharp(\mathbb{Z}^r/\langle \vec{B}\rangle)}
%    \prod_{h\in \Lambda\setminus B}
%    (-1)^{a_h}
    \\
    &
    \qquad
    \times
    \sum_{\mathbf{v}\in \mathbb{Z}^r/\langle \vec{B}\rangle}
    \exp
    \Bigl(
    \sum_{g\in \Lambda\setminus B}
    (t_g-2\pi\sqrt{-1}\const{g}) a_g
    +\sum_{f\in B}
    (t_f-2\pi\sqrt{-1}\const{f})
    \{\langle 
    \mathbf{y}+\mathbf{v},\vec{f}^{B}
    \rangle\}
    \Bigr)
    \\
    &
    \qquad
    \times
    \prod_{h\in \Lambda\setminus B}
    \frac{(-1)^{a_h}}
    {
      (t_h-2\pi\sqrt{-1}\const{h})
      -\sum_{g\in B}
      (t_g-2\pi\sqrt{-1}\const{g})
      \langle \vec{h},\vec{g}^{B}\rangle
    }
    \\
    &=
    (-1)^{\sharp \Lambda-r}
    \Bigl(
    \prod_{f\in \Lambda}\frac{t_f}{\exp(t_f-2\pi\sqrt{-1}\const{f})-1}
    \Bigr)
    \sum_{B\in\mathscr{B}}
    \sum_{A\in\mathscr{A}}
    \Bigl(\prod_{h\in \Lambda\setminus B}(-1)^{a_h}
    \exp((t_h-2\pi\sqrt{-1}\const{h}) a_h)\Bigr)
    \\
    &
    \qquad
    \times
    \frac{1}
    {\sharp(\mathbb{Z}^r/\langle \vec{B}\rangle)}
    \sum_{\mathbf{v}\in \mathbb{Z}^r/\langle \vec{B}\rangle}
    \exp
    \Bigl(
    \sum_{f\in B}
    (t_f-2\pi\sqrt{-1}\const{f})
    \{\langle 
    \mathbf{y}+\mathbf{v},\vec{f}^{B}
    \rangle\}
    \Bigr)
    \\
    &
    \qquad
    \times
    \prod_{h\in \Lambda\setminus B}
    \frac{1}
    {
      (t_h-2\pi\sqrt{-1}\const{h})
      -\sum_{g\in B}
      (t_g-2\pi\sqrt{-1}\const{g})
      \langle \vec{h},\vec{g}^{B}\rangle
    }.
  \end{split}
\end{equation}
Since
\begin{equation}
(-1)^{a_h}
    \exp((t_h-2\pi\sqrt{-1}\const{h}) a_h)=
    \begin{cases}
      -\exp(t_h-2\pi\sqrt{-1}\const{h}) & \text{if } a_h=1, \\
      1 & \text{if } a_h=0,
    \end{cases}
  \end{equation}
we have
\begin{equation}
\begin{split}
\sum_{A\in\mathscr{A}}
    \Bigl(\prod_{h\in \Lambda\setminus B}(-1)^{a_h}
    \exp((t_h-2\pi\sqrt{-1}\const{h}) a_h)\Bigr)&=
    \prod_{h\in \Lambda\setminus B}(1-\exp(t_h-2\pi\sqrt{-1}\const{h}))\\
    &=(-1)^{\sharp\Lambda-r} \prod_{h\in \Lambda\setminus B}(\exp(t_h-2\pi\sqrt{-1}\const{h})-1).
 \end{split}
\end{equation}
Therefore the rightmost side of
\eqref{eq:exp_F_H} is finally equal to
\begin{equation}
  \begin{split}
        &%=
    \Bigl(
    \prod_{f\in \Lambda}\frac{t_f}{\exp(t_f-2\pi\sqrt{-1}\const{f})-1}
    \Bigr)
    \sum_{B\in\mathscr{B}}
    \Bigl(
    \prod_{h\in \Lambda\setminus B}
    (\exp(t_h-2\pi\sqrt{-1}\const{h})-1)
    \Bigr)
    \\
    &
    \qquad
    \times
    \frac{1}
    {\sharp(\mathbb{Z}^r/\langle \vec{B}\rangle)}
    \sum_{\mathbf{v}\in \mathbb{Z}^r/\langle \vec{B}\rangle}
    \exp
    \Bigl(
    \sum_{f\in B}
    (t_f-2\pi\sqrt{-1}\const{f})
    \{\langle 
    \mathbf{y}+\mathbf{v},\vec{f}^{B}
    \rangle\}
    \Bigr)
    \\
    &
    \qquad
    \times
    \prod_{h\in \Lambda\setminus B}
    \frac{1}
    {
      (t_h-2\pi\sqrt{-1}\const{h})
      -\sum_{g\in B}
      (t_g-2\pi\sqrt{-1}\const{g})
      \langle \vec{h},\vec{g}^{B}\rangle
    },
  \end{split}
\end{equation}
and coincides with \eqref{eq:exp_F}
for $\mathbf{y}\in V\setminus\mathfrak{H}_{\mathscr{R}}$.
This completes the proof of \eqref{tilde_F=F}, and hence the proof of
Theorems \ref{thm:main1} and \ref{thm:main1b}.
\section{A Hierarchy and Differential Equations}\label{sec-10}
%%%%%%%%%%%%%%%%%%%%%%%%%%%%%%%%%%%%%%%%%%%%%%%%%%%%%%%%%%%%%%%%%%%%%%%%%%%%%%%%%%%%%%%%%%%%%%%

We conclude this paper with a theorem which asserts that the family of our generating functions has a hierarchy.
Let 
\begin{equation}
  \mathbf{\Lambda}_r=\{\Lambda\subset(\mathbb{Z}^{r}\setminus\{\vec{0}\})\times\mathbb{C}~|~ \sharp \Lambda<\infty,\rank\langle\vec{\Lambda}\rangle=r\}.
\end{equation}
For $\mathbf{v}=(v_1,\ldots,v_r)\in V$, let
\begin{equation}
  \partial_{\mathbf{v}}=v_1\partial_{y_1}+\cdots+v_r\partial_{y_r},
\end{equation}
where $\partial_{y_j}$ is the $j$-th partial differential operator acting on
$\mathbf{y}=(y_j)_{1\leq j\leq r}\in V$.
For $g=(\vec{g},\const{g})\in(\mathbb{Z}^{r}\setminus\{\vec{0}\})\times\mathbb{C}$, 
define 
\begin{equation}
  D_g=\frac{t_g-2\pi\sqrt{-1}\const{g}}{t_g}-\frac{1}{t_g}\partial_{\vec{g}}.
\end{equation}

\begin{theorem}
Let $\Lambda, \Lambda'\in\mathbf{\Lambda}_r$ with
$\Lambda'\subset\Lambda$, and $\mathbf{t}=(t_g)_{g\in \Lambda}$, $\mathbf{t}'=(t_g)_{g\in \Lambda'}$.
We have
\begin{equation}
\Bigl(\prod_{g\in\Lambda\setminus\Lambda'}
  D_g\Bigr)
F(\mathbf{t},\mathbf{y};\Lambda)=
F(\mathbf{t}',\mathbf{y};\Lambda'),
\end{equation}  
where on the left-hand side $D_g$ is understood to act at $\mathbf{y}\in V\setminus\mathfrak{H}_{\mathscr{R}(\Lambda)}$ and the resulting function, to be continuously extended by the one-sided limit along $\phi$.
\end{theorem}

\begin{proof}
It is sufficient to show the assertion in the case $\Lambda\setminus\Lambda'=\{g\}$ and $\mathbf{y}\in V\setminus\mathfrak{H}_{\mathscr{R}(\Lambda)}$.
  For $B\in\mathscr{B}(\Lambda)$ and
$\mathbf{w}\in \mathbb{Z}^r/\langle \vec{B}\rangle$,
we put
\begin{equation}
\label{eq:Fbw}
  F_{B,\mathbf{w}}(\mathbf{t},\mathbf{y};\Lambda)
:=
  \Bigl(
    \prod_{h\in \Lambda\setminus B}K(\mathbf{t},h)^{-1}
%    \frac{t_h}
%    {t_h-2\pi\sqrt{-1}\const{h}-\sum_{f\in B}(t_f-2\pi\sqrt{-1}\const{f})\langle \vec{h},\vec{f}^B\rangle}
    \Bigr)
    \Bigl(
    \prod_{f\in B}\frac{t_f\exp
      ((t_f-2\pi\sqrt{-1}\const{f})\{\mathbf{y}+\mathbf{w}\}_{B,f})}{\exp(t_f-2\pi\sqrt{-1}\const{f})-1}
    \Bigr),  
  \end{equation}
where for $h\in \Lambda$,
\begin{equation}
K(\mathbf{t},h)=\frac
    {t_h-2\pi\sqrt{-1}\const{h}-\sum_{f\in B}(t_f-2\pi\sqrt{-1}\const{f})\langle \vec{h},\vec{f}^B\rangle}{t_h},
\end{equation}
so that
\begin{equation}
  F(\mathbf{t},\mathbf{y};\Lambda)=
    \sum_{B\in\mathscr{B}(\Lambda)}
    \frac{1}{\sharp(\mathbb{Z}^r/\langle \vec{B}\rangle)}
    \sum_{\mathbf{w}\in \mathbb{Z}^r/\langle \vec{B}\rangle}
    F_{B,\mathbf{w}}(\mathbf{t},\mathbf{y};\Lambda).
\end{equation}

By simple computations we obtain
\begin{equation}
  D_gF_{B,\mathbf{w}}(\mathbf{t},\mathbf{y};\Lambda)  =
  K(\mathbf{t},g)
  F_{B,\mathbf{w}}(\mathbf{t},\mathbf{y};\Lambda).
\end{equation}
If $g\in \Lambda\setminus B$, then
the factor $K(\mathbf{t},g)$ cancels with the factor $K(\mathbf{t},g)^{-1}$ appearing on the
right-hand side of \eqref{eq:Fbw}, and so the variable $t_g$
disappears. Thus we have
\begin{equation}
  D_gF_{B,\mathbf{w}}(\mathbf{t},\mathbf{y};\Lambda)=
F_{B,\mathbf{w}}(\mathbf{t}',\mathbf{y};\Lambda\setminus\{g\}).
\end{equation}
In this case $B\in\mathscr{B}(\Lambda')$.

If $g\in B$, then 
\begin{equation}
  \sum_{f\in B}(t_f-2\pi\sqrt{-1}\const{f})\langle \vec{g},\vec{f}^B\rangle
=t_g-2\pi\sqrt{-1}\const{g}
\end{equation}
and hence $K(\mathbf{t},g)=0$ and
\begin{equation}
  D_gF_{B,\mathbf{w}}(\mathbf{t},\mathbf{y};\Lambda)=0.
\end{equation}
Thus the sum runs over all $\mathscr{B}(\Lambda')$ and
\begin{equation}
  D_gF(\mathbf{t},\mathbf{y};\Lambda)=
F(\mathbf{t}',\mathbf{y};\Lambda\setminus\{g\}).
\end{equation}
\end{proof}

%%%%%%%%%%%%%%%%%%%%%%%%%%%%%%%%%%%%%%%%%%%%%%%%%%%%%%%%%%%%%%%%%%%%%%%%%%%%%%%%%%%%%%%%%%%%%%%%%

\ 

\begin{flushleft}
\begin{small}
Y. Komori\\
Department of Mathematics \\
Rikkyo University \\
Nishi-Ikebukuro, Toshima-ku\\
Tokyo 171-8501, Japan\\
komori@rikkyo.ac.jp

\ 

K. Matsumoto\\
Graduate School of Mathematics \\
Nagoya University\\
Furo-cho, Chikusa-ku \\
Nagoya 464-8602, Japan\\
kohjimat@math.nagoya-u.ac.jp

\ 

H. Tsumura\\
Department of Mathematics and Information Sciences\\
Tokyo Metropolitan University\\
1-1, Minami-Ohsawa, Hachioji \\
Tokyo 192-0397, Japan\\
tsumura@tmu.ac.jp

\end{small}
\end{flushleft}

\end{document}